\documentclass[11pt]{amsart}

\usepackage{epigamath}

\usepackage[english]{babel}

\numberwithin{equation}{section}

\usepackage[shortlabels]{enumitem}

\usepackage{amsmath,amsthm,amssymb,hyperref}
\usepackage{physics}
\usepackage{pst-node}
\usepackage{tikz-cd}
\usepackage{capt-of}
\usepackage{caption}
\usepackage{booktabs}
\usepackage{multirow}
\usepackage{rotating}
\usepackage{float}
\usepackage{appendix}          
\usepackage[section]{placeins} 
\usepackage{cleveref}
\usetikzlibrary{cd}

\usepackage{adjustbox}

\usepackage{cellspace}

\usepackage{url}

\theoremstyle{definition}
\newtheorem*{definition*}{Definition}
\newtheorem{definition}{Definition}[section]

\theoremstyle{plain}
\newtheorem*{theorem*}{Theorem}
\newtheorem{theorem}[definition]{Theorem}
\newtheorem{lemma}[definition]{Lemma}
\newtheorem{corollary}[definition]{Corollary}

\newtheorem*{proposition*}{Proposition}

\theoremstyle{remark}

\newtheorem{remark}[definition]{Remark}
\newtheorem*{remark*}{Remark}
\newtheorem*{fact*}{Fact}

\theoremstyle{claim}
\newtheorem{claim}[definition]{Claim}
\newtheorem*{claim*}{Claim}

\newcommand{\Z}{\mathbb{Z}}

\newcommand{\Q}{\mathbb{Q}}
\newcommand{\C}{\mathbb{C}}
\newcommand{\F}{\mathbb{F}}
\newcommand{\PP}{\mathbb{P}}

\newcommand{\vphi}{\varphi}

\newcommand{\on}[1]{\operatorname{#1}}

\newcommand{\Mod}{\on{Mod}}

\newcommand{\inn}[1]{\langle #1 \rangle}
\DeclareMathOperator{\lcm}{lcm}
\DeclareMathOperator{\SL}{SL}
\DeclareMathOperator{\GL}{GL}
\DeclareMathOperator{\Char}{char}
\DeclareMathOperator{\Order}{order}
\DeclareMathOperator{\Orb}{Orb}
\newcommand{\Number}{\mathrm{number}}
\newcommand{\Abf}{\textbf{A}}

\newcommand{\relmiddle}[1]{\mathrel{}\middle#1\mathrel{}}

\setlist[enumerate,1]{label={\rm(\arabic*)}, ref={\rm\arabic*}}

\newlist{romanenum}{enumerate}{2}
\setlist[romanenum,1]{label={\rm(\roman*)}, ref={\rm\roman*}}

\newcommand{\supth}[1]{\ensuremath{#1^{\mathrm{th}}}}
\newcommand{\suprd}[1]{\ensuremath{#1^{\mathrm{rd}}}}

\EpigaVolumeYear{10}{2026} \EpigaArticleNr{10} \ReceivedOn{October 21, 2024}
\InFinalFormOn{September 1, 2025}
\AcceptedOn{December 25, 2025}

\title{Enumerating finite braid group orbits on $\boldsymbol{\SL_2(\C)}$-character varieties}
\titlemark{Enumerating finite braid group orbits on $\boldsymbol{\SL_2(\C)}$-character varieties}

\author{Amal Vayalinkal}
\address{Department of Mathematics, University of Toronto, Toronto, ON, M5S 2E4, Canada}
\email{amalrose.vayalinkal@mail.utoronto.ca, amalrosevayalinkal@gmail.com}

\authormark{A.~Vayalinkal}

\AbstractInEnglish{We analyze finite orbits of the natural braid group action on the character variety of the $n$-times punctured sphere. Building on recent results relating middle convolution and finite complex reflection groups, our work implements Katz's middle convolution to explicitly classify finite orbits in the $\SL_2(\C)$-character variety of the punctured sphere. We provide theoretical results on the existence of finite orbits arising from the imprimitive finite complex reflection groups and formulas for constructing such examples when they exist. In the primitive finite complex reflection groups, we perform an exhaustive search and provide computational results. Our contributions also include Magma computer code for middle convolution and for computing the orbit under this action when it is known to be finite.}

\MSCclass{14H30, 37C05, 37-04, 13-04}

\KeyWords{Algebraic geometry, dynamical systems, representation theory, middle convolution, character variety}

\begin{document}



\maketitle

\begin{prelims}

\DisplayAbstractInEnglish

\bigskip

\DisplayKeyWords

\medskip

\DisplayMSCclass

\end{prelims}


\newpage

\setcounter{tocdepth}{1}

\tableofcontents


	\section{Introduction}
	Dynamics of the natural braid group action on character varieties have been studied for over a century. The growing interest in classifying the finite orbits of this action has led to a wealth of new techniques and examples. Inspired by Boalch \cite{boalch2005klein}, who explored finite complex reflection groups as a source of finite orbits, Lam--Landesman--Litt \cite{LLL} showed, up to mild conditions, that all finite orbits come from finite complex reflection groups via Katz's middle convolution operation; see \cite{katz1996}. With the goal of completing their analysis, we search the finite complex reflection groups to enumerate all possibilities and study them. We classify all finite orbits arising via middle convolution from the finite complex reflection groups.
    
    We now set notation and briefly recall the connection between local systems and representations. An expert should feel free to skip to \Cref{exp sec}.
	
	Let $\Sigma_{T+1}$ denote the $(T+1)$-punctured sphere, or equivalently, let $\Sigma_{ T+1} = \C\PP^1 \setminus \{x_1, \dots, x_{T+1}\}$. The fundamental group of the projective line minus $T+1$ points based at $x$ has the following presentation:
    \begin{align*}
        \pi_1(\Sigma_{T+1}, x) = \left\langle \gamma_1, \dots, \gamma_{T+1} \relmiddle | \prod_{i = 1}^{T+1} \gamma_i = 1 \right\rangle, 
    \end{align*}
    where each $\gamma_i$ is the simple closed loop around the $\supth{i}$ puncture.

    A representation of $\pi_1(\Sigma_{T+1}, x)$ is a homomorphism
    $$
    \rho\colon \pi_1(\Sigma_{T+1}, x) \longrightarrow \GL_n(\C)
    $$
    which, in the presentation given above, is completely determined by the image of each $\gamma_i$.
    Thus, a rank $n$ representation of $\pi_1(\Sigma_{T+1}, x)$ can be identified with a tuple of matrices $(A_1, \dots, A_{T+1}) \in \GL_n(\C)^{T+1}$ satisfying \begin{align*}
		\prod_{i = 1}^{T+1} A_i = I, 
	\end{align*}
    where $A_i = \rho(\gamma_i)$ for each $i$. This condition tells us that $A_{T+1}$ is completely determined by the first $T$ matrices of the tuple. Hence, any representation of the free group on $T$ generators, denoted $F_T$, defined by a tuple of $T$ matrices $(A_1, \dots, A_{T})$, gives us a representation of $\pi_1(\Sigma_{T+1}, x)$ by defining $A_{T+1} = (\prod_{i = 1}^{T} A_i)^{-1}$ and \textit{vice versa}.
    
    For fundamental groups, changing the base point from $x$ to $y$ is equivalent to conjugating by a path between $x$ and $y$. Thus, if we instead consider tuples of matrices up to simultaneous conjugation by an element of $\GL_n(\C)$, we can view these as representations of $\pi_1(\Sigma_{T+1})$, without specifying a base point. On representations, this corresponds to the action of $\GL_n(\C)$ via $$g \cdot \rho = g \rho g^{-1}.$$
    Quotienting by this action gives us isomorphism classes of representations, and the resulting space is called the \textit{character variety} (see \Cref{charvar}). For tuples of matrices $(A_1, \dots, A_{T+1})$, we write $[A_1, \dots, A_{T+1}]$ for the corresponding equivalence class in the character variety. 

    A local system $\mathcal{L}$ with fiber isomorphic to a vector space $V$ on $\Sigma_{T+1}$ is a sheaf on $\Sigma_{T+1}$ locally isomorphic to the constant sheaf defined by $V$. At each point $x \in \Sigma_{T+1}$, we have an isomorphism of the stalk $\mathcal{L}_x \cong V$; the dimension of $V$ (as a vector space over $\C$) is called the \textit{rank} of the local system $\mathcal{L}$. For each loop $\gamma\colon [0, 1] \rightarrow \Sigma_{T+1}$ with $\gamma(0) = \gamma(1)$, we have isomorphisms $$ V \cong \mathcal{L}_{\gamma(0)} \cong  \mathcal{L}_{\gamma(1)} \cong V.$$ 
    The isomorphism $V \rightarrow V$ given by composing the above maps only depends on the homotopy class of $\gamma$ and corresponds to some invertible linear transformation $A \in \GL(V)$. We call $A$ the monodromy of $\mathcal{L}$ along $\gamma$.

    Given a local system $\mathcal{L}$ of rank $n$ on $\Sigma_{T+1}$, the monodromy along each generator $\gamma_i$ of $\pi_1(\Sigma_{T+1}, x)$ provides an element $A_i$ of $\GL_n(\C)$, and these $A_i$ are referred to as the local monodromy. The representation defined by $\rho(\gamma_i) = A_i$ is called the monodromy representation, and its image (as a subgroup of $\GL_n$) is called the monodromy group. Moreover, given a representation $\rho\colon \pi_1(\Sigma_{T+1},x) \rightarrow \GL_n(\C)$, a local system $\mathcal{L}$ with monodromy determined by $\rho$ can be constructed via the universal cover of $\Sigma_{T+1}$. We refer the reader to \cite{locsysexpo} for more exposition on local systems.

    The above correspondence between local systems and representations is well-behaved on isomorphism classes, so we get the following 3-way bijections: 
    \begin{align*}
        &\Bigl\{ \text{Isomorphism classes of rank $n$ local systems $\mathcal{L}$ on $\Sigma_{T+1}$}\Bigr\} \\
&\qquad\qquad\qquad\updownarrow \\ &\left\{ \text{Conjugacy classes of representations $\pi_1(\Sigma_{T+1}) \xrightarrow{\rho} \GL_n(\C)$} \right\} \\ &\qquad\qquad\qquad\updownarrow \\ &\left\{ \begin{array}{c}
\text{Simultaneous conjugacy classes $[A_1, \dots, A_{T+1}]$,} \\ \text{$A_i \in \GL_n(\C)$ with $\prod_{i=1}^{T+1} A_i = I$}  \end{array} \right\} = \Biggl\{\begin{array}{c} \text{Simultaneous conjugacy classes $[A_1, \dots, A_{T}]$},\\ \text{$A_i \in \GL_n(\C)$} \end{array}\Biggr\} \end{align*}

    \subsection{Mapping class group and middle convolution}\label{exp sec}
     On the character variety, we also have the following operations:
     \begin{enumerate}
     \item {\it Action of the mapping class group}: Associated to the surface $\Sigma_{T+1}$ is its mapping class group
       $$\Mod_{T+1} = \pi_0\left(\operatorname{Homeo}^{+}\left(\Sigma_{T+1}\right)\right),$$
       which has an action on representations of $\pi_1(\Sigma_{T+1})$ via the natural action on $\Sigma_{T+1}$. Note that since $\Mod_{T+1}$ typically does not preserve the base point of the fundamental group, we really do want to consider conjugacy classes of representations. 
         
         It is well known that $\Mod_{T+1}$ is isomorphic to the spherical braid group on $T+1$ strands, denoted $B_{T+1}$.  This group has $T$ generators $\sigma_1,  \dots, \sigma_{T}$ and acts on $(T+1)$-tuples in the following way: 
	$$\sigma_i\left(\left[A_1, A_2, \dots, A_T, A_{T+1}\right]\right) = \left[A_1, \dots, A_{i-1}, A_iA_{i+1}A_{i}^{-1}, A_{i}, A_{i+2}, \dots, A_{T+1}\right].$$	
    By the definition of this action, we see that $\sigma_i([A_1, A_2, \dots, A_T, A_{T+1}])$ will always be a tuple consisting of matrices from the group $\langle A_1, \dots, A_{T+1} \rangle$. Hence, whenever the $A_i$ all live in some finite group, the corresponding tuple will have finite orbit under this braid group action. 
    \item {\it Middle convolution}: Let $\on{LocSys}(\Sigma_{T+1})$ denote isomorphism classes of local systems on $\Sigma_{T+1}$ for any rank $n$. For each $\lambda \in \C^{\times} \setminus \{1 \}$, Katz \cite{katz1996} introduces an invertible functor called \textit{middle convolution}:
    \begin{align*}
       MC_{\lambda}\colon \on{LocSys}(\Sigma_{T+1}) \longrightarrow \on{LocSys}(\Sigma_{T+1}).
    \end{align*}
    This operation takes an isomorphism class of rank $n$ local systems on $\Sigma_{T+1}$ to an isomorphism class of rank $n'$ local systems on $\Sigma_{T+1}$. The rank $n'$ depends on $\lambda$ and on the conjugacy class of the local monodromy. For a more precise definition of middle convolution of local systems, see Lam--Landesman--Litt \cite[Definition 3.1.1]{LLL}.
    For tuples of matrices, the corresponding operation of middle convolution takes as input only the first $T$ matrices of the tuple and is discussed further in \Cref{mc}. This algebraic version of middle convolution, formulated for representations of free group on $T$ generators,  was introduced by Dettweiler--Reiter in \cite[Proposition~2.6]{KatzAlgo} and \cite{MC-L},  and is mainly what we use in this paper. Properties relevant to our work are that middle convolution is equivariant with respect to the action of the braid group, see \cite[Theorem~5.1 and Corollary~3.6]{KatzAlgo}, and preserves irreducibility, which we define below. 
     \end{enumerate}
     
	\begin{definition}
		 Representations, or equivalently tuples $[A_1, \cdots, A_{T+1}]$ with $\prod_{i = 1} ^{T+1} A_{i} = I$, in the character variety with finite orbit under the mapping class group action are called \emph{MCG-finite}. We say a tuple is \emph{irreducible} if the associated representation is irreducible.
	\end{definition}

	Recently, Lam--Landesman--Litt \cite{LLL} classified MCG-finite tuples of rank~2  with at least one matrix of infinite order.  They proved that such Zariski-dense representations come from pullbacks of curves or arise via middle convolution of local systems whose monodromy group is a \textit{finite complex reflection group}.
	
	The first kind, ``pullback'' representations, were classified by Diarra \cite[Sections 3--5]{diarra2013}. 
	
	Finite irreducible complex reflection groups, which can be viewed as generalizations of Coxeter groups, were classified by Shephard and Todd in 1954 as belonging to an infinite family $G(m, p, n)$ (imprimitive groups) for each rank $n$ or being one of 34 exceptional cases (primitive groups) denoted $G_3$ to $G_{37}$; see \cite{stfinite}. Each of these groups $W$ comes with a canonical finite-dimensional complex representation (given by reflection matrices acting on some vector space $V$), and the rank of $W$ is defined as the rank of this representation. The rank of $W$ will be the rank of the local system taken as input for middle convolution.
    
    In view of this classification, Lam--Landesman--Litt \cite[Corollary~5.1.4]{LLL} prove that for $T+1 \geq 7$, MCG-finite representations do not come from reflection groups of rank 5 or greater. Their paper leaves open for analysis the imprimitive groups of rank 3 and 4 and most of the primitive groups of rank 3 and 4, as well as the case when $T + 1 < 7$, which we investigate.

    One wants to classify remaining cases of MCG-finite rank 2 local systems on $\Sigma_{T+1}$, knowing that they come by middle convolution applied to certain rank $n$ local systems on $\Sigma_{T+1}$ with monodromy group a finite complex reflection group $W$. As a result, the issue is to classify those rank $n$ local systems on $\Sigma_{T+1}$ (or equivalently, $(T+1)$-tuples of $n \times n$ matrices multiplying to the identity that come from a rank $n$ complex reflection group $W$),
    which indeed give rise by middle convolution to a rank 2 local system on $\Sigma_{T+1}$ with finite MCG-orbit.

    Using that $(T+1)$-tuples which multiply to the identity are determined by the first $T$ matrices in the tuple, we restrict our attention to $T$-tuples in reflection groups with desirable properties. 

    	\begin{definition}\label{nice tup}
		A $T$-tuple $\Abf = [A_1, A_2, \dots, A_T]$ of matrices (defined up to simultaneous conjugation) from a finite complex reflection group $W$ is called \emph{nice} if $\Abf$ satisfies the following: 
		\begin{enumerate}[label = \arabic*., ref = {Condition~\arabic*}]
			\item All $A_i$ are reflections in $W$.
			\item \label{gen W cond} $W = \inn{A_1, A_2, \dots, A_T}$.
			\item \label{eigv cond}$(A_1A_2 \cdots A_T)^{-1}$ has an eigenvalue $\lambda \neq 1$ of multiplicity $T-2$. 
		\end{enumerate}
	\end{definition}

        Note that by~\ref{eigv cond}, nice tuples should never multiply to the identity, which is why we restrict to only the first $T$ matrices of the $(T+1)$-tuple corresponding to our rank $n$ local system on $\Sigma_{T+1}$. The reason for each of these conditions is explained in more detail in \Cref{methods}, but we briefly outline the idea here. 

	\begin{remark}\label{nice+ mcg-fin}
          Given a nice $T$-tuple $\mathbf{A} = [A_1, A_2, \dots, A_T]$, any $\lambda$ satisfying~\ref{eigv cond} ensures that $MC_{\lambda}(\Abf) = [\tilde{A}_1, \tilde{A}_2, \dots, \tilde{A}_T]$ is a $T$-tuple of rank 2 (see \Cref{mc}). Taking the $\supth{(T+1)}$
          matrix to be
          $$\tilde{A}_{T+1} = \left(\Pi_{i = 1}^T \tilde{A}_i\right)^{-1},$$
          we get a $(T+1)$-tuple $[\tilde{A}_1, \tilde{A}_2, \dots, \tilde{A}_T, \tilde{A}_{T+1}]$
          whose product is the identity, namely a rank 2 local system on $\Sigma_{T+1}$. As the MCG-orbit of $\mathbf{A}$ will consist of tuples with matrices from the group $\langle A_1, \dots, A_{T} \rangle$, which for nice tuples is the finite group $W$, we know that $\mathbf{A}$ is MCG-finite. Since middle convolution is equivariant with the braid group action (see Section~\ref{mc}), we know that $[\tilde{A}_1, \tilde{A}_2, \dots, \tilde{A}_T, \tilde{A}_{T+1}]$
          must be MCG-finite. Moreover, if $\{A_1, A_2, \dots, A_T\}$ generates an irreducible group, then, as middle convolution preserves irreducibility, we know our rank 2 $(T+1)$-tuple is irreducible as well. Thus, we have used a nice $T$-tuple of rank $n$, where $n = \rank (W)$, to produce an irreducible MCG-finite rank 2 local system on $\Sigma_{T+1}$. 
	\end{remark}

    We have written computer code for each step of the process described above. We start by searching finite complex reflection groups for nice tuples. Since middle convolution depends on a parameter $\lambda$, the same nice $T$-tuple can be used to produce nonisomorphic rank 2 local systems on $\Sigma_{T+1}$, as discussed in \Cref{same nice+ diff mcg}.
    Although Katz \cite{katz1996} originally introduced middle convolution in the language of perverse sheaves, one of our contributions is implementing computer code to explicitly compute the middle convolution given a tuple of matrices and parameter $\lambda \in \C^{\times}$, following the algebraic reformulation of Dettweiler--Reiter \cite[Section~2]{MC-L}. We also provide code for computing the braid group orbit for MCG-finite tuples in the  $\SL_2(\C)$-character variety of $\Sigma_{T+1}$, for $T = 3, 4$. All the computer code is publicly accessible online at \cite{Vay24}, and we hope it encourages further applications of this operation. \\

    We now provide a list of exclusions that this work will not treat:
    \begin{enumerate}[label = (\roman*), ref = Exclusion~(\roman*)]
        \item \label{I excl} We will not consider local systems on $\Sigma_{T+1}$ produced from local systems on $\Sigma_{T}$ by inserting the identity matrix into the $T$-tuple. This method ``artificially'' produces MCG-finite local systems on the sphere with more punctures, so we consider only tuples without any identity matrix.
        \item \label{scalar excl} We will not consider nice $T$-tuples $[A_1, \dots, A_T]$ from finite complex reflection groups such that their inverse product $A_{T+1} = (A_1 \cdots A_T)^{-1}$ is a scalar matrix $\lambda^{-1} I \in \GL_n$. 
        Writing $MC_{\lambda}(A_1, \dots, A_{T}) = [\tilde{A}_1, \dots, \tilde{A}_{T}]$, we have a local system of rank 2. Because of how eigenvalues transform under middle convolution (see \Cref{MC eigenvals}), we know that $\tilde{A}_{T+1} = (\prod_{i=1}^{T} \tilde{A_{i}})^{-1}$ will be the scalar $\lambda I \in \GL_2$. Multiplying by a character to put this tuple into $\SL_2$, we still have that the new $\supth{(T+1)}$ matrix $\tilde{A'}_{T+1}$ of the tuple is a scalar matrix. In $\SL_2$ the only scalar matrices are $I$ or $-I$. If $\tilde{A'}_{T+1} = I$, then we exclude it by~\eqref{I excl} above. If $\tilde{A'}_{T+1} = -I$, then for some nonscalar $\tilde{A'_j}$, $j \neq T+1$, we can construct an equivalent tuple by replacing $\tilde{A'_j}$ with $-\tilde{A'_j}$ and replacing $\tilde{A'}_{T+1}$ with $-\tilde{A'}_{T+1}$. But this means $\tilde{A'}_{T+1} = I$, and this tuple will be excluded. 
        \item \label{rk 2 excl} We do not search finite reflection groups of rank 2 for nice tuples. These can only produce local systems on $\Sigma_3$ or $\Sigma_4$, and we would be going from a rank 2 local system to another rank 2 local system, possibly with infinite monodromy group after middle convolution. For $\Sigma_3$, the action of the braid group factors through $S_3$ (see \cite[Proposition~2.3]{farbmarg}), so all orbits are MCG-finite. For $\Sigma_4$, rank 2 local systems over the 4-punctured sphere have been studied in connection to the Painlev\'e VI equation (see \Cref{subsec: prior work}), and finite orbits were classified by Lisovyy--Tykhyy \cite{T-L}. 
        \item \label{red excl} We do not consider nonirreducible tuples. These are classified by Cousin--Moussard; see \cite[Section~3]{cousinmouss}.
    \end{enumerate}

    With the above exclusions in mind, we discuss which cases remain to be considered. 

    One property of finite complex reflection groups (see \Cref{FCRG}) is that a group of rank $n$ is generated by either $n$ reflections or $n+1$ reflections. By the definition of nice tuples, the condition of generating the whole group (\ref{gen W cond}) tells us we need $n \leq T$, and since an $n \times n$ matrix cannot have an eigenvalue of multiplicity greater than $n$, the eigenvalue condition (\ref{eigv cond}) tells us $T-2 \leq n$. However, if $T-2 = n$, then $A_{T+1}$ will be a scalar matrix (as Jordan blocks with nonzero eigenvalue of size greater than $1$ do not have finite order), hence falls under~\ref{scalar excl}, so in fact we require $T-2 < n$, or equivalently $T < n+2$. Thus, for rank 2 local systems on $\Sigma_{T+1}$, we need only search for nice $T$-tuples in reflection groups of rank $n$ satisfying the relation 
    \begin{equation*} 
        n \leq T \leq n+1.
    \end{equation*}

    For $T+1 < 3$, $\pi_1(\Sigma_{T+1})$ is abelian. Hence, every representation is completely reducible, and there can be no irreducible local systems of rank 2 over $\Sigma_{T+1}$ (\ref{red excl}). When $T+1 = 3$,  we are searching the complex reflection groups for nice $2$-tuples,
    namely two reflections that generate the whole group. Any irreducible finite complex reflection group generated by two reflections will necessarily satisfy $\rank (W) = 2$ and is therefore excluded (\ref{rk 2 excl}). Hence, we may assume $T+1 \geq 4$.

    For an upper-bound on $T+1$, Lam--Landesman--Litt 
    \cite[Corollary~5.1.4]{LLL} prove that irreducible rank~2 MCG-finite local systems for $T+1 \geq 7$ cannot come from reflection groups of rank 5 or higher. The only solution to $T+1 \geq 7$ and $T-2 < 5$ is $T = 6$, but then $n$ $(= \rank (W))$ has to be greater than or equal to~5. So the only cases we need to consider are $3 < T+1 < 7$ and finite reflection groups $W$ of rank greater than~2 satisfying $\operatorname{rank} (W) = T-1$ or $\operatorname{rank} (W) = T$. 

    Thus, we search for nice tuples in finite complex reflection groups $W$ with $\operatorname{rank} (W) \in \{3, 4, 5\}$ with hopes of using middle convolution to produce rank 2 MCG-finite local systems on $\Sigma_{4}, \Sigma_{5}$, and $\Sigma_{6}$.
    
    We now note that the only rank 5 primitive group is $G_{33}$, also denoted as $W(K_5)$, which Lam--Landesman--Litt \cite[Lemma 5.4.1]{LLL} show cannot produce an irreducible rank 2 MCG-finite local system via middle convolution. 
	
    For imprimitive groups $W =G(m, p, n)$, the parameter $n$ denotes the rank of $W$.  The case where $n \geq 5$ is easiest to consider, as these groups do not contain any nice $n$-tuples. 

    \begin{theorem} \label{imprimthmlargen}
		There are no nice $n$-tuples in $G(m, p, n)$ for $n \geq 5$. 

    \end{theorem}

    The next two theorems pertain to the imprimitive groups of rank $3$ and $4$. In these groups we either prove no nice tuples exist or provide a system of equations for an infinite family of solutions. \Cref{imprim} provides the proofs of the following theorems, as well as conditions on the eigenvalues of the product of the matrices in the tuple, or equivalently, conditions on the parameter $\lambda$ that will be used for middle convolution whenever nice tuples are shown to exist.

    \begin{theorem}\label{rank 3 imprim}
	  In the rank $3$ imprimitive groups, there are no nice $3$-tuples in $G(m, p, 3)$ if $p \notin \{1, m\}$, and an infinite family of nice $3$-tuples exists in $G(m, 1, 3)$ and $G(m, m, 3)$ for all $m >1$. Nice $4$-tuples exist in $W = G(m, p, 3)$ for $m >1$ if and only if\, $W$ is one of the following groups: 
		\begin{enumerate}
			\item $G(m, m, 3)$ with $m \neq 3$,   
			\item $G(m, 1, 3)$ with $\gcd(m, 3) = 1$,  
			\item $G(3k, 3, 3)$ with $k > 1$. 
		\end{enumerate}
		In all three cases the parameter for middle convolution is a root of unity of order $m$.
	\end{theorem}

    \begin{theorem} \label{rank 4 imprim}
		In the rank $4$ imprimitive groups, nice $4$-tuples and nice $5$-tuples exist in $W = G(m, p, 4)$ if and only if\, $W$ is one of the following groups: 
        \begin{enumerate}
            \item $G(2, 2, 4)$, 
            \item $G(4, 4, 4)$.
        \end{enumerate}
	\end{theorem}	

    The next collection of theorems all pertain to primitive reflection groups of rank 3 and 4. For these groups, we performed a computer search for nice tuples and analyzed our findings. In our computations nice tuples within each group were partitioned into different \emph{types} based on eigenvalues (see \Cref{type def} and \Cref{type rmk}), and we counted the nice tuples in each group by type. We summarize the distinct rank~2 MCG-finite tuples (up to simultaneous conjugation and braid group equivalence) produced by taking middle convolution of the different nice tuples (up to type equivalence). We also note whether any of the tuples generate an infinite subgroup of $\SL_2$, providing examples of MCG-finite local systems with infinite monodromy group.  

    \begin{theorem}
        All primitive reflection groups of rank $3$ contain nice $3$-tuples, and all but $G_{24}$ and $G_{26}$ contain nice $4$-tuples. Tables~\ref{tab: Prim Rank 3 of 3-tup Data Table} and~\ref{tab: Prim Rank 4 of 4-tup Data Table} summarize the data before and after applying middle convolution. 
    \end{theorem}

    \begin{table}[h]
\centering
\begin{tabular}{|c|c|c|c|}
\hline
\multicolumn{4}{|c|}{Summary of MCG-finite rank 2 local systems on $\Sigma_4$ produced from rank 3 primitive groups} \rule[0pt]{0pt}{3ex}\\ 
\hline
\shortstack{Rank 3 \\ primitive group} &
  \# Nice 3-tuples &
  \# MCG-finite rank 2 tuples & 
  \shortstack{ \phantom{aaaaaaaaaaaaaaa} \\ \# Tuples generating infinite \\ subgroup of $\SL_2(\C)$} \\ [1ex]
\hline
$G_{23}$ & 3  &  6  &
  3 \\
\hline
$G_{24}$ &
  1 & 3 & 3\\
\hline
$G_{25}$ & 4 &
  8 & 5 \\
\hline
$G_{26}$ & 3 & 7 & 5 \\
\hline
$G_{27}$ & 3 & 9 & 6 \\
\hline
\end{tabular}
\caption{}
\label{tab: Prim Rank 3 of 3-tup Data Table}
\vspace{1ex}
\centering
\begin{tabular}{|c|c|c|c|}
\hline
\multicolumn{4}{|c|}{Summary of MCG-finite rank 2 local systems on $\Sigma_5$ produced from rank 3 primitive groups} \rule[0pt]{0pt}{3ex}\\ 
\hline
\shortstack{Rank 3 \\ primitive group} &
  \# Nice 4-tuples &
  \# MCG-finite rank 2 tuples & 
   \shortstack{ \phantom{aaaaaaaaaaaaaaa} \\ \# Tuples generating infinite \\ subgroup of $\SL_2(\C)$} \\ [1ex]
\hline
$G_{23}$ & 1  &  1  &
  1 \\
\hline
$G_{24}$ &
  NA & - & -\\
\hline
$G_{25}$ & 3 &
  2 & 1 \\
\hline
$G_{26}$ & NA & - & - \\
\hline
$G_{27}$ & 1 & 1 & 0 \\
\hline
\end{tabular}
\caption{}
\label{tab: Prim Rank 4 of 4-tup Data Table}
\end{table}

    \begin{remark} \label{same nice+ diff mcg}
	For some nice $T$-tuples, it is possible that several eigenvalues $\lambda$ have multiplicity $T-2$ (\ref{eigv cond}). Hence, the different choices of parameter $\lambda$ give rise to possibly different rank 2 MCG-finite $(T+1)$-tuples after middle convolution. For example, when $T= 3$, we only need eigenvalues of multiplicity~1, which is why each type of nice 3-tuple can produce up to three distinct MCG-finite 4-tuples. It is also possible that different eigenvalues or different types produce equivalent MCG-finite tuples after middle convolution, resulting in fewer MCG-finite tuples than nice tuples. An example of this is the nice 4-tuples produced from the rank 3 group $G_{25}$.
\end{remark}

    Now we discuss primitive groups of rank 4. 

    \begin{theorem}
        In $G_{28}$ there is only one nice $4$-tuple, which produces one MCG-finite $5$-tuple after middle convolution. There are no nice $5$-tuples in this group. 
    \end{theorem}
    
    \begin{theorem}
        In $G_{29}$ and $G_{31}$ there are no nice tuples. 
    \end{theorem}

    For the remaining two primitive groups of rank 4, some of the nice tuples exhibit interesting behaviour. After applying middle convolution, when we put the matrices into $\SL_2(\C)$, we have that one of the matrices of the MCG-finite $(T+1)$-tuple becomes the identity (\ref{I excl}). This means they come from MCG-finite local systems on $\Sigma_T$ (\textit{i.e.}~MCG-finite $T$-tuples) by the operation of inserting the identity matrix as the $\supth{(T+1)}$ matrix. In this way we get that our nice $T$-tuples produce MCG-finite $T$-tuples of rank 2 rather than the expected MCG-finite $(T+1)$-tuples of rank 2.

    \begin{theorem}
        In $G_{30}$ there are three nice $4$-tuples, each producing one distinct MCG-finite $5$-tuple after middle convolution. All three of these $5$-tuples comes from $4$-tuples by inserting the identity, so the nice $4$-tuples in this group produce MCG-finite $4$-tuples and no MCG-finite $5$-tuples. There are no nice $5$-tuples in this group. 
    \end{theorem}

    \begin{theorem}
        In $G_{32}$ there are three nice $4$-tuples, resulting in four distinct MCG-finite $5$-tuples after middle convolution. One of these $5$-tuples generates an infinite subgroup. Two of the remaining three $5$-tuples come from MCG-finite $4$-tuples by inserting the identity matrix. This group also contains at least two nice $5$-tuples, producing at least one example of an MCG-finite $6$-tuple.  The $6$-tuple does not generate an infinite subgroup. In total $G_{32}$ produced two MCG-finite $4$-tuples, two MCG-finite $5$-tuples, and at least one MCG-finite $6$-tuple. 
    \end{theorem}

    Note that the search for nice 5-tuples in $G_{32}$ was too large computationally, so this list may be incomplete. 
    For each nice tuple in the primitive groups, we provide representatives (up to simultaneous conjugation and braid group orbit equivalence) before (\Cref{prim}) and after middle convolution (in Appendices~\ref{appendix A},~\ref{appendix B}, and~\ref{appendix C}). We also list the size of the subgroup generated by the rank 2 matrices (after middle convolution), and in most cases we provide an approximate ``size'' for the braid group orbit in the appendices.

    A consolidated summary of which finite complex reflection groups contain nice tuples that can be used to produce MCG-finite rank 2 local systems on $\Sigma_{T+1}$ is given in Tables~\ref{tab:Imprimitive Data Table} and~\ref{tab:Primitive Data Table}.
  
\begin{table}[ht]
\centering
{\renewcommand{\arraystretch}{1.2}
\begin{tabular}{|c|c|c|c|}
\hline
Rank($W$) &
  \shortstack{ \phantom{aaaaaaaaaaaaaaa} \\ Imprimitive groups that \\ product MCG-finite rank 2 \\ local systems on $\Sigma_{4}$} &
  \shortstack{ \phantom{aaaaaaaaaaaaaaa} \\ Imprimitive groups that \\ product MCG-finite rank 2 \\ local systems on $\Sigma_{5}$} &
  \shortstack{ \phantom{aaaaaaaaaaaaaaa} \\ Imprimitive groups that \\ product MCG-finite rank 2 \\ local systems on $\Sigma_{6}$} \\ [1ex]
\hline
3 & \multicolumn{1}{c|}{$G(m, 1, 3), G(m, m, 3)$}  &  \shortstack{ \phantom{a} \\ $G(m, m, 3)$, for $m\neq 3$, \\ $G(m, 1, 3)$, for $\gcd(m, 3) = 1$\\ $G(m, 3, 3)$, $m \neq 3$ }  &
  NA \\ [1ex]
\hline
4 &
  NA &
  $G(2,2, 4), G(4, 4, 4)$ &
  $G(2, 2, 4), G(4, 4, 4)$ \\
\hline
$\geq 5$ &
  NA &
  NA &
  NA \\
\hline
\end{tabular}
}
\caption{For all the above, we assume $m > 1$.}
\label{tab:Imprimitive Data Table}
\vspace{2ex}
\centering
{\renewcommand{\arraystretch}{1.2}
\begin{tabular}{|c|c|c|c|}
\hline
\shortstack{Finite complex \\ reflection groups \\ \phantom{aaaaaaaaaaaaaaa}} & \shortstack{ \phantom{aaaaaaaaaaaaaaa} \\ \#MCG-finite rank 2 \\ local systems produced \\ on $\Sigma_{4}$} &
  \shortstack{ \phantom{aaaaaaaaaaaaaaa} \\ \#MCG-finite rank 2 \\ local systems produced \\ on $\Sigma_{5}$} &
  \shortstack{ \phantom{aaaaaaaaaaaaaaa} \\ \#MCG-finite rank 2 \\ local systems produced \\ on $\Sigma_{6}$} \\ \hline
$G_{23}$ & 6 & 1 & 0          \\ \hline
$G_{24}$ & 3 & 0 & 0          \\ \hline
$G_{25}$ & 8 & 2 & 0          \\ \hline
$G_{26}$ & 7 & 0 & 0          \\ \hline
$G_{27}$ & 9 & 1 & 0          \\ \hline
$G_{28}$ & 0 & 1 & 0          \\ \hline
$G_{29}$ & 0 & 0 & 0          \\ \hline
$G_{30}$ & 3 & 0 & 0          \\ \hline
$G_{31}$ & 0 & 0 & 0          \\ \hline
$G_{32}$ & 2 & 2 & At least 1 \\ \hline
\end{tabular}
\captionsetup[table]{labelsep=none}
\caption{}
\label{tab:Primitive Data Table} }
\end{table}

	\subsection{Prior work} \label{subsec: prior work}
	MCG-finite representations correspond to algebraic solutions to nonlinear differential equations, such as the Painlev\'e VI equation and more generally the  Schlesinger equations; see \cite{cou17}.
	Regarded as a nonlinear analogue of the hypergeometric equation, the Painlev\'e VI (PVI) equation, posed over 100 years ago, has inspired much work on the study of finite orbits. Algebraic solutions to the PVI equation correspond to irreducible MCG-finite representations when $n = 2$ and $T+1 = 4$ and motivate progress towards a nonlinear Schwartz list; see \cite{boalch2007towards}.   Among the many advances made in this area, notable contributions on finite orbits come from Hitchin \cite{hitchinponcelet}, Dubrovin--Mazzocco \cite{dubromazz2000}, Boalch \cite{boalch2005klein}, Kitaev \cite{kitaev2006}, and an eventual computer-aided classification by Lisovyy--Tykhyy \cite{T-L}. The history and further details regarding the 45 equivalence classes of exceptional algebraic solutions of PVI can be found in \cite{P6survey}; see Slide~18 of \textit{op.~cit.}~and \cite[Section~11.7]{boalch2007towards} for a summary. Research on finite orbits of the character variety is still ongoing; see \cite{Gol2024,calli2018}, and more recently \cite{bronsteinmaret}.
	
	The work of Dubrovin--Mazzocco originally studied the PVI equation in connection with reflection groups; see  \cite{dubromazz2000}. Expanding on this idea,  solutions constructed from finite reflection groups through special cases of middle convolution, like the ``Okamoto transform'',  were explored by others in \cite{boalch2005klein,okamototrans,dettweilerpainleve}. More generally, a computer-aided classification was provided by Tykhyy in \cite{Tyk}, with explicit orbits listed when $T+1 =4, 5$ and conjectural descriptions of the different possibilities provided for $T+1 \geq 6$.  We found exact matches for orbits given in this list, except in two cases where the closest match differs in just one choice of parameter. As that parameter is considered up to sign, we believe our orbits still belong to those close-match orbits. Details on these two tuples are provided in Sections~\ref{G_{25} 3-tups} and~\ref{G_{26} 3-tups}, and the differing coordinates are marked with an asterisk in the tables.
	
	The connection between finite complex reflection groups and middle convolutions was further explored more recently in the work of Lam--Landesman--Litt \cite[Corollary~1.1.7]{LLL}, who proved that all MCG-finite tuples with at least one matrix of infinite order  either are pullback representations (classified in \cite{diarra2013}) or come from finite complex reflection groups via middle convolution. 
	
	Our work investigates the finite complex reflection groups for MCG-finite tuples and examines their middle convolution in relation to the braid group orbit. In the applicable cases, we also explore how the examples we find relate to the classifications provided by Lisovyy--Tykhyy \cite{T-L} and Tykhyy \cite{Tyk}. Note that the analysis in these two papers excludes the tuples where all the matrices share a common eigenvector; these correspond to nonirreducible cases and were classified by Cousin--Moussard; see \cite[Section 3]{cousinmouss}. 

    Released shortly after our own work was posted to arXiv, theoretical results by Bronstein--Maret \cite{bronsteinmaret} completed the classification of finite orbits by studying the case where all local monodromies have finite order. In light of their research and the work before them, the tuples we provide serve as concrete examples of representations with finite orbits. 

    \section{Background}

    \subsection{Middle convolution} \label{mc}
	Defined by Katz \cite{katz1996}, middle convolution  is an invertible functor in the category of perverse sheaves depending on some parameter $\lambda \in \C^{\times}$ and is denoted $MC_{\lambda}$. Middle convolution can also be viewed as a functor in the category of local systems on the punctured sphere to itself, and with this interpretation, Katz proved that all irreducible rigid local systems on the punctured sphere can be constructed from rank 1 local systems via iteratively applying scaling and middle convolution operations.  
	
	Letting $k$ be a field and $F_{T}$ denote the free group on $T$ generators, middle convolution can equivalently be understood as a functor on the category of finite-dimensional $k[F_{T}]$-modules to itself as described by Dettweiler--Reiter \cite[Proposition~2.6]{KatzAlgo}. The relation to local systems on the $(T+1)$-punctured sphere is by taking $k = \C$ and viewing a tuple of $T$ matrices as defining the action of $\C[F_T]$ on the module, \textit{i.e.}~a representation of $\C[F_T]$. In this way irreducible tuples correspond to irreducible representations, which correspond to local systems.
	While middle convolution preserves irreducibility, see \cite[Corollary~3.6]{KatzAlgo}, it typically changes the dimension of the module (rank of the local system on the $(T+1)$-punctured sphere) and does not preserve basic properties, \textit{e.g.}~finiteness of the subgroup generated by the matrices, or equivalently, the monodromy group of the associated local system. Important to the study of finite orbits, one desirable property of middle convolution is that it is equivariant with the action of the braid group $B_{T+1}$; see \cite[Theorem 5.1]{KatzAlgo}. One way to construct nontrivial finite orbits is to take irreducible tuples generating a finite subgroup, hence finite braid group orbit, and taking the middle convolution to (hopefully) get a tuple generating an infinite group which will necessarily be irreducible and have finite braid group orbit. 
	
	The purely algebraic analogue by Dettweiler--Reiter \cite[Definition 2.5]{KatzAlgo} gives an explicit description of middle convolution in terms of tuples of matrices and relates some properties of middle convolution to the matrices in the tuple.

        \begin{definition}
		Given a tuple of $T$ matrices $\Abf = [A_1, A_2, \dots, A_T]$ with $A_i \in \GL_n(\C)$, the \emph{product} of such a tuple is the product of the matrices in the tuple, in the ordering given by the tuple. So the product of the tuple $\mathbf{A}$ is $A_1A_2\cdots A_T$. The \emph{inverse} of the tuple $\Abf$ is the tuple $[A_T^{-1}, \dots, A_1^{-1}]$, denoted $\mathbf{A^{-1}}$. The \emph{inverse product} of a tuple is the matrix product of the inverse tuple.
	\end{definition}

        \begin{definition}
		Given a tuple of $T$ matrices $\Abf = [A_1, A_2, \dots, A_T]$ with $A_i \in \GL_n(\C)$, we denote the middle convolution of $\mathbf{A}$ with parameter $\lambda \in \C^{\times}$ as $MC_{\lambda}(\Abf) = [\tilde{A}_1, \tilde{A}_2, \dots, \tilde{A}_T]$. Note that $MC_{\lambda}(\Abf)$ will again be a tuple of $T$ matrices, well-defined up to simultaneous conjugation, with $\tilde{A_i} \in \GL_{n'}(\C)$ for some $n'$ not necessarily equal to $n$. 
	\end{definition}
	
	\begin{remark}\label{mc no prod}
		Middle convolution does not preserve the product of the matrices as it is only well-defined up to simultaneous conjugation. 
	\end{remark}
    
	Let $V$ denote the underlying vector space associated to the $\C[F_{T}]$ module defined by a tuple $\mathbf{A}$. Then the relation between $n$ and $n'$ is as follows. 	

        \begin{theorem}\label{gen_prop of MC}
		Let $V$ denote the $\C[F_{T}]$-module defined by the irreducible tuple $\Abf = [A_1, A_2, \dots, A_T]$, and let $\lambda \in \C^{\times}$. 
	 If $\lambda \neq 1$, then
                  $$ \dim (MC_{\lambda}(V)) = \sum_{i = 1}^{T} \on{rk}(A_i -I) - (\dim(V) -\on{rk} (\lambda \cdot A_1A_2 \cdots A_T -I)). $$
	\end{theorem} 

        \begin{proof}
		The proof of this statement, as well as further details and properties of the middle convolution, can be found in \cite[Section 2]{MC-L} and \cite[Section 2.5]{KatzAlgo}. 
	\end{proof}

        \begin{remark} \label{rem: diff meanings of rank}
        Note that in the above, when we write $\on{rk} (A)$ for $A$ a matrix, we mean the dimension of its image subspace, \textit{i.e.}~as in the rank-nullity theorem. When we write $\rank (W) $ for $W$ a finite complex reflection group, we mean the rank of the canonical representation of $W$ provided by reflections on the vector space~$V$, so $\rank (W)  = \dim V$. When we simply write rank $n$ or rank 2, we mean the rank of the local system on $\Sigma_{T+1}$, or equivalently the rank of the associated representation of $\pi_1(\Sigma_{T+1})$. 
    \end{remark}

        In light of \Cref{gen_prop of MC}, to get an MCG-finite $(T+1)$-tuple of rank 2, we want to find a $T$-tuple of reflections (hence $\on{rk}(A_i -I) = 1$) such that the tuple product has a nontrivial eigenvalue of multiplicity $T-2$. This is why we search for \textit{nice tuples}; see Definition~\ref{nice tup}. 

    \begin{theorem}[Dettweiler--Reiter, 2007] \label{MC eigenvals}
        Let $\lambda \in \C^{\times}$ and $MC_{\lambda}(A_1, \dots, A_T) = [\tilde{A}_1, \dots, \tilde{A}_T]$. Let $A_{T+1}$, $\tilde{A}_{T+1}$ be the inverse products of the tuples $\mathbf{A}$ and $\mathbf{\tilde{A}}$, respectively. Writing $J(\alpha, l)$ for the Jordan block of size $l$ with eigenvalue~$\alpha$, the eigenvalues before and after middle convolution are related as follows:
        \begin{romanenum}
            \item Every Jordan block $J(\alpha, l)$ in the Jordan decomposition of $A_{k}$ contributes a Jordan block $J(\alpha \lambda, l')$ to the Jordan decomposition of $\tilde{A}_k$, where  
            \begin{align*}
            l' := \begin{cases} 
            l &\text{if\, $\alpha \neq 1, \lambda^{-1}$}, \\
            l -1 &\text{if\, $\alpha = 1$}, \\
            l+1 & \text{if\, $\alpha = \lambda^{-1}$}.
            \end{cases}
            \end{align*}
            The only other Jordan blocks in the Jordan decomposition of $\tilde{A}_k$ is $J(1, 1)$. 
            \item Every Jordan block $J(\alpha, l)$ in the Jordan decomposition of $A_{T+1} = (A_1A_2\cdots A_T)^{-1}$ contributes a Jordan block $J(\alpha \lambda^{-1}, l')$ to the Jordan decomposition of $\tilde{A}_{T+1} = (\tilde{A}_1 \tilde{A}_2\cdots \tilde{A}_T)^{-1}$, where  
            \begin{align*}
            l' := \begin{cases} 
            l &\text{if\, $\alpha \neq 1, \lambda$}, \\
            l +1 &\text{if\, $\alpha = 1$}, \\
            l-1 & \text{if\, $\alpha = \lambda$}.
            \end{cases}
            \end{align*}
            The only other Jordan blocks in the Jordan decomposition of $\tilde{A}_{T+1}$ is $J(\lambda^{-1}, 1)$. 
        \end{romanenum}
    \end{theorem}

    \begin{proof}
        This statement in terms of the Jordan decomposition is a reformulation of \cite[Lemma 5.1]{dettweilerpainleve} and follows from \cite[Lemma 4.1]{KatzAlgo}. See also \cite[Proposition~3.2.2]{LLL}.
    \end{proof}
    
	Another property of middle convolution that will greatly simplify our computations is the following. 

	\begin{theorem}[Dettweiler--Reiter, 2000]\label{Inverse MC Thm}
		If $\mathbf{A}$ is an irreducible tuple of rank greater than $1$, then the modules $MC_{\lambda^{-1}}(\mathbf{A^{-1}})$ and $MC_{\lambda}(\mathbf{A})^{-1}$ are isomorphic. 
	\end{theorem}

        \begin{proof}
		This is proved in \cite[Theorem 5.5]{KatzAlgo}.
	\end{proof}
	
	Our code for middle convolution uses the formulation given by Dettweiler--Reiter in \cite[Definition~2.1]{MC-L} rather than their formulation in \cite[Section 2]{KatzAlgo}, which only differs by changing the order of blocks in the matrices. The code requires a matrix group $W$ containing the matrices in the tuple $\Abf$, and the base ring of $W$ should contain the parameter $\lambda$ used for middle convolution. Note that the final function \texttt{Compute\_MC} will automatically check if our intermediate constructions act invariantly on the subspaces $\mathcal{K}$ and $\mathcal{L}$, as defined in the notation of \cite[Definition of 2.1]{MC-L}. The full computer code can be found online at \cite{Vay24}.	
	\subsection{Finite complex reflection groups} \label{FCRG}
	A real reflection can be identified as a matrix with eigenvalue $-1$ of multiplicity 1 and all other eigenvalues equal to 1. Pseudo-reflections refer to matrices with a root of unity as an eigenvalue of multiplicity 1 and all other eigenvalues equal to 1.  A \emph{finite complex reflection group} (\emph{FCRG}\,) is a finite group generated by pseudo-reflections. 
	As we are only working with complex groups, the distinction between real or complex eigenvalues is not relevant, and henceforth reflections can refer to real reflections or pseudo-reflections.
	
	All irreducible finite complex reflection groups  were classified by Shephard--Todd \cite{stfinite} as belonging to an infinite family of \textit{imprimitive} reflection groups or being one of 34 exceptional \textit{primitive} cases. These groups come with  a canonical finite-dimensional representation in terms of an action generated by reflections on a vector space $V$.  Every FCRG of rank $n$, defined below, has a minimal generating set consisting of either $n$ or $n+1$ reflections. 

	\begin{definition} \label{rank+wg}
		The \emph{rank} of a reflection group, presented as being generated by reflections on a vector space~$V$, is the dimension of $V$.
		 A rank $n$ group is said to be \emph{well-generated} if it can be generated by $n$ reflections. 
	\end{definition}

        As we want irreducible MCG-finite $(T+1)$-tuples, we know that our nice $T$-tuple (see Definition~\ref{nice tup}) must at least be of length $n$ to generate an irreducible FCRG, so in each group we only need to search for $T$-tuples with $T > n$. Among the rank 3 and 4 primitive irreducible groups, only $G_{31}$ is not well-generated (see \cite[Appendix D.3]{ld2009}. An imprimitive group $G(m, p, n)$ is well-generated if $p = m$ or $p = 1$. 
	
	Associated to each FCRG of rank $n$ is a set of $n$ invariants called the degrees of the reflection group; see \cite[Theorem 5.1]{stfinite} and \cite[Table D.3]{ld2009}. An important property of these invariants relevant for our purposes is their relation to the eigenvalues of the group elements. Note that we already know the eigenvalues for any element of the group must be roots of unity since the group is finite.
        
	\begin{theorem}\label{ReflEigvThm}
		Let $W$ be a FCRG of rank $n$, and let $\{d_i\}_{i = 1, \dots, n}$ denote the degrees of\, $W$. Suppose $\lambda$ is some primitive $\supth{d}$ root of unity that is an eigenvalue for some element of\, $W$. Then, $d \mid d_i$ for some degree $d_i$. Moreover, the maximum dimension of the eigenspace corresponding to $\lambda$ for any element of the group is the number of degrees $d_i$ that are divisible by $d$, and for some element in $W$ the eigenspace corresponding to $\lambda$ achieves this maximal dimension. 
	\end{theorem}

        \begin{proof}
		This is proved in \cite[Theorem 3.4]{OrigEigv}.
	\end{proof}

         Theorem~\ref{ReflEigvThm} will be helpful in our computations for the primitive groups as we know how large our field needs to be to ensure the code finds all eigenvalues. 
	
	Further properties specific to imprimitive reflection groups are described in Section~\ref{imprim}.
	
	\subsection{Representations and character varieties}\label{charvar}
	
	Given a finitely generated group $\Gamma$, a representation of $\Gamma$ is a homomorphism $\rho\colon \Gamma \rightarrow G$, where $G < \GL_n(\C)$ is some algebraic group. Considering the vector space $V$ the matrices $\rho(\gamma)$ act on, $\rho$ is said to be irreducible if there is no nonzero subspace $U \subsetneq V$ that is left invariant by the
    matrix $\rho(\gamma)$ for all $\gamma \in \Gamma$. Two representations $\rho$ and $\rho'$ are isomorphic if there exists some $g \in G$ such that $\rho(\gamma) = g^{-1}\rho'(\gamma) g$ for all $\gamma \in \Gamma$. 	The character of a representation  $\rho$ is the map $\chi_{\rho}\colon \Gamma \rightarrow \C$ given by $\chi_{\rho}(\gamma) = \tr \rho(\gamma)$. Isomorphic representations have the same character since trace is fixed under conjugation. The character variety $\mathfrak{X}(\Gamma, G)$ parametrizes isomorphism classes of semi-simple representations $\rho\colon \Gamma \rightarrow G$ in terms of suitable traces.  
	
    In the case of the punctured sphere $\Sigma_{T+1}$, the fundamental group is
    $$\pi_1(\Sigma_{T+1}, x) = \inn{\gamma_1, \dots, \gamma_{T+1} \mid \gamma_1 \cdot\gamma_2\cdots \gamma_{T+1} = 1}.$$
    Any representation of $\Gamma = \pi_1(\Sigma_{T+1})$ is identified by $\rho(\gamma_i) = M_i \in G$ such that $M_1 \cdot M_2 \cdots M_{T+1} = I$ is the identity matrix. Two such representations  $\rho$ and $\rho'$ are equivalent if the tuple $[\rho(\gamma_1), \dots, \rho(\gamma_{T+1})]$ is related to $[\rho'(\gamma_1), \dots, \rho'(\gamma_{T+1})]$ by simultaneous conjugation.
	
	Taking $G = \SL_2(\C)$, we can consider the $\SL_2(\C)$-character variety of $\pi_1(\Sigma_{T+1})$. 
	
	Given a tuple of matrices $[M_1, M_2, \dots, M_{T+1}]$, let
        \begin{align*}
		t_i = \Trace(M_i), \quad t_{ij} = \Trace(M_iM_j), \quad t_{ijk} = \Trace(M_iM_jM_k) \quad\text{for }1\leq i < j< k \leq T+1.
	\end{align*}
	Since $\pi_1(\Sigma_{T+1}) = F_{T}$, we have a natural equivalence $\mathfrak{X}(\pi_1(\Sigma_{T+1}), \SL_2(\C)) \simeq \mathfrak{X}(F_{T}, \SL_2(\C))$.  For our purposes, this means $M_{T+1}$ will always be $(M_1M_2\cdots M_T)^{-1}$, so $t_{T+1} = \Trace(M_1M_2\cdots M_T)$ as all the matrices $M_i$ will be in $\SL_2(\C)$, where matrices have the same trace as their inverse.
	
	\begin{theorem}\label{trace coords thm}
		Let $[M_1, M_2, \dots, M_{T+1}] \in \SL_2(\C)^{T+1}$ be such that $\Pi M_i = I$. Then up to simultaneous conjugation, any such $(T+1)$-tuple can be identified by the following trace coordinates: 
		\begin{enumerate}
			\item for $T = 3$: $(t_1, t_2, t_3, t_{12}, t_{13}, t_{23}, t_4)$, 
			\item for $T = 4$: $(t_1, t_2, t_3, t_4, t_{12}, t_{13}, t_{14}, t_{23}, t_{24}, t_{34}, t_{123}, t_{124}, t_{134}, t_{234}, t_5)$.
		\end{enumerate}
	\end{theorem}

	\begin{proof}
		For a proof and further discussion of character varieties of the punctured sphere, see \cite[Section~4]{goldman2009tr} and \cite[Section 5]{VM2002}. 
	\end{proof}

    \begin{definition} \label{defsignature}
    The trace coordinates above are referred to as the \emph{signature} of a tuple. 
    \end{definition}

	\section{Methods} \label{methods}
	For the rest of this written work, fix $\zeta_d = e^{\frac{2\pi i}{d}}$ as a primitive $\supth{d}$ root of unity. 
		
	Since the product of $[A_1, \cdots, A_{T+1}]$ must be the identity, we know any MCG-finite tuple is determined by the first $T$ matrices, up to simultaneous conjugation. As middle convolution does not preserve products (see Remark~\ref{mc no prod}), our idea is to take middle convolution of the first $T$ matrices to get $2\times 2$ matrices and then induce a representation of $\pi_1(\Sigma_{T+1})$ by taking the $\supth{(T+1)}$ matrix to be the inverse product. The condition for middle convolution being of rank 2 (stated in Theorem~\ref{gen_prop of MC}) depends on the dimension of the subspace fixed by it,  or equivalently the multiplicity of the eigenvalue 1 of each matrix in the tuple and the eigenvalues of their product. For much better control on the former, we consider reflections since they only have one nontrivial eigenvalue, of multiplicity 1. 
	
	Thus, for each irreducible finite complex reflection group $W$ of rank $n$, we want to find all irreducible $T$-tuples of reflections $\Abf = [A_1, A_2, \dots, A_T]$ in $W$ such that for a suitable choice of parameter $\lambda$, $MC_{\lambda}(\Abf)$ has rank 2. Given $T$, the formula given in Theorem~\ref{gen_prop of MC} tells us that the product $\Pi_{i = 1}^T A_i$ of reflections in~$\Abf$ should have a nontrivial eigenvalue $\lambda^{-1}$ of multiplicity $T-2$ so that $MC_{\lambda}(\Abf) = [\tilde{A_1}, \tilde{A_2}, \dots, \tilde{A_T}]$ is a $T$-tuple of rank 2. Since $W$ itself is irreducible, if we further require  that $W = \langle A_1, A_2, \dots, A_{T} \rangle$, we can avoid ``over-counting'' as some reflection groups contain other reflection groups as subgroups, and we only want to study irreducible MCG-finite tuples. 
    These conditions are precisely our definition of nice tuples given in \Cref{nice tup}!
	
	\subsection{Computing the braid group orbit}\label{BGSection}
	The braid group can be presented as $B_{T+1} = \inn{\sigma_1, \dots, \sigma_T \mid \sigma_i\sigma_{i+1}\sigma_i = \sigma_{i+1}\sigma_i\sigma_{i+1}, \sigma_i\sigma_j = \sigma_j\sigma_i}$, where the first relation holds for $1 \leq i\leq T-3$ and the second for $i-j \ge 2$. A generator $\sigma_i$ acts on $(T+1)$-tuples in the following way: 
	$$\sigma_i(M_1, M_2, \dots, M_T, M_{T+1}) = \left(M_1, \dots, M_{i-1}, M_iM_{i+1}M_{i}^{-1}, M_{i}, M_{i+2}, \dots, M_{T+1}\right).$$	
	
	Let $F_+$ denote the set of all words in positive powers of these generators. Further, let $\Orb(B_{T+1}, x)$ denote the orbit of some tuple $x$ under the action of $B_{T+1}$ and $\Orb(F_+, x)$ denote the set of images of $x$ under the action of words in $F_+$. 
	
	\begin{lemma}
		For a tuple $x$ of matrices, $\Orb(B_{T+1}, x)$ is finite if and only if\,  $\Orb(F_+, x)$ is finite. 
	\end{lemma}

	\begin{proof}
		Since $F_+ \subset B_{T+1}$, one direction is obvious. For the other direction, first note that for any fixed $g \in B_{T+1}$, if the set $\{g^{i} x : i \geq 0\}$ is finite, then we can write $g^{-1}x = g^n x$ for some $n \geq 0$, so we get that $\{g^i x: i \in \Z\}$ must be finite. Further, since the elements of $B_{T+1}$ act one at a time, it is enough to study how words in positive powers of the generators act on $x$. Thus, if $\Orb(F_{+}, x)$ is finite, then the braid group orbit must be finite as well. 
	\end{proof}
	
	To help with computing the braid group orbit, we need a way to capture the equivalence relation of tuples being related by simultaneous conjugation. Using character varieties, we can write down trace coordinates for our tuple, which we know how to do from Theorem~\ref{trace coords thm} when the matrices are in $\SL_2(\C)$.
	
	\begin{definition}
		Given a nice tuple $\Abf = [A_1, A_2, \dots, A_T]$, let  $MC_{\lambda}(\Abf) = [\tilde{A}_1, \tilde{A}_2, \dots, \tilde{A}_{T}]$ denote the middle convolution. If $(c_1, c_2, \dots, c_{T+1}) \in (\C^\times)^{T+1}$ is such that $c_i \tilde{A}_i \in \SL_2(\C)$ for $i = 1, \dots T$ and $c_{T+1} ((\tilde{A}_1\tilde{A}_2\cdots\tilde{A}_T)^{-1}) \in \SL_2(\C)$ with $\Pi^{T+1} c_i = 1$, the tuple 
		$$ \mathbf{\hat{A}} = [c_1 \tilde{A}_1, c_2 \tilde{A}_2, \dots, c_{T+1}(\tilde{A}_1\tilde{A}_2\cdots\tilde{A}_T)^{-1}] $$
		is said to be \emph{induced} by $MC_{\lambda}(\Abf)$ or by  $\Abf$ if the parameter $\lambda$ is clear. The induced tuple $\mathbf{\hat{A}}$ is a tuple in $(\SL_2(\C))^{T+1}$ which multiplies to the identity.
	\end{definition}

    \begin{remark}\label{induced tup equiv}
    Note that over $\C$, it is always possible to find a scalar $c_i$ so that $\det (c_i A_i) = 1$ for any $A_i \in \GL_n(\C)$, as any $\supth{n}$ root of $\det A_i$ can be taken for $c_i$. As the matrices $[\tilde{A}_1, \dots, \tilde{A}_{T+1}]$ multiply to the identity, there is always some choice of roots that multiply to the identity.  In this work we choose some scalars $(c_1, \dots, c_{T+1})$ to construct our induced tuples. There may be multiple choices of such tuples of scalars, so our results should be considered up to choosing a different one. 
    \end{remark}

	When we say trace coordinates of a tuple $MC_{\lambda}(\Abf)$, we will always mean the trace coordinates of the induced tuple $\mathbf{\hat{A}}$ unless otherwise specified.
	We want to compute the orbit of induced tuples $\mathbf{\hat{A}}$ up to simultaneous conjugation. This is equivalent to computing the orbit of the point associated to $\mathbf{\hat{A}}$ in the character variety $\mathfrak{X}(\pi_1(\Sigma_{T+1}), \SL_2(\C))$. Note that since all of our tuples come from FCRG and middle convolution is equivariant with the action of the braid group, we already know all the orbits are finite. 
	
	The code for the algorithm described below can be found online at \cite{Vay24}.
	
	Given a tuple $\mathbf{M} = [M_1, \dots, M_{T+1}]$ of $\SL_2(\C)^{T+1}$ such that their product is the identity, the braid group $B_{T+1}$ orbit is computed as follows. 
	
	We first compute the signature of $\mathbf{M}$, and then apply each of the $T$ generators to the tuple $\mathbf{M}$. If $\sigma_i(\mathbf{M})$ has a new signature, then we store the tuple in \texttt{To\_Test} and the signature in \texttt{Signatures} and add $\mathbf{M}$ to \texttt{Tested}.  We then apply $\sigma_i$ again to the (at most $T$) new tuples in \texttt{To\_Test} we got from the action on $\mathbf{M}$, once again storing only tuples which have new signatures. We repeat this process until the sets \texttt{To\_Test} and \texttt{Tested} are the same. The function returns the set \texttt{Tested}, which should contain only one representative per distinct signature, and can optionally return the set \texttt{Signatures} as well. 
	Note that here our definition of signature is the trace coordinates (see Definition~\ref{defsignature}), which  differs from the definition used in Tykhyy's classification in terms of residues \cite{Tyk}.
	
	\subsection{Finding nice tuples} \label{finding nice tups}

	Nice $T$-tuples correspond to  MCG-finite $(T+1)$-tuples in the following way: given a nice tuple $[A_1, \dots, A_T]$, we get a corresponding representation  $[A_1, \dots, A_{T}, A_{T+1}]$ of the $(T+1)$-punctured sphere if we take $A_{T+1} = (A_1A_2 \cdots A_T)^{-1}$. 
	Thus, it is enough to search each complex reflection group for nice tuples.
	
	By Theorem~\ref{ReflEigvThm}, an eigenvalue of multiplicity $T-2$ for an element in $W$ must be a $\supth{d}$ root of unity for~$d$ dividing at least $T-2$ degrees of $W$. 
	Any reflection group can be constructed in Magma by providing the generating reflections; the irreducible finite complex reflection groups can be easily constructed in Magma by providing the Shephard--Todd number or the corresponding parameters for $G(m, p, r)$. In most cases Magma generates the group $W$ via the reflection representation over the trace field of $W$; see \cite[Tables 1 and~2]{feit2003}.

	One issue is that not all characteristic polynomials will split over the trace field, as Magma computes eigenvalues for matrices over the field of entries. Since our search depends on the eigenvalues of the product of matrices in the tuple(s), we must change the base field of $W$ to the smallest cyclotomic field containing all the possible eigenvalues. The largest this extension will be is $\Q(\zeta_d)$, where $d$ is the least common multiple  of all the degrees of $W$, but in certain cases smaller fields might suffice. For large $d$ computation time is typically slower, but we have expanded the field to the LCM of the degrees unless otherwise stated. For example, in the case of $W = G_{32}$, the degrees are $ \deg(W) = \{12,18,24,30 \}$, and so $\lcm (\deg(W)) = 360$, which slows down computation time considerably. For this group, we prove in Section~\ref{G32 4-tups} that it is enough to search for 4-tuples over $\Q(\zeta_{12})$ by showing any eigenvalue of multiplicity 2 must be a $\supth{12}$ root of unity. 
	
	Before outlining the steps of our search, we define some terms. 
	
	\begin{definition} \label{type def}
		Given a tuple of matrices $ \mathbf{A} = [A_1, A_2, \dots, A_T]$, the \emph{type} of $\mathbf{A}$ is determined by the eigenvalues of $(A_1A_2 \cdots A_T)^{-1}$ counted with multiplicity. Tuples whose inverse product has the same eigenvalues with multiplicity are said to be of the same type.  An \emph{exemplar} of a type is a choice of nice tuple of that type. Given a tuple of some type $A$, the type of the inverse tuple is called the \emph{inverse type} of $A$, and the two types together make an \emph{inverse pairs}. 
	\end{definition}

	The use of types allows us to partition a set of nice tuples based on the eigenvalues of the inverse product. Since there are many nice tuples of the same type, we choose exemplars for each type to perform the in-depth computations on. Note that the choice of exemplar for any type is arbitrary and that it may be possible that different exemplars result in different orbits. To ensure this does not happen, we go back through all the tuples for each type and make sure that it belongs to the braid group orbit we compute for the middle convolution of the exemplar. More details on the braid group computations can be found in Section~\ref{BGSection}.  
	
	In some groups we get types that are inverse pairs. In these cases we only provide the computations for one exemplar from each pair. To ensure the inverse type would not produce a new orbit, we check that the inverse tuple of the middle convolution lies in the orbit of the middle convolution of the exemplar; this is enough by the relationship between inverse tuples and inverse parameters given in Theorem~\ref{Inverse MC Thm}. 
	The code for how we checked each type is available at \cite{Vay24}. 

	\begin{remark} \label{type rmk}
		Aside from a few  exceptional cases that were computationally too large, we verified that for each choice of the parameter $\lambda$, the rank 2 MCG-finite tuples coming from nice tuples of the same type are in the same braid group orbit, as are exemplars from the inverse type. In this sense types can be viewed as ``almost'' equivalence classes for nice tuples up to inverse pairs.
	\end{remark}
	
	After coercing the matrix group into a larger field if necessary, the general process for searching for $T$-tuples is as follows: 
	\begin{enumerate}
		\item Fix a representative $Q_i$ for each conjugacy class of reflections in $W$, and fix an ordering $[Q_1, \ldots, Q_n]$.
		\item Given $T$, construct the list $M$ of distinct sets of $T$ reflections in $W$, using the ordering of $Q$ to eliminate some tuples equivalent up to simultaneous conjugacy. 
		\item Search $M$ for all sets such that the matrices generate $W$ (this means the tuple is irreducible) and the product of the matrices, for some order of multiplying, has an eigenvalue of multiplicity $T-2$. Each such set is then organized into a tuple based on the order of multiplication and stored in a list X of irreducible tuples. 
		\item For each nice tuple $x \in X$, compute the eigenvalues of the inverse product of $x$, and partition $X$ into different types corresponding to the different sets of eigenvalues of the inverse product. 
		\item Choose exemplars $\Abf$, $\mathbf{B},\ldots$ for each type, and compute the middle convolutions  $MC_{\lambda_1}(\Abf)$, $MC_{\lambda_2}(\Abf)$, $MC_{\lambda}(\mathbf{B}),\ldots$ for suitable choice(s) of parameter $\lambda$.
		\item Compute the subgroup of $\GL_2(\C)$ generated by $MC_{\lambda}(\Abf), MC_{\lambda}(\mathbf{B}), \dots$
		\item For each middle convolution of the exemplars, find a suitable character $[a_1, \dots, a_{T+1}]$, and construct the induced tuple $\mathbf{\hat{A}}$.
		\item Compute the braid group orbit of $\mathbf{\hat{A}}$, and identify the $\SL_2(\C)$ subgroup generated by the matrices of~$\mathbf{\hat{A}}$.
		\item Check that the other exemplars and inverse tuples do not produce new orbits.
		\item Use the tuples in the braid group orbit to compute the parameter and residue signatures of the full equivalence class as defined in the previous classifications to try and match our examples to theirs. 
	\end{enumerate}
	We recorded all the details for our exemplars and the middle convolution computations for each of the primitive reflection groups of rank 3 and 4. The middle convolution tables are listed in the appendix. The types and exemplars are listed in the section for each group, as are the comparison tables with the other classifications.  
	More details on what we use to compare with the existing classifications of \cite{T-L} and \cite{Tyk} are given in Sections~\ref{prim3} and~\ref{prim4}.
	
		\section{Imprimitive reflection groups}\label{imprim}
	Imprimitive reflection groups $G(m, p, n)$ are parameterized by $m, p , n \geq 1$ with $p \mid m$. These groups can be viewed as $n \times n$ generalized permutation matrices, which each row and column containing only one nonzero entry which is an $\supth{m}$ root of unity such that the product of all the nonzero entries is an $\supth{\left(\frac{m}{p}\right)}$ root of unity.  For $m > 1$ and $(m, p, n) \neq (2, 2, 2)$, these groups are always irreducible, meaning they act irreducibly on $\C^n$ (see \cite[Section 2]{cohenfin}). 
	
	An element $A\in G(m, p, n)$ can be written as $A = [a_1, a_2, \dots, a_n \mid \sigma]$ for $\sigma \in S_n$, $a_i \in \Z$, and $p$ dividing $a_1 + a_2+ \cdots + a_n$. This element can be identified with the matrix where the $\supth{(i, \sigma(i))}$ entry of $A$ is $\zeta_m^{a_i}$ and all the other entries are 0; see \cite[Section 2.3]{shiautos}. Note that all these $a_i$ should be read modulo $m$ since we are working with powers of an $\supth{m}$ root of unity. 
	Additionally, there is a natural surjective homomorphism $\vphi\colon G(m, p, n) \rightarrow S_n$ given by sending $A$ to $\sigma$. 
	
	One observation that follows from the definition is that for any $p$,  $G(m, m, n) \leq G(m, p, n)$ and $G(m, p, n) \leq G(m, 1, n)$ as subgroups.  
	
	Just as in the primitive case, the degrees of imprimitive groups tell us about the possible eigenvalues.  Writing $q = \frac{m}{p}$, we have 
	\begin{align}
		\deg (G(m, p, n)) &= \{m, 2m, \dots, (n-1)m, qn\}, \\
		|ZG(m, p, n)| &= q \gcd(p, n),
	\end{align}
	where $ZG(m, p, n)$ is the center of $G(m, p, n)$ and a cyclic subgroup.

	Reflections in the group $G(m, p, n)$ will be of one of the following two types:
	\begin{enumerate}[label={\textit{Type}~\arabic*:}, ref=\arabic*,leftmargin=!,widest={Type 1:}]
		\item \label{type1} These are elements such that $\sigma = (ij)$ is a transposition with $a_k \equiv 0 \mod m$ for $k\neq i, j$ and $ a_i + a_j \equiv 0 \mod m$. All reflections of type~\ref{type1} are of order 2 and lie in the subgroup $G(m, m, n)$ of $G(m, p, n)$. Using the notation from \cite[Section 2.3]{shiautos}, we write a type~\ref{type1} reflection as $s(i, j; a_i)$ with the understanding that this is the same as $s(j, i; -a_i)$.
		\item \label{type2} These are elements such that $\sigma = e$ and for some $i$, $a_i \not \equiv 0  \mod m$ and $a_k \equiv 0 \mod m$ for all $k \neq i$. These matrices are diagonal matrices of order $m/\gcd(m, a_i)$ and will be denoted $s(i; a_i)$. Type~\ref{type2} matrices can only exist when $p < m$. 
	\end{enumerate}

	\begin{theorem}\label{GenSet}
		Let $S$ be a set of reflections with minimal possible cardinality such that $S$ generates $G(m,p,n)$. Then $S$ satisfies the following conditions: 
		\begin{enumerate}
			\item If\, $p = 1$, then $S$ contains $n-1$ reflections of type~{\rm\ref{type1}} and one reflection of type~{\rm\ref{type2}} of order $m$.
			\item If\, $p = m$, then $S$ contains $n$ reflections of type~{\rm\ref{type1}}.
			\item If\, $p \neq 1, m$, then $S$ contains $n$ reflections of type~{\rm\ref{type1}} and one reflection of type~{\rm\ref{type2}} of order $\frac{m}{p}$.

		\end{enumerate}
	\end{theorem}

	\begin{proof}
		The theorem is stated in \cite[Lemma 3.2]{shiautos}, with proofs in \cite[Section 2.10 and Lemma 2.1]{shi2005mm} and \cite[Lemma 2.2]{shi2005mp}.
	\end{proof}
        
	{\samepage
	\begin{theorem}\label{imprim homo}
		For any reflection $r \in G(m, p, n)$, the homomorphism $\vphi\colon G(m, p, n) \rightarrow S_n$ must satisfy one of the following: 
		\begin{enumerate}
			\item $\vphi(r)$ is a transposition and $r$ has order 2.
			\item $\vphi(r)$ is the identity.
		\end{enumerate} 
	\end{theorem}}

	\begin{proof}
		This is true by the definition of reflections and the definition of the homomorphism. For a more thorough examination of related properties, we refer to \cite[Proposition 2.2]{cohenfin}.
	\end{proof}
	
	Given any set of reflections generating an irreducible imprimitive group, the permutations associated to the  reflections in the set must generate a transitive subgroup of $S_n$, or else the group cannot be irreducible. Note that in the case where the associated permutations are transpositions (\textit{i.e.}~the reflections are of type~\ref{type1}), the only transitive subgroup of $S_n$ generated by transpositions is $S_n$ itself; see \cite[Lemma 2.13]{ld2009}. 
       
	A subgroup of a reflection group generated by reflections is called a reflection subgroup. We modify some notation from \cite{wang2010} to state relevant theorems for identifying reflection subgroups in the imprimitive groups. Let $X = \{s(a_h, a_{h+1}; b_h)\}$, $a_h \neq a_{h+1}$,  be a subset of reflections in $G(m, m, n)$. We can associate to $X$ the graph $\Gamma_X$ with nodes $\{a_i\}$ and edges $(a_h, a_{h+1})$ for $s(a_{h}, a_{h+1}; b_h)$ in $X$, where we allow multi-edges but no self-loops (\textit{i.e.}~no edges from a node to itself). Letting $\inn{X}$ denote the reflection group generated by the reflections in $X$, we see that this group is irreducible if and only if the graph $\Gamma_X$ is connected. 
    	Whenever $\Gamma_X$ is such that the graph contains exactly one cycle, we can define a parameter $\delta(X)$ as follows:  given a cycle consisting of $2 \leq r \leq n$ edges corresponding to reflections $\{s(a_{h}, a_{h+1}; b_h)\}$ for $1 \leq h \leq r$, take $\delta(X)$ to be the absolute value of $ \sum_{1}^r b_h$.
    
	If $p \neq m$, we say $X$ is a nice reflection set of $G(m, p, n)$ if $X$ contains exactly one reflection of type~\ref{type2}. Note that in \cite{wang2010}, reflections of type~\ref{type2} are represented in the graph $\Gamma_X$ as ``rooted'' nodes, but as these do not affect the parameter $\delta(X)$, we do not need this terminology for our purposes.  
	
	\begin{remark}
		Note that $s(a_{h}, a_{h+1}; b_h) = s(a_{h+1}, a_{h}; -b_h)$, so if $X$ is such that the graph contains only one cycle of length $r$, then the cycle corresponds to some permutation $(a_1a_2 \cdots a_r)$ (in cycle notation) and the signs of $b_h$ must be chosen accordingly when computing $\delta(X)$. In particular, if the cycle is of length~2, \textit{i.e.}~given as a double edge between two nodes corresponding to reflections $s(a_1, a_2; b_1)$ and $s(a_1, a_2; b_2)$, then $\delta(X) = |b_1 -b_2|$.
	\end{remark}

	\begin{theorem}\label{reflsubgp}
		With the notation as defined above, we have the following relations for the reflection subgroup generated by $X$:  
		\begin{enumerate}
			\item Let $X$ be reflection set of\, $G(m, m, n)$ such that the graph $\Gamma_X$ is connected and contains exactly one cycle. Then $\inn{X} = G(m, m, n)$ if and only if the integer $\delta(X)$ is coprime to $m$.
			\item Let $X$ be a nice reflection set of\, $G(m, p, n)$ with $s(a_1; b)$ the one reflection of type~{\rm\ref{type2}}, such that $\Gamma_X$ is connected with $n_1$ nodes and exactly one cycle. Then $\inn{X} = G(\frac{m}{m_1}, \frac{\gcd(b, m)}{m_1}, n_1)$, where $m_1 = \gcd(b, m, \delta(X))$.
			\item Let $X$ be a nice reflection set of\, $G(m, p, n)$ with $s(a_1; b)$ the unique reflection of type~{\rm\ref{type2}} such that the graph $\Gamma_X$ is a connected tree with $n_1$ nodes. Then $\inn{X} = G(\frac{m}{\gcd(b, m)}, 1, n_1)$. 
		\end{enumerate}
	\end{theorem}

	\begin{proof}
		The first part is proven in \cite[Theorem 2.19]{shi2005mm}, and the second is proven in \cite[Theorem 3.10]{wang2010}. Note that in both these papers, the authors refer to cycles as circles. The third statement is proven in \cite[Theorem~3.9]{wang2010}, where the authors define rooted trees to capture the presence of type~\ref{type2} reflections. We refer to these paper for more details on the graph-theoretic interpretation of generating sets and their connection to reflection subgroups of the imprimitive reflection groups.
	\end{proof}

		\begin{remark}\label{imprim just}
		Among the imprimitive groups, we have to check for $n$-tuples and $(n+1)$-tuples in $G(m, p, n)$. By Theorem~\ref{GenSet}, the only possible groups that can contain nice $n$-tuples are the well-generated groups $G(m, 1, n)$ and $G(m, m, n)$. We prove there are no nice $n$-tuples in either of these groups when $n \geq 5$ in Section~\ref{larger n}. By results in \cite[Lemma 5.3.2]{LLL}, we know that $(n+1)$-tuples can only exist in $G(m, p, n)$ for $n < 5$. This work will fully analyze the existence of nice $(n+1)$-tuples for $2< n <5$.
	\end{remark}
	
		\subsection{Nice 3-tuples in $\boldsymbol{G(m, p, 3)}$}
	In this case we are seeking nice tuples $[r_1, r_2, r_3]$ from $G(m, p, 3)$ such that their product has a nontrivial eigenvalue of multiplicity 1. Note that since $\inn{r_1, r_2, r_3} = G(m, p, 3)$, we know that $p = 1$ or $p = m$ are the only possibilities as these are the only well-generated imprimitive groups (see Definition~\ref{rank+wg})

	\begin{lemma}\label{lem4.6}
		In $G = G(m, 1, 3)$,  for a nice tuple $[r_1, r_2, r_3]$ with $r_1, r_2$ of type~{\rm\ref{type1}} and $r_3$ of type~{\rm\ref{type2}} with nontrivial eigenvalue $\zeta_m'$ of order $m$, letting $A = r_1r_2r_3$, the eigenvalues of $A$ must be the distinct cube roots of\, $\zeta_m'$. 
	\end{lemma}

	\begin{proof}
		Say we have three reflections $$r_1r_2r_3 = A$$ for some $A \in G$ with Eigv(A) = $\{\lambda_1, \lambda_2, \lambda_3\}$ with some $\lambda_i \neq 1$. Note that since the $r_i$ generate $G$, exactly two of them are of type~\ref{type1} corresponding to distinct transpositions, and the third is of type~\ref{type2} with order $m$, say with nontrivial eigenvalue $\zeta_m'$. Using the homomorphism $\vphi\colon G \rightarrow S_3$, we know that $\vphi(A)$ is even, and since it cannot be the identity, it must be a $3$-cycle.
		Since $A$ is a generalized permutation matrix, we know $\tr(A) = 0$, and we get the following equalities: 
		\begin{align}
			\zeta_m' &= \det (r_1r_2r_3) = \det(A) = \lambda_1\lambda_2\lambda_3, \\
			0 &= \tr(A) = \lambda_1 + \lambda_2 + \lambda_3.
		\end{align}
		Since the $\lambda_i$ are all roots of unity, we know their sum being zero means the arguments differ by $\frac{2 \pi}{3}$.
		Further writing $A = [a_1, a_2, a_3 | (123)]$, without loss of generality, gives us the matrix 
		$$ A = \begin{pmatrix}
			0 & \zeta_m^{a_1} & 0 \\
			0 & 0 & \zeta_m^{a_2} \\
			\zeta_m^{a_3} & 0 & 0 
		\end{pmatrix}. $$
		Thus we can compute the characteristic polynomial as $\Char_A(x) = -(x^3 - \zeta_m^{a_1 + a_2 + a_3})$, and since the constant term of the characteristic polynomial is the determinant of $A$, we get that $\zeta_m' = \zeta_m^{a_1 + a_2 + a_3}$. Thus the $\lambda_i$ must be the distinct cube roots of $\zeta_m'$.
	\end{proof}

	\begin{theorem}
		There are nice 3-tuples in $G(m, 1, 3)$ for all $m >1$. 
	\end{theorem}

	\begin{proof}
		The reflections in the tuple $[s(2, 3; 1), s(1, 2; 1), s(3; 1)]$ will generate $G(m, 1, 3)$ by Theorem~\ref{reflsubgp}. By Lemma~\ref{lem4.6}, the product will have eigenvalues corresponding to the three distinct cube roots of $\zeta_m$, all of which cannot be 1.  
	\end{proof}

	\begin{theorem}
		In $G = G(m, m , 3)$, if we have $r_1r_2r_3 = A$ for reflections $r_i$ generating $G$, then the eigenvalues of~$A$ are $\{\lambda^{-2}, \lambda, -\lambda \}$ for $\lambda$ a $\supth{(2m)}$ root of unity.
	\end{theorem}

	\begin{proof}
		Since all the $r_i$ are of type~\ref{type1}, using $\vphi\colon G \rightarrow S_3$, we see that $\vphi(A)$ must be a transposition, say $(ij) \in S_3$. Writing $A = [a_1, a_2, a_3 \mid (ij)]$, we see that $\tr(A) = \zeta_m^{a_k}$, where $k \neq i, j$, and $\det(A) = -1$. Moreover, we know that $\zeta_m^{a_k}$ is an eigenvalue of $A$ since the basis vector $e_k$ is an eigenvector. 
        Letting $\{ \lambda_i, \lambda_j, \zeta_m^{a_k}\}$ denote the eigenvalues of $A$, from the trace and determinant of $A$, we have
		\begin{align*}
			 \lambda_i + \lambda_j &= 0,  \\
		\lambda_i \lambda_j& = -\zeta_m^{-a_k}.
		\end{align*}
		This tells us that $\lambda_i = -\lambda_j$ and $\lambda_i^2 = \zeta_m^{-a_k}$, so $\lambda_i$ must be a $\supth{(2m)}$ root of unity (not necessarily primitive).
	\end{proof}

    \begin{remark}
        Note that in the proof above and throughout this work, the convention for the characteristic polynomial is $\det(A -xI)$.
    \end{remark}

    \begin{corollary}
		Nice $3$-tuples exist in $G(m, m, 3)$ for all $m >1$. 
    \end{corollary}

    \begin{proof}
		In this case we can consider reflections $[s(2, 3; b_1), s(1, 2; b_2), s(1, 2; b_3)]$ with $b_2$ and $b_3$ chosen so that $|b_2 -b_3|$ is coprime to $m$. Taking $b_1 = 1$, $b_2 = 1$, and $b_3 = 0$, we get that $\delta(X) = |b_2 -b_3|$ is coprime to $m$, and so by Theorem~\ref{reflsubgp}, we know these reflections will generate $G(m, m, 3)$. The characteristic polynomial of the product $r_1r_2r_3 = A$ will in this case be $\Char_A(x) = -(x-\zeta_m)(x^2 -\zeta_m^{-1})$. Since the roots of the quadratic will be distinct, we know we will have at least two distinct eigenvalues, so the only worry is if the eigenvalues might be $\{\zeta_m, \zeta_m, 1\}$ as then there is no nontrivial eigenvalue of multiplicity 1. If 1 is a root of the quadratic, then we know $\zeta_m^{-1} = 1$ and $\zeta_m = 1$, which is not possible. Hence, such nice tuples always exist. 
	\end{proof}
	
	\subsection{Nice 4-tuples in $\boldsymbol{G(m, p, 3)}$}
	We are seeking nice tuples of four reflections from $G(m, p, 3)$ such that their product has an nontrivial eigenvalue $\lambda$ of multiplicity 2. This can be written as the identity 
	\begin{align} \label{4-tup imprim 3}
		r_1r_2r_3r_4 = \lambda r_{5}^{-1},  
	\end{align}
	where the $r_i$ generate $G(m, p, 3)$ and $r_5$ is some reflection, not necessarily in $G(m, p, 3)$. 
	
	\begin{lemma}\label{eqn 5 thm}
		In the notation of Equation~\eqref{4-tup imprim 3} above, for any such nice tuple, it must be that $r_5^{-1}$ is in $G(m, 1, 3)$ and $\lambda$ is an $\supth{m}$ root of unity. Moreover, in any $G(m, p,3 )$ $($including $p = 1$ and $p= m)$,  it must be that exactly four of these reflections are of type~{\rm\ref{type1}} and only one is of type~{\rm\ref{type2}}. 
	\end{lemma}

	\begin{proof}
		Note that since $\lambda$ is an eigenvalue of multiplicity 2 for some matrix in $G(m, 1, 3)$, it must be that the order of $\lambda$ divides at least two degrees in $\deg G(m, 1, 3) = \{m, 2m , 3m\}$. Thus the order of $\lambda$ divides $m$, and so $\lambda I $ is in the center of $G(m, 1, 3)$. This tells us that $r_5$ is always in $G(m, 1, 3)$. To generate any of $G(m, p, 3)$, we know the set $\{r_i\}_{i \leq 5}$ must contain at least two reflections of type~\ref{type1} by Theorem~\ref{GenSet}.  For any $p$, rearranging the equation and using the homomorphism $\vphi\colon G(m, 1, 3) \rightarrow S_3$, we get that $\vphi(r_1r_2r_3r_4r_5) = \vphi(\lambda I) = e$. Since $e$ is even, we know there must be exactly two or four reflections of type~\ref{type1}. The product of two transpositions can only be the identity if they are the same transposition, and we need at least two distinct transpositions to generate a transitive subgroup of $S_3$; hence there are exactly four reflections of type~\ref{type1} and one reflection of type~\ref{type2} for any $G(m, p, 3).$  
	\end{proof}
        
	Note that this does not prove that $\lambda$ is a \textit{primitive} $\supth{m}$ root; however, in the next three theorems, we will conclude this is indeed always the case. 
	Up to rearranging  Equation~\eqref{4-tup imprim 3}, we may assume $r_5^{-1}$ is the unique reflection of type~\ref{type2} and label its nontrivial eigenvalue $\zeta$, which we know must be some $\supth{m}$ root of unity, not necessarily primitive. 
	
	\begin{remark} \label{trans 4 fac of e}
		To study the tuples in the cases where they might exist, we first study the number of ways to factorize the identity in $S_3$ into a transitive product of four transpositions. In $S_3$ any two distinct transpositions will generate a transitive subgroup which must be all of $S_3$, so fixing the first transposition as $(12)$, we only need to find, up to simultaneous conjugacy, tuples of three transpositions that multiply to $(12)$. Using computer search to find such conjugacy classes of 3-tuples, then adjoining with $(12)$ on the left yields four options: 
		\begin{enumerate}
			\item $[(12), (12), (23), (23)]$,
			\item $[(12), (23), (12), (13)]$,
			\item $[(12), (23), (23), (12)]$,
			\item $[(12), (13), (23), (13)]$.
		\end{enumerate}
		Note that these tuples are multiplied starting at the leftmost transposition, although that does not matter in this case. 
		Converting these tuples to type~\ref{type1} reflections, we can compute the product for each option: 
		\begin{enumerate}
			\item\label{fact1} $s(1, 2; b_1)s(1, 2; b_2)s(2, 3; b_3)s(2, 3; b_4) = \begin{pmatrix}
				\zeta_{m}^{b_1 -b_2} & 0 & 0 \\
				0  & \zeta_m^{-b_1 +b_2 +b_3-b_4} & 0 \\
				0 & 0 & \zeta_m^{-b_3 + b_4} 
			\end{pmatrix}$,
			\item\label{fact2} $s(1, 2; b_1)s(2, 3; b_2)s(1, 2; b_3)s(1, 3; b_4) = \begin{pmatrix}
				\zeta_{m}^{b_1 +b_2 -b_4} & 0 & 0 \\
				0  & \zeta_m^{-b_1 +b_3} & 0 \\
				0 & 0 & \zeta_m^{-b_2 -b_3+ b_4} 
			\end{pmatrix}$,
			\item\label{fact3} $s(1, 2; b_1)s(2, 3; b_2)s(2, 3; b_3)s(1, 2; b_4) = \begin{pmatrix}
				\zeta_{m}^{b_1 +b_2 -b_3-b_4} & 0 & 0 \\
				0  & \zeta_m^{-b_1 +b_4} & 0 \\
				0 & 0 & \zeta_m^{-b_2 + b_3} 
			\end{pmatrix}$,
			\item\label{fact4} $s(1, 2; b_1)s(1, 3; b_2)s(2, 3; b_3)s(1, 3; b_4) = \begin{pmatrix}
				\zeta_{m}^{b_1 +b_3 -b_4} & 0 & 0 \\
				0  & \zeta_m^{-b_1 +b_2 -b_3} & 0 \\
				0 & 0 & \zeta_m^{-b_2 + b_4} 
			\end{pmatrix}$.
		\end{enumerate}
		Equations for $\{b_i\}$ can be found by setting each of these matrices equal to $\lambda r_5^{-1}$:
		\begin{align*} \begin{pmatrix}
				\lambda & 0 & 0 \\
				0 & \lambda & 0 \\
				0 & 0 & \lambda \zeta
			\end{pmatrix}
		\end{align*}
		for some root $\zeta$ with order dividing $m$. The choice of putting $\zeta$ in the third row is arbitrary; $\zeta$ could have been placed in any diagonal entry.
	\end{remark}

	\begin{lemma} \label{delX lemma}
		For any nice $4$-tuple in $G(m, p, 3)$, $p\neq 1$, it must be that either $\lambda$ or $\lambda \zeta$ in the notation of Remark~\ref{trans 4 fac of e} has the same order as $\zeta_m^{\delta(X)}$, where $X$ is any generating reflection set given by the nice $4$-tuple.
	\end{lemma}

	\begin{proof}
		By Lemma~\ref{eqn 5 thm}, given any nice 4-tuple in $G(m, p, 3)$ satisfying Equation \eqref{4-tup imprim 3}, we can rearrange all the type~\ref{type1} reflections to be one of the factorizations given in Remark~\ref{trans 4 fac of e} up to simultaneous conjugation. As any reflection set $X$ generating $G(m, p, 3)$, $p \neq 1$, will only need at least three type~\ref{type1} reflections, we see that for each of the four factorizations above, there are multiple ways to choose $X$ so that the associated graph only contains one cycle. For example, the first factorization gives us the possibilities: $X = \{s(1, 2; b_1), s(1, 2; b_2), s(2, 3; b_i)\}$ for $i = 3, 4$ and $X = \{s(1, 2; b_i), s(2, 3; b_3), s(2, 3; b_4)\}$ for $i = 1, 2$. So we know $\delta(X)$ must be of the form
		$$ \delta(X) = \begin{cases}
			|b_1 -b_2|, \\
			|b_3 -b_4|.
		\end{cases}$$
		By the equations each factorization must satisfy, we know $\zeta_m^{b_1 -b_2} = \lambda$ and $\zeta_m^{-b_3 +b_4} = \lambda \zeta$, so one of these must have the same order as $\zeta_m^{\delta(X)}$. The argument is similar for the third factorization since the formula for $\delta(X)$ only differs in the indices. The possible $\delta(X)$ values from the second and fourth factorizations only differ in the indices, so we only justify the result for the second factorization. In this case the options for $\delta(X)$ are 
		$$ \delta(X) = \begin{cases}
			|b_1 +b_2 -b_4|, \\
			|b_1 -b_3|, \\
			|b_2 +b_3 -b_4|.
		\end{cases}$$
		Up to taking absolute value, the first two options must equal $\lambda$, and the third must be $\lambda \zeta$, so one of these has the same order as $\zeta_m^{\delta(X)}$.
	\end{proof}

	\begin{theorem}
	  Nice $4$-tuples exist in $G(m, m, 3)$ if and only if $m \neq 3$. For $m \neq 3$, we have the following:
		\begin{enumerate}
		\item If\, $\gcd(m, 3) = 1$, then $\Order(\lambda) = \Order(\zeta) = m$.
			\item If\, $\gcd(m ,3) = 3$, then $\Order(\lambda) = m = 3 \Order(\zeta)$ with $\Order(\zeta) \neq 1$. 
		\end{enumerate} 
	\end{theorem}

	\begin{proof}
		To generate $G(m, m, 3)$, we know that our nice 4-tuple must contain a generating reflection set $X$ such that $\delta(X)$ is coprime to $m$, which means either $\lambda$ or $\lambda \zeta$ has order $m$, by Lemma~\ref{delX lemma}. If $\lambda$ does not have order $m$, then $\lambda \zeta$ has order $m$, but taking the determinant of Equation~\eqref{4-tup imprim 3}, we know that $1 = \lambda^3 \zeta$, so~$\lambda^2$ has order $m$, which means $\lambda$ has order $m$. Thus, $\lambda$ is always a primitive $\supth{m}$ root. Moreover, we know $m \neq 3$ since the determinant tells us $\lambda^3 = \zeta^{-1}$ and $\zeta \neq 1$ as it is the nontrivial eigenvalue of the type~\ref{type2} reflection~$r_5^{-1}$.
		
		If $\gcd(m, 3) = 1$, then we know $\lambda^3 = \zeta^{-1}$ is also a primitive $\supth{m}$ root, and hence so is $\zeta$. Using factorization~\eqref{fact1},  we can take the tuple $[s(1, 2; 1), s(1, 2; 0), s(2, 3; 2), s(2, 3; 0)]$, so $\lambda = \zeta_m$ and $\zeta = \zeta_m^{-3}$. This tuple generates $G(m, m, 3)$  by Theorem~\ref{reflsubgp} since taking $X = \{s(1, 2; 1), s(1, 2; 0), s(2, 3; 2)\}$ gives $\delta(X) = 1$, which is coprime to $m$.  
		
		If $\gcd(m, 3) = 3$, then $\lambda^3$ has order $m/3$ and thus so too will $\zeta$. Again taking the tuple as above will work, since $\delta(X) = 1$ will still be coprime to $m$ and the product will give $\lambda = \zeta_m$ and $\zeta = \zeta_m^{-3}$, which now is not a primitive $\supth{m}$ root.  
	\end{proof}
	
		\begin{theorem}
		Nice $4$-tuples exist in $G(m, 1, 3)$ if and only if\, $\gcd(m, 3) = 1$. In this case the order of\, $\lambda$ is $m$.
	        \end{theorem}
                
	\begin{proof}
		Let $[r_2, r_3, r_4, r_5]$ be a nice 4-tuple in $G(m, 1, 3)$. Let $r_5$ be the unique reflection of type~\ref{type2} with nontrivial eigenvalue $\zeta^{-1}$ of order $m$, so we can write $\zeta^{-1} = \zeta_m^{-b}$ with $(-b, m) = 1$. Let $d = \Order(\lambda)$. We know $d\mid m$ since $\lambda I$ is in the center $ZG(m, 1, 3)$. Taking determinants, we get the equation
		$$\lambda^3 = \zeta^{-1},$$
		so $m \mid d$ since $(\zeta^{-1})^d = 1$;  hence $d = m$. Since $\zeta^{-1}$ is a primitive $\supth{m}$ root, so too is $\lambda^3$, and therefore we know $\gcd(m, 3) = 1$.
		
		Let $m$ be such that $\gcd(m, 3) = 1$, so we know $m-3$ is coprime to $m$. Taking $b = m-3$, we can consider the tuple $[s(1, 2; 0), s(2, 3; 1), s(2, 3; -1), s(3; m-3)]$. Taking $X = \{s(1, 2; b_2), s(2, 3; b_3), s(3; -b)\}$, we see by Theorem~\ref{reflsubgp} that $\inn{X} = G(m, 1, 3)$ since $\gcd(m-3, m) = 1$.
		Taking the product, we get $r_2r_3r_4r_5 = \lambda r_1^{-1}$ for some type~\ref{type1} reflection $r_1 = s(1, 2; -1)$ with $\lambda = \zeta_m$.
	\end{proof}

	\begin{theorem}
		Suppose $p \neq 1$ and $p\neq m$. Then there exists a nice $4$-tuple in $G(m, p, 3)$ if and only if $p = 3$. In this case $\lambda$ must be of order $m$.
	\end{theorem}
	\begin{proof}
		Let $[r_2, r_3, r_4, r_5]$ be a nice 4-tuple in $G(m, p, 3)$. Letting $r_5$ denote the unique reflection of type~\ref{type2}, we know it must have order $q = p^{-1}m$, so we can denote its nontrivial eigenvalue as $\zeta = \zeta_q^b = \zeta_m^{pb}$ with $\gcd(b, m) = 1$. 
		Further, we have the product $r_2r_3r_4r_5 = \lambda r_1^{-1}$ with $r_1 \in G(m, m, 3) \leq G(m, p, 3)$; hence $\lambda I$ is in the center $ZG(m,p , 3)$. Letting $d = \Order(\lambda)$, we see that since $|ZG(m, p, 3)| = q \gcd(p, 3)$, we must have that $d \mid q \gcd(p, 3)$. Taking determinants, we get $\lambda^3 = \zeta^{-1}$ and hence $q \mid d$ and $d \mid 3q$.  If $d \nmid q$, then since $d \neq 3$ (because $\zeta^{-1} \neq 1$) and $d \mid q \gcd(p, 3)$,  we must have that $\gcd(p, 3) = 3$.
		This gives us two cases: 
		\begin{enumerate}
			\item $d = q$, 
			\item $d = 3q$ and $\gcd(p, 3) = 3$.
		\end{enumerate}
		Recall that by Theorem~\ref{reflsubgp}, for a reflection set $X$ to generate $G(m, p, 3)$, we need $\gcd(pb, m) = p$ and $\gcd(pb, \delta(X), m) = 1$.  
		
		In the first case $d  = q$, so we can write $\lambda = \zeta_q^a = \zeta_m^pa$ and $\zeta = \zeta_q^b = \zeta_m^{pb}$ with $(b, q)= 1 = (a, q)$. By Lemma~\ref{delX lemma}, we see that $$\gcd(\delta(X), m) = \begin{cases}
			\gcd(pa, m), \\
			\gcd((a+b)p, m).
		\end{cases} $$
		Since neither of these satisfies $\gcd(pb, \delta(X), m) = 1$, it is not possible for a nice 4-tuple in $G(m, p, 3)$ to have $d = q$. 
		
		Now suppose that $p = 3k$, so $m = 3kq$. Again we can write $\zeta = \zeta_q^b = \zeta_m^{pb}$ with $\gcd(b, q) = 1$, and this time since $d= 3q$, we write $\lambda = \zeta_m^{ka}$ with $\gcd(a, m) = 1$. In this case the possible options for $\delta(X)$ give us
		$$ \gcd(\delta(X), m) = \begin{cases}
			\gcd(ka, m), \\
			\gcd(ka + pb, m).
		\end{cases} $$
		In either case we get $\gcd(pb, \delta(X), m) = k$, so we conclude $k = 1$. This also gives us that $\lambda$ must be of order $m$. 
		Thus, we conclude nice 4-tuples can only exist in $G(m, 3, 3)$. In this case we can write $\lambda = \zeta_m^a$ with $(a, m) = 1$, and we still have $\zeta = \zeta_m^{3b}$ with $\gcd(b, m) = 1$ for $q = p^{-1}m$. We can take our 4-tuple to be $[s(1, 2; 0), s(2, 3; 1), s(2, 3; -1), s(3; m-3)]$, noting that since $3 \mid m$ we also have that $\gcd(m-3, m) =3$. Here $\delta(X) = 2$, so $\gcd(m-3, 2, m) = 1$, and we conclude this tuple is indeed nice.  
	\end{proof}
	
	\begin{remark}
		Note that in all three cases described above, we will have to take middle convolution with parameter $\lambda^{-1}$ since our equations use the tuple product and not the inverse product.
	\end{remark}

	\subsection{Nice 4-tuples in $\boldsymbol{G(m, p, 4)}$} \label{4-tups in imprim4}
	
	We are seeking nice tuples satisfying $$r_1r_2r_3r_4 = A$$ with $r_i$ a reflection in $G(m, p,4)$ and $A$ has a nontrivial eigenvalue of multiplicity 2. 		Again the only possibilities are $p = 1$ and $p = m$ since the group must be generated by four reflections.
        
	\begin{theorem}
		In $G = G(m, 1, 4)$ there are no nice tuples of four reflections.
	\end{theorem}

        \begin{proof}
		Suppose we had $r_1r_2r_3r_4 = A$ such that the $r_i$ generate $G$ and $A$ has an eigenvalue $\lambda \neq 1 $ of multiplicity~2. Since the $r_i$ generate $G$, by Theorem~\ref{GenSet} we know exactly three of them are of type~\ref{type1} and exactly one is of type~\ref{type2} with nontrivial eigenvalue an $\supth{m}$ root of unity, say $\zeta_m'$. Note that since the three type~\ref{type1} reflections should correspond to transpositions that generate all of $S_4$, they must be three distinct transpositions with each of $\{1, 2, 3, 4\}$ showing up at least once. This means their product cannot be a transposition. Using $\vphi\colon G \rightarrow S_4$, we get that $\vphi(A) = \sigma$ for some 4-cycle $\sigma$ since the permutation must be odd and cannot be a transposition. 
		Writing $A = [a_1, a_2, a_3, a_4 \mid \sigma]$, we can compute that the characteristic polynomial of $A$ must be
                $$\Char_A(x) = x^4 - \zeta_m^{a_1+ a_2 + a_3 + a_4}.$$
		This gives us that all the eigenvalues of $A$ are the distinct $\supth{4}$ roots of $\zeta_m^{a_1 + a_2 + a_3 + a_4}$, which means $A$ cannot have an eigenvalue of multiplicity 2. 
	\end{proof}
	
	In the case of $G(m, m, 4)$, we know that for any nice 4-tuple,  all the reflections must be of type~\ref{type1}. This means they all map to transpositions in $S_4$, and hence the permutation associated to $A$ must be even. In $S_4$ the only options are $e$, the disjoint product of transpositions, and a $3$-cycle. We would need at least six transpositions to write $e$ as the transitive product of transpositions, so we only have two choices. Without loss of generality, if $A$ corresponds to disjoint transpositions, we take it to be $(12)(34)$, and if $A$ corresponds to a 3-cycle, we take it to be $(132)$.
	
	\begin{remark} \label{4-tup 2 trans}
		
		The number of transitive ordered factorizations of $(12)(34)$ into four transpositions in $S_4$ is 96; see \cite{transfac}. Using a computer, we find that there are 12 factorizations up to simultaneous conjugation. We can reduce these 12 even further by eliminating factorizations that are related by the action of one of the generators of $B_4$, since this means they lie in the same braid group orbit. 
		These preliminary computations give us three equivalence classes of factorizations (which may still lie in the same braid group orbit), which we write as follows with representatives: 
		\begin{enumerate}
			\item 	$[(12), (34), (13), (13)]$,
			\item 	$[(13), (13), (12), (34)]$,
			\item 	$[(13), (24), (23), (14)]$.
		\end{enumerate}
		Choosing general reflections in $G(m, m, 4)$ associated to these factorization and computing the products, we get
		\begin{enumerate}
			\item $s(1,2; b_1) \cdot s(3, 4; b_2) \cdot s(1, 3; b_3) \cdot s(1, 3; b_4) = \begin{pmatrix}
				0 & \zeta_m^{b_1} & 0 & 0 \\
				\zeta_m^{-b_1+b_3-b_4}& 0 & 0 & 0\\
				0 & 0 &0 &\zeta_m^{b_2}\\
				0 & 0 & \zeta_m^{-b_2-b_3+b_4} & 0
			\end{pmatrix}$,
			\item $s(1,3; b_1) \cdot s(1,3; b_2) \cdot s(1, 2; b_3) \cdot s(3, 4; b_4) = \begin{pmatrix}
				0&\zeta_m^{b_1-b_2+b_3} & 0&0 \\
				\zeta_m^{-b_3} & 0 & 0 & 0 \\
				0&0&0& \zeta_m^{-b_1+b_2+b_4} \\
				0&0&\zeta_m^{-b_4}& 0
			\end{pmatrix}$,
			\item $s(1, 3; b_1) \cdot s(2,4; b_2) \cdot s(2, 3; b_3) \cdot s(1, 4; b_4) = \begin{pmatrix}
				0& \zeta_m^{b_1-b_3} & 0 & 0 \\
				\zeta_m^{b_2-b_4} & 0&0&0 \\
				0&0&0&\zeta_m^{-b_1+b_4} \\
				0&0& \zeta_m^{-b_2 +b_3} &0
			\end{pmatrix}$.
		\end{enumerate}
	\end{remark}
	
	\begin{remark} \label{4-tup 3-cyc}
		Similarly, we can consider transitive ordered factorizations of $(132)$ into four transpositions in $S_4$. Up to simultaneous conjugation and the action of a generator of $B_4$, we again get three equivalence classes of factorizations (which may not be distinct): 
		\begin{enumerate}
			\item $[(14), (14), (12), (23)]$,
			\item $[(14), (14), (13), (12)]$,
			\item $[(14), (14), (23), (13)]$.
		\end{enumerate}
		These factorizations give us the following matrices: 
		\begin{enumerate}
			\item $s(1, 4; b_1) \cdot s(1, 4; b_2) \cdot s(1, 2; b_3) \cdot s(2, 3; b_4) = \begin{pmatrix}
				0&0&\zeta_m^{b_1-b_2+b_3+b_4}& 0 \\
				\zeta_m^{-b_3} &0 &0&0 \\
				0& \zeta_m^{-b_4} &0&0 \\
				0&0&0&\zeta_m^{-b_1 +b_2}
			\end{pmatrix}$, 
			\item $s(1, 4; b_1) \cdot s(1, 4; b_2) \cdot s(1, 3; b_3) \cdot s(1,2; b_4) = \begin{pmatrix}
				0&0&\zeta_m^{b_1 -b_2+b_3} &0 \\
				\zeta_m^{-b_4} &0&0&0 \\
				0& \zeta_m^{-b_3 +b_4} &0&0 \\
				0&0&0& \zeta_m^{-b_1+b_2}
			\end{pmatrix}$,
			\item $s(1, 4; b_1) \cdot s(1, 4; b_2) \cdot s(2, 3; b_3) \cdot s(1, 3; b_4) = \begin{pmatrix}
				0&0& \zeta_m^{b_1-b_2+b_4} &0 \\
				\zeta_m^{b_3 -b_4} &0 &0 &0 \\
				0& \zeta_m^{-b_3} & 0& 0 \\
				0&0&0& \zeta_m^{-b_1 +b_2}
			\end{pmatrix}$.
		\end{enumerate}
	\end{remark}
        
		\begin{lemma} \label{4-tup mm eigv}
			In $G = G(m, m, 4)$, if we have $r_1r_2r_3r_4 = A$ with the $r_i$ reflections generating $G$ and if $A$ has a nontrivial eigenvalue $\lambda$ of multiplicity $2$, then the eigenvalues of $A$ are one of the following: 
			\begin{enumerate}
				\item\label{case1} $\{\lambda, \lambda, -\lambda, -\lambda\}$ for $\lambda$ a $\supth{4}$ root of unity, 
				\item\label{case2} $\{\lambda, \lambda, \lambda_1, \lambda_2\}$ for $\lambda$ a $\supth{4}$ root of unity and $\lambda_i \neq \lambda$ the two distinct cube roots of $\lambda^{-1}$ other than $\lambda$.  
			\end{enumerate}
		\end{lemma}
                
		\begin{proof}
			We know all the $r_i$ are of type~\ref{type1}. Hence using $\vphi\colon G \rightarrow S_4$, we know that $\vphi(A)$ must be an even permutation, so it is either a double transposition or a 3-cycle. Let the eigenvalues of $A$ be $\{\lambda, \lambda, \lambda_1, \lambda_2\}$. 
			If $\vphi(A)$ is a double transposition, then we see that 
			\begin{align*}
				0 &= \tr(A) = 2\lambda + \lambda_1 + \lambda_2, \\
				1 &= \det(A) = \lambda^2 \lambda_1 \lambda_2. 
			\end{align*}
			Rearranging the equation for trace gives us that $$1 = |\lambda| = \frac{| \lambda_1 + \lambda_2|}{2} = \frac{|\lambda_1| + |\lambda_2|}{2} = 1.$$
			By properties of the triangle inequality, we get that $\lambda_1$ and $\lambda_2$ are colinear; hence $\lambda_1 = \pm \lambda_2$. However, if $\lambda_1 = -\lambda_2$, then the trace equation means $\lambda = 0$, which is not possible; hence $\lambda_1 = \lambda_2$. Now the trace equation tells us that $2(\lambda + \lambda_1) = 0$; hence $\lambda_1 = -\lambda$. 
			Next, we can use the determinant equation to see that $1 = \lambda^4$, and so $\lambda$ must be a $\supth{4}$ root of unity. This gives us that the eigenvalues of $A$ are as described in case~\eqref{case1}. 
			
			Further, when $\vphi(A)$ is a double transposition, say $A = [a_1, a_2, a_3, a_4 \mid (12)(34)]$, the characteristic polynomial is $\Char_A(x) = (x^2 -\zeta_m^{a_1 + a_2})(x^2 - \zeta_m^{a_3 + a_4})$. However, since we know the roots of these polynomials agree, we see that $\zeta_m^{a_1 + a_2} = \zeta_m^{a_3 + a_4}$, but we already knew that $\zeta_m^{a_1 + a_2} = \zeta_m^{-a_3 - a_4}$ since $m \mid a_1 + a_2 + a_3 + a_4$. Thus we conclude that $\zeta_m^{a_1 + a_2} = \pm 1$ and so $\zeta_m^{a_1} = \pm \zeta_m^{-a_2}$ and, similarly, $\zeta_m^{a_3} = \pm \zeta_m^{-a_4}$, where the signs must agree for both pairs. 

			If $\vphi(A)$ is a 3-cycle, we write $A = [a_1, a_2, a_3, a_4 \mid (132)]$, since taking a different 3-cycle just changes the indices. The determinant condition stays the same as above; however, our equation for the trace is now 
			$$ \zeta_m^{a_4} = \tr(A) = 2\lambda + \lambda_1 + \lambda_2.$$
			Computing the characteristic polynomial, we get $\Char_A(x) = (x-\zeta_m^{a_4})(x^3 - \zeta_m^{a_1+a_2+a_3})$, and since the cube roots of $\zeta_m^{a_1 + a_2 + a_3}$ are all distinct, we know that for $\lambda$ to have multiplicity~2, it must be one of the cube roots and $\lambda = \zeta_m^{a_4}$. Since $\zeta_m^{a_1 + a_2 + a_3} = \zeta_m^{-a_4} = \lambda^{-1}$, we see that $\lambda$, $\lambda_1$, and $\lambda_2$ must be the three distinct cube roots of $\lambda^{-1}$. Now we have that $\lambda^3 = \lambda^{-1}$ and hence $\lambda^4 = 1$, proving that the eigenvalues of $A$ are as described in case~\eqref{case2}.
		\end{proof}
	
	\begin{theorem} \label{4-tup mm}
		Nice $4$-tuples exist in $G(m, m, 4)$ if and only if  $m = 2$ or $m = 4$. In either case $\lambda$ is of order $m$. 
	\end{theorem}
	\begin{proof}
		Say we had a nice tuple in $G(m, m, 4)$ such that $r_1r_2r_3r_4 = A$. 
		 First suppose  that $A$ corresponds to a 3-cycle; then our nice 4-tuple is equivalent to one of the factorizations in Remark~\ref{4-tup 3-cyc}.  For all three factorizations, by Lemma~\ref{4-tup mm eigv} we know that $\zeta_m^{-b_1 +b_2} = \lambda$, and by Theorem~\ref{reflsubgp} we know that $\delta(X) = |b_1 - b_2|$ must be coprime to $m$. This means $\lambda$ is a primitive $\supth{m}$ root, and by Lemma~\ref{4-tup mm eigv} we know $m = 4$ or $m = 2$ in this case. 
		
		If $A$ corresponds to a double transposition, then our nice 4-tuple is equivalent to one of the factorizations in Remark~\ref{4-tup 2 trans} and we may write $A = [a_1, a_2, a_3, a_4 \mid (12)(34)]$. From the first factorization and the proof of Lemma~\ref{4-tup mm eigv}, we see that $\pm 1 = \zeta_m^{a_1 + a_2} = \zeta_m^{b_3 -b_4} = \lambda^2$. In this case $\delta(X) = |b_3 -b_4|$, which we know must be coprime to $m$ by Theorem~\ref{reflsubgp}; we get that $\lambda^2$ is a primitive $\supth{m}$ root, so $m = 2$. The same argument will also give that $m$ must be 2 if our nice tuple is equivalent to the second factorization in Remark~\ref{4-tup 2 trans}. For the third possible factorization, we see that the cycle has four edges, so $\delta(X) = |b_1 -b_3 +b_2 -b_4|$ since the cycle corresponds to the permutation $(1324)$. But again we have that $\pm 1 = \zeta_m^{a_1 + a_2} = \zeta_m^{b_1 + -b_3 + b_2 -b_4}$, so $\lambda^2$ is a primitive $\supth{m}$ root, so $m = 2$. 
		If $m = 4$, then we can take $r_1 = s(1, 4; -1)$, $r_2 = s(1, 4; 0)$, $r_3 = s(1, 2; -1)$, $r_4 = s(2, 3; 0)$. Then $\delta(X) = 1$, so these reflections generate $G(4, 4, 4)$. Further, we have
                $$A = \begin{pmatrix}
			0 & 0 & -1 & 0 \\
			i & 0 & 0 & 0 \\
			0 & 1 & 0 & 0 \\
			0 & 0 & 0 & i
		\end{pmatrix}, $$
                which has eigenvalue $i$ with multiplicity 2.	
		If $m = 2$, then we can take $r_1 = s(1, 4; 1)$, $r_2 = s(1, 4; 0)$, $r_3 = s(1, 2; 0)$, $r_4 = s(2, 3; 0)$, so again $\delta(X)$ equals $1$
                and generates $G(2, 2, 4)$. In this case
                $$ A = \begin{pmatrix}
			0 & 0 & -1 & 0 \\
			1 & 0 & 0 & 0 \\
			0 & 1 & 0 & 0 \\
			0 & 0 & 0 & -1
		\end{pmatrix}, $$
                which has eigenvalue $-1$ with multiplicity 2.
	\end{proof}
	Note that when taking the middle convolution for such tuples, the parameter should be $\lambda^{-1}$ since these equations are for the tuple product and not the inverse product.
	
		\subsection{Nice 5-tuples in $\boldsymbol{G(m, p, 4)}$}\label{5-tups in imprim4}
	Searching for 5-tuples $G(m, p, 4)$ so that the middle convolution will be of rank 2 means we want the product to have an eigenvalue of multiplicity 3. A $4 \times 4$ matrix with an eigenvalue of multiplicity 3 can be viewed as a scalar times a reflection, so we want to explore the existence of $[r_1, r_2, r_3, r_4, r_5]$ such that
	$$ r_1r_2r_3r_4r_5 = \lambda r_6^{-1},$$
        where $ \lambda \neq 1$ and $r_6$ is some reflection, not necessarily in $G(m, p, 4)$. 
	
	\begin{lemma}\label{lem4.23}
		If we have five reflections generating $G = G(m , p, 4)$ such that 
		\begin{align}
			r_1r_2r_3r_4r_5 = \lambda r_6^{-1},
		\end{align} 
		then it must be that
		\begin{enumerate}
			\item  $r_6 \in G$ and $G = G(m, m, 4)$ with $\Order(\lambda) \mid m$,  
			\item $\lambda^4 = 1$.
		\end{enumerate}
	\end{lemma}
	
	\begin{proof}
	  First note that  $G(m, m, 4) \leq G(m, p, 4) \leq G(m, 1, 4)$; hence $r_i \in G(m, 1, 4)$ for $i < 6$ and therefore $\lambda r_6^{-1} \in G(m, 1, 4)$. Since this matrix has an eigenvalue of multiplicity 3, we know $d = \Order(\lambda)$ divides three of the degrees in $\deg G(m, 1, 4) = \{m, 2m , 3m, 4m\}$. This tells us that $d \mid m$ as if it did not, then $d \mid 3m -2m = m$, which would give a contradiction. Since $\lambda$ has order dividing $m$, $\lambda I$ is in the center of $G(m, 1, 4)$, and thus we can conclude $r_6 \in G(m, 1, 4)$, though $r_6$ is not yet necessarily in $G(m, p, 4)$. This changes our problem into studying six reflections in $G(m, 1, 4)$ satisfying
		$$r_1r_2r_3r_4r_5r_6 = \lambda I.$$
		
		Next we use the homomorphism $\vphi\colon G(m, 1, 4) \rightarrow S_4$. By Theorem~\ref{imprim homo}, we know $ \vphi(r_1r_2r_3r_4r_5r_6) = \vphi(\lambda I) = e$, and since $e$ is even, we know an even number of $\vphi(r_i)$ must be transpositions.
		
		To conclude all six are transpositions, note that $\{r_i\}_{i < 6}$ are supposed to generate $G(m, p, 4)$, which means $\{\vphi(r_i)\}_{i<6}$ must generate a transitive subgroup of $S_4$, hence $S_4$ itself. This removes the possibility of just two transpositions since we need at least three distinct transpositions to generate $S_4$.
		
		While it is possible to generate $S_4$ with four transpositions, our equation also requires that their product is the identity, which is not possible with at least three distinct transpositions (as can be verified by hand). 
		
		Thus we know all six reflections must be of type~\ref{type1}, and so $\{r_i\}_{i < 6}$ can only generate $G(m , m, 4)$; we conclude $G = G(m , m ,4)$.
		
		Using that all reflections of type~\ref{type1} have nontrivial eigenvalue $-1$, taking the determinant gives us that
$$ 1 = (-1)^ 6 = \det(r_1r_2r_3r_4r_5r_6) =  \det(\lambda I) = \lambda^4, $$
so we know $\lambda$ must be a $\supth{4}$ root of unity. Further, using that $\Order(\lambda) \in \{2, 4 \}$ must divide at least three degrees in $\deg G(m, m , 4) = \{m, 2m, 3m, 4\}$, we conclude that $\Order(\lambda) \mid m$. 
	\end{proof}
	
	\begin{remark} \label{5-tup mm}
		To study such tuples in $G(m, m, 4)$, we first found the ways to decompose $e$ in $S_4$ into a transitive product of six transpositions up to conjugacy. To make such a list, note that we can assume the first two transpositions are not disjoint, as if they are we can swap them around. This tells us their product is a 3-cycle, and up to conjugation we may assume it is $(132)$. We then found all the 4-tuples of transpositions that multiply to $(123)$ up to simultaneous conjugation and generate all of $S_4$. Adjoining these options to the $(12)$, $(13)$ on the left, we get 27 options for transitive factorizations of the identity, and using computer code to partially compute the braid group orbit, we were able to reduce the list further to two options (not necessarily distinct). 
		The tuples in $S_4$ are
		\begin{enumerate}
			\item $[(12), (13), (12), (23), (24), (24)]$,
			\item $[(12), (13), (24), (24), (12), (23)]$.
		\end{enumerate}
		Using the type~\ref{type1} notation for the reflections in $G(m, m, 4)$ as $r_i =s(j_i, k_i; b_i)$, we can take the matrix product. The first tuple gives us the equation
		\begin{align*}
			\begin{pmatrix}
				\zeta_m^{b_1 -b_3} & 0 & 0 & 0 \\
				0 & \zeta_m^{-b_1 +b_2-b_4 +b_5-b_6} & 0 & 0 \\
				0 & 0 & \zeta_m^{-b_2 +b_3 +b_4} & 0 \\
				0 & 0 & 0 & \zeta_m^{-b_5 +b_6}
			\end{pmatrix} &= \begin{pmatrix}
				\lambda & 0 & 0 & 0 \\
				0 & \lambda & 0 & 0 \\
				0& 0 & \lambda & 0 \\
				0 & 0 & 0 &\lambda
			\end{pmatrix}, 
		\end{align*}
		and the second tuple gives us the equation 
		\begin{align*}
			\begin{pmatrix}
				\zeta_m^{b_1 +b_3 -b_4 -b_5} & 0 & 0 & 0 \\
				0 & \zeta_m^{-b_1 +b_2-b_6} & 0 & 0 \\
				0 & 0 & \zeta_m^{-b_2 +b_5 +b_6} & 0 \\
				0 & 0 & 0 & \zeta_m^{-b_3 +b_4}
			\end{pmatrix} &= \begin{pmatrix}
				\lambda & 0 & 0 & 0 \\
				0 & \lambda & 0 & 0 \\
				0& 0 & \lambda & 0 \\
				0 & 0 & 0 &\lambda
			\end{pmatrix}.
		\end{align*}
	\end{remark}

	\begin{theorem}
		Nice $5$-tuples exist in $G(m, m, 4)$ if and only if $m = 2$ or $m = 4$, and in either case $\lambda$ is of order~$m$. 
	\end{theorem}
	\begin{proof}
		Suppose we have a nice 5-tuple in $G(m,m, 4)$, so by Lemma~\ref{lem4.23} we know $2 \mid m$. We know it must be equivalent to one of the two factorizations in Remark~\ref{5-tup mm}. Considering the first factorization and looking at all transitive subsets of four reflections we can choose from the six, we get the following possibilities for~$\delta(X)$: 
		\begin{align*}
			\delta(X) = \begin{cases}
				|b_1 -b_3|, \\
				|b_6 -b_5|, \\
				|b_4 +b_3-b_2|, \\
				|b_4 +b_1 -b_2|.
			\end{cases}
		\end{align*}
		The first three of these are such that $\zeta_m^{\delta(X)}$ equals $\lambda$ up to the absolute value on $\delta(X)$, so if any of those are coprime to $m$, then $\lambda$ is of order $m$. So the only way $m \neq 2, 4$ is if the only option for $\delta(X)$ that is coprime to $m$ is $|b_4 +b_1-b_2|$. 
		Writing $m = 2^sk$ with $(k, 2) =1$ and $s \geq 1$, we can write $\lambda = \zeta_m^{k (s-j)}$ for $j = 1, 2$ depending on the order of $\lambda$. The equations for $b_i$  modulo $m$ then becomes 
		\begin{align*}
			b_1-b_3 \equiv k2^{s-j}, \\
			-b_1 + b_2-b_4+b_5 -b_6 \equiv k2^{s-j}, \\
			-b_2 +b_3 +b_4 \equiv k2^{s-j}, \\
			-b_5 +b_6 \equiv k2^{s-j}.
		\end{align*}
		The second and fourth equations give us $-b_1 +b_2-b_4 \equiv 2 k2^{s-j}$, which is not coprime to $m$. Thus, this option for $\delta(X)$ does not work, so one of the first three options is coprime to $m$ and hence $\lambda$ is of order $m$. Since $\lambda$ is a $\supth{4}$ root of unity, we know $m = 2$ or $m = 4$. 
		
		A similar argument will give the same conclusion for the second factorization. 
		For $m = 4$, consider the tuple $[s(1, 2; 1), s(1, 3; 0), s(1, 2; 0), s(2,3; 1), s(2, 4; 1)]$. Then $\delta(X) = 1$, and so these reflections will generate $G(4, 4, 4)$. The product of this tuple will be $i \cdot s(2, 4; 2)$ as this corresponds to the first factorization.
		For $m = 2$, we can take the tuple $[s(1, 2; 1), s(1, 3; 0), s(1, 2; 0), s(2,3; 1), s(2, 4; 0)]$. Then $\delta(X) = 1$, and so these reflections will generate $G(2, 2, 4)$.  The product of this tuple will be $-1 \cdot s(2, 4; 1)$ as this corresponds to the first factorization.
	\end{proof}
		Note that with the examples above, we will have to take middle convolution with parameter $\lambda^{-1}$ since our equations are for the tuple product and not the inverse product.
	
	\subsection{Nice $\boldsymbol{n}$-tuples in $\boldsymbol{G(m, p, n)}$ for $\boldsymbol{n \geq 5}$} \label{larger n}
	Recall that the only possible groups when nice $n$-tuples can exist are when $p = 1$ and $p = m$ since these are the only well-generated imprimitive groups (see Definition~\ref{rank+wg}).

\begin{fact*}
 A connected graph of $n$ vertices and $n-1$ edges must be a tree.
\end{fact*}	
	
	\begin{claim} The product of\, $n-1$ transpositions is an $n$-cycle if and only if the graph associated to the set of\, $n-1$ transpositions in $S_n$ is a tree, where the edge $(a,b)$ is present if and only if\, $(ab)$ is one of the $n-1$ transpositions. 
	\end{claim}
	\begin{proof}

        We first check the ``only if'' direction: if the product of $n-1$ transpositions is an $n$-cycle, the associated graph must be connected; since it has $n$ vertices and $n-1$ edges, the fact above implies that it is a tree. It remains to prove the ``if'' direction.
        
        We proceed by induction. Assume the result holds for $k < n$. Given a   tree with $n$ vertices, there must exist some vertex $b$ of degree~1. When we remove this vertex $b$ and its only edge $(a_m, b)$, the result is still a tree, so by the induction hypothesis,  the product of the remaining $n-2$ transpositions is an $(n-1)$-cycle $(a_1 \cdots a_m \cdots a_{n-1})$ with $a_i \neq b$ for all $i$. This tells us that we have $(a_1 \cdots a_m \cdots a_{n-1})(a_m b) = (a_1 \cdots a_m b a_{m+1} \cdots a_{n-1})$, which is an $n$-cycle.
	\end{proof}
	
	\begin{theorem}
		There are no nice $n$-tuples in $G(m, 1, n)$ for $n \geq 5$. 
	\end{theorem}
	
	\begin{proof}
	  In $G(m, 1, n)$ any nice $n$-tuple consists of $n-1$ reflections of type~\ref{type1} and one reflection of type~\ref{type2}, say $r_n$. Suppose we have $r_1r_2\cdots r_n = A$ with $[r_1, \dots, r_n]$ a nice tuple. Then by the claim(s) above, the graph associated to the transpositions $\{\vphi(r_i) \}$, with $\vphi(r_n) = e$, must be a tree since the graph is connected by transitivity and therefore $\vphi(A)$ is an $n$-cycle. Thus the characteristic polynomial satisfies 
          $$\Char_A(x) = (-1)^n (x^n - \zeta)$$
          for some $\supth{m}$ root of unity $\zeta$ and therefore has $n$ distinct roots, so no $A$ has no eigenvalues of multiplicity $n-2$ for $n > 3$.
	\end{proof}

        For $G(m, m, n)$ we need the additional observations provided below.
	
	\begin{claim}\label{claim4.29}
	  Let $\sigma = (x_1x_2 \cdots x_n)$ be an $n$-cycle in $S_n$ and $\tau = (y_1y_2)$ be some transposition in $S_n$ with $y_1 <y_2$. Letting $c$ denote the distance between $y_1$ and $y_2$ in $\sigma$ modulo the length of $\sigma$, we have that the product $\sigma \tau$ is the product of a $c$-cycle and an $(n-c)$-cycle which are disjoint. 
	\end{claim}
	\begin{proof}
	  Since $\sigma$ fixes no points, it is enough to show that multiplying with $\tau$ cannot fix two points in $\{1, \dots, n\}$. Letting $y_1 = x_i$ and $y_2 = x_j$ with $i <j$, we then get
          $$y_1  \longmapsto x_{j+1} \longmapsto x_{j+2} \longmapsto \dots \longmapsto x_{i-1}  \longmapsto x_i = y_1,$$
          where the subscripts are read modulo $n$. The length of this cycle is precisely $j-i \mod n$. Doing a similar process for $y_2$, we get a cycle of length $n -(j-i) \mod n$. 
	\end{proof}
        
	Note that since $\sigma$ is written in disjoint cycle notation, the ``distance'', \textit{i.e.}~number of elements, between $y_1$ and $y_2$ read in the order of multiplication (\textit{e.g.}~left to right if multiplication is left to right) is well-defined.
        
	\begin{corollary}
		In the notation of Claim~\ref{claim4.29}, if $n > 2$, then the product $\sigma \tau$ has at most one fixed point, with this occurring only if $y_1$ and $y_2$ are adjacent in the cycle $\sigma$. 
	\end{corollary}
	\begin{proof}
		If $y_1$ and $y_2$ are adjacent, then $c  = 1$, so we are left with just an $(n-1)$-cycle which has exactly one fixed point when $n-1 > 1$ and two fixed points when $ n = 2$. 
	\end{proof}
	
	\begin{theorem}
		There are no nice $n$-tuples in $G(m, m, n)$ for $n \geq 5$.
	\end{theorem}
	\begin{proof}
		In $G(m, m, n)$ any nice tuple $[r_1, r_2, \dots, r_n]$ consists of $n$ reflections of type~\ref{type1}. Again by transitivity, we know the graph associated to $\{ \vphi(r_i)\}$ is connected, and since we only need $n-1$ edges to make a connected graph on $n$ vertices, we know there is some transposition we can remove so that the associated graph is still connected. By cyclic permutation of the $r_i$ via simultaneous conjugation, we can assume the redundant transposition corresponds to $r_n$, so $\vphi(A \cdot r_n^{-1})$ is an $n$-cycle by  Claim~\ref{claim4.29}. This tells us that $\vphi(A)$ is the product of an $n$-cycle and a transposition. 
		By the proofs above, we know $\vphi(A)$ is either an $(n-1)$-cycle or the product of a $c$-cycle and an $(n-c)$-cycle. So  the characteristic polynomial is either 
                $$\Char_A(x) = (-1)^n(x^{n-1}-\zeta_1)(x-\zeta_2)$$ or
                $$\Char_A(x) = (-1)^n (x^c -\zeta_1)(x^{n-c}-\zeta_2).$$
		We see that in both cases an eigenvalue of $A$ can have at most multiplicity 2, so when $n > 4$, there are no eigenvalues of multiplicity $n-2$, hence no such nice tuples can exist in $G(m, m, n)$. 
	\end{proof}
	
	\section{Primitive reflection groups} \label{prim}
	
	For each rank $r \in \{3, 4\}$, we investigate the existence of nice $T$-tuples for $T = r$ and $T = r+1$ in the primitive reflection groups, denoted $G_{i}$ for $i = 23$, $24$, $\dots$, $32$. 
	For each group and each choice of $T$, we provide a table of the different types (up to inverse pairs), including the label for the type, the residues of the eigenvalues, as well as the nontrivial eigenvalues of the exemplar tuple under ``Tuple eigenvalues''.
	We also list the chosen exemplars $\Abf = [A_1, A_2, \dots, A_T]$ for each type. For each exemplar we compute the middle convolution with the suitable choice(s) of parameter $\lambda$; this data is available in Appendices~\ref{appendix A},~\ref{appendix B},~\ref{appendix C}. These tables also include, for each choice of eigenvalue $\lambda$, the residue of $\lambda$ in the ``$\xi$'' column,  as well as the notation used in Magma to identify the root of unity under ``$\lambda$''.  We list the $2 \times 2$ matrices composing $MC(\Abf) = [M_1, \dots, M_{T}]$ and record the order of the $\GL_2(\C)$ subgroup generated by these matrices under ``S.~size''. Note that a zero in this column means the subgroup is infinite. 
	
	To get the induced tuple $\mathbf{\hat{A}}$ from $MC_{\lambda}(\Abf)$, many choices of character will work, so we provide the one we use under the column ``Character'', sometimes abbreviated as ``Char.''. 
	In some cases we may need to go to a larger field to construct a suitable character. We then compute and record the braid group orbit size of the induced tuple (under the action of the positive powers of the braid group generators) and record it under ``O.~size''. More information about how we compute the braid group orbit is provided in Section~\ref{BGSection}.
	For examples where the braid group orbits are the same size, we check if the signature of one lies in the orbit of the other.
        
	\begin{definition}\label{equivorb}
		We say two braid group orbits are \emph{equivalent} if the signature of one of the induced tuples lies in the orbit of another induced tuple.
	\end{definition}  

        This is one reason why our braid group code optionally can return the list of signatures in each orbit. 
	Note that this differs from the broader notion of equivalence in other literature such as Tykhyy's classification \cite{Tyk}. Note that the work of Lisovyy--Tykhyy \cite{T-L} and Tykhyy \cite{Tyk} only consider the case where matrices in the tuple do not have a shared eigenvector. In our case since nice tuples give an irreducible representation and middle convolution preserves irreducibility, we know all our tuples cannot share a common eigenvector. Thus, we expect to identify which tuples in these classifications come from the complex reflection groups via middle convolution. Details on what parameters we use to compare our orbits with these classifications is provided below at the start of Sections~\ref{prim3} and~\ref{prim4}.   
	
	For each distinct orbit we find, we provide data in the comparison tables on how our orbits relate to other classifications, as well as identify the finite $\SL_2(\C)$ subgroups generated by the tuples. We list the subgroups of $\SL_2(\C)$ in the Magma notation, given by $\inn{\Order, \Number}$, where $\Order$ is the order of the finite group and $\Number$ is the numbering of the group in the Small Groups Database. The numbering used by Magma is the same as that in the Small Groups Database implemented in GAP (see \cite[p.~779]{MagmaHandbook})
        and can be readily found online for all the groups we see. For infinite groups, the order is listed as ``0'', following the notation used in Magma. The notation $\SL(2, k)$ means $2\times 2$ matrices of determinant 1 over the finite field of~$\F_k$.
	
		\subsection{Nice 3-tuples in rank 3 groups} \label{prim3}
	Nice tuples in this case will be tuples of three reflections generating the whole group such that their inverse product has some nontrivial eigenvalue of multiplicity 1. 
	After middle convolution, the induced tuple will correspond to $\SL_2(\C)$ local systems on the 4-punctured sphere, and we can compare our braid group orbits to those in the literature. We computed the parameters $\omega_X, Y, Z, 4 = (\omega_X, \omega_Y, \omega_Z, 4-\omega_4)$ in the notation of \cite[Equations 11--12]{T-L} and found that the $\omega$-parameters for each of our orbits match (up to permutation) with one listed in the work of Lisovyy--Tykhyy. Since Tykhyy's classification \cite{Tyk} further refines the original classification by Lisovyy--Tykhyy \cite{T-L}, each tuple of $\omega$-parameters gives multiple orbits that Tykhyy identifies uniquely in terms of residues of the trace coordinates  $\theta_{1, 2, 3, 4} =(\theta_1, \theta_2, \theta_3, \theta_4)$ and $\sigma_{12, 23, 13, 24} = (\sigma_{12}, \sigma_{23}, \sigma_{13}, \sigma_{24})$ for some representative. In the nice 3-tuple case, using the $\omega$-parameters, we are able to easily identify how our orbit size relates to the orbit size listed by Tykhyy for each $\mathcal{M}^{(4)}$ (in Tykhyy's notation) and narrow down the possibilities for which of Tykhyy's orbits our tuple corresponds to. 
	In Tykhyy's classification,  tuples are considered equivalent up to  multiplying any two matrices by $-1$, complex conjugation of all elements of all matrices, cyclically permuting the matrices in the tuple, and up to taking inverses and reversing the order of the tuple; see \cite{Tyk}. For each of our orbits, we perform all these operations and compute the residue signatures based on Tykhyy's definition  (see \cite[Section 2, p.~122, Items~(1)--(3)]{Tyk})  so we can compare with the examples in Table 1 of Tykhyy's work. In most cases we find an exact match, with a few exceptions corresponding to Orbit 5 in Tykhyy's Table 1. For those tuples, we mark the differing coordinate with an asterisk (see Sections~\ref{G_{25} 3-tups} and~\ref{G_{26} 3-tups}), noting that they still agree with Tykhyy's list as the $\theta$- and $\sigma$-coordinates are defined modulo 2 and up to a sign in Tykhyy's work; see \cite[Section 2, p.~122, Item~(3)]{Tyk}.   
        
	Note that we only list this data for orbits that are considered distinct by our Definition~\ref{equivorb},  and these may end up being the same orbit in the Lisovyy--Tykhyy or Tykhyy classifications; see \cite{T-L,Tyk}. 
	For each orbit we compute, we provide the identifying label used in both classifications, as well as our orbit size and the orbit size they provided  (which is the same in both works).

	\subsubsection{$G_{23}$} \label{G_{23} 3-tups}
	The group $G_{23}$ of order 120 is also denoted $W(H_3)$ with degrees $\deg (G_{23}) = \{2, 6, 10\}$; see \cite[Table D.3]{ld2009}. All reflections in this group are conjugate with nontrivial eigenvalue $-1$. 
	For nice 3-tuples in $G_{23}$, the computation returned three possible sets of eigenvalues for the inverse product.  With $q = \zeta_5$, we provide one example of a nice tuple such that the inverse product has eigenvalues corresponding to each type.
        
	\begin{table}[ht!]
		\centering
		\begin{tabular}{|c|c|c|} 
			\hline 
			\multicolumn{3}{|c|}{$G_{23}$ types} \rule[0pt]{0pt}{3ex}\\ 
			\hline
			Type & Tuple eigenvalues & Residues of inverse product  \rule[0pt]{0pt}{3ex}\\ [0.5ex] 
			\hline
			A & $(-1, -1, -1)$ & $\{\frac{1}{2}, \frac{1}{10}, \frac{9}{10}\}$  \rule[0pt]{0pt}{3ex} \\ [0.5ex] 
			\hline 
			B & $(-1, -1, -1)$ & $\{ \frac{1}{2}, \frac{3}{10},  \frac{7}{10} \}$  \rule[0pt]{0pt}{3ex}\\ [0.5ex] \hline
			C & $(-1, -1, -1)$ & $\{ \frac{1}{2}, \frac{1}{6}, \frac{5}{6}\}$  \rule[0pt]{0pt}{3ex}\\ [0.5ex]\hline
		\end{tabular} \caption{\label{G_{23} Types}}
	\end{table}

	\begin{table}
		\centering
		\begin{tabular}{|c|c|c|c|c|c|c|c|c|} 
			\hline
			$\lambda$ & $\omega_{X, Y, Z, 4}$ & \begin{sideways} Lisovyy--Tykhyy $\#$ \end{sideways} & $\theta_{1, 2, 3, 4}$ & $\sigma_{12, 23, 13, 24}$   & \begin{sideways} Tykhyy $\#$  \end{sideways}  & \begin{sideways} Their orbit size \end{sideways} & \begin{sideways} Our orbit size  \end{sideways} & \begin{sideways} $\SL_2(\C)$ subgroup \vspace{0.5ex}  \end{sideways} \rule[0pt]{0pt}{4ex}\\ [0.5ex] 
			\hline 
			\multicolumn{9}{|c|}{$G_{23}$ type A comparisons} \rule[0pt]{0pt}{3ex}\\ 
			\hline
			-1 & \small{$(3 -\sqrt{5}, 3-\sqrt{5}, 3-\sqrt{5}, \frac{7\sqrt{5}-11}{2})$ }& $16$ & $(1, 1, 1, \frac{1}{5})$ & $(\frac{3}{5}, \frac{3}{5}, \frac{3}{5}, \frac{1}{3})$  & 60 &  10 & 40 & 0 \rule[0pt]{0pt}{4ex} \\ [1ex] 
			\hline 
			$-q^3$ & \small{$(3 -\sqrt{5}, 3-\sqrt{5}, 3-\sqrt{5}, \frac{7\sqrt{5}-11}{2})$} & 16 & $(\frac{3}{5}, \frac{2}{5}, \frac{3}{5}, \frac{2}{5})$ & $(\frac{2}{5}, \frac{2}{5}, \frac{1}{3} , \frac{3}{5})$ & 57  & 10 & 10 & $\inn{120, 5}$ \rule[0pt]{0pt}{4ex}\\ [1ex] \hline 
			\multicolumn{9}{|c|}{$G_{23}$ type B comparisons} \rule[0pt]{0pt}{3ex}\\ 
			\hline
			-1 & \small{$(3+\sqrt{5}, 3 + \sqrt{5}, 3 + \sqrt{5}, -\frac{7\sqrt{5}+11}{2})$} & 17 & $(1, 1, 1, \frac{3}{5})$ & $(0, \frac{1}{5}, \frac{1}{3},\frac{1}{3})$ & 58 & 10 & 40 & 0 \rule[0ex]{0ex}{4ex} \\ [1ex]
			\hline
			$-q^4$ & \small{$(3+\sqrt{5}, 3 + \sqrt{5}, 3 + \sqrt{5}, -\frac{7\sqrt{5}+11}{2})$} & 17 & $(\frac{4}{5}, \frac{4}{5}, \frac{4}{5}, \frac{4}{5})$ & $(\frac{1}{5}, 0, \frac{1}{3}, \frac{1}{3})$ & 59& 10 & 10 & $\inn{120, 5}$ \rule[0pt]{0pt}{4ex} \\ [1ex] \hline
			\multicolumn{9}{|c|}{$G_{23}$ type C comparisons} \rule[0pt]{0pt}{3ex}\\ 
			\hline
			-1 & $(2, 2,2, -1)$ & 31 & $(1, \frac{1}{3}, 0 , 0)$ & $(\frac{1}{5},1, \frac{2}{5}, \frac{2}{5})$ & 99 & 18 & 72 & 0 \rule[0ex]{0ex}{4ex} \\ [1ex] \hline
			$z^5$ & $(2, 2, 2, -1)$ & 31& $(\frac{2}{3}, \frac{2}{3}, \frac{2}{3}, \frac{2}{3})$ & $ (0, \frac{1}{5}, \frac{3}{5}, \frac{3}{5})$ & 96 & 18 & 18 & $\inn{120, 5}$  \rule[0ex]{0ex}{4ex} \\ [1ex] \hline
		\end{tabular} \caption{\label{G_{23} Nice 3-Tuple Comparisons}Here $q = \zeta_5$, $z = \zeta_{30}$, and all three subgroups of order 120 in $\SL_2(\C)$ are isomorphic to $\SL(2, 5)$.}
	\end{table}

	For type A in $G_{23}$, one exemplar is
	\begin{align*}
		A_1 &= \begin{pmatrix} 1 & 0 & 0 \\ 0 & 1 & 1 \\ 0 & 0 & -1
		\end{pmatrix}, \quad A_2 = \begin{pmatrix}
			0 & -1 & 0 \\ -1 & 0 & 0 \\ -q^3-q^2 + 1 & -q^3-q^2+1 & 1 
		\end{pmatrix},  \quad A_3 = \begin{pmatrix}
			1 & 0 & 0 \\ q^3+q^2 & -1 & -1 \\ 0 & 0 & -1
		\end{pmatrix}.
	\end{align*}
	Similarly, for type B, an exemplar is 
	\begin{align*}
		B_1 = \begin{pmatrix}1&0&0\\ 0&1&1\\ 0&0&-1  \end{pmatrix}, \quad B_2 =\begin{pmatrix}1&-q^3-q^2&-q^3-q^2\\ 0&0&-1\\ 0&-1&0\end{pmatrix}, \quad B_3 = \begin{pmatrix}1&0&0\\ q^3+q^2&-1&-1\\ 0&0&1\end{pmatrix}.
	\end{align*} 	
	For type C, an exemplar is 	\begin{align*}
		C_1 = \begin{pmatrix}1&0&0\\ 0&1&1\\ 0&0&-1\end{pmatrix}, \quad C_2 = \begin{pmatrix}1&-q^3-q^2&0\\ 0&-1&0\\ 0&1&1\end{pmatrix}, \quad C_3 = \begin{pmatrix}0&-1&0\\ -1&0&0\\ -q^3-q^3+1&-q^3-q^2+1&1\end{pmatrix}. 
	\end{align*}
	The data for all three types for $G_{23}$ is compiled in Table~\ref{G_{23} Nice 3-Tuples}. For all three types, choosing a different exemplar would not have resulted in a new orbit. The two orbits of size 10 arising from type A are the same, as are the two orbits of size 10 from type B. However, these two orbits are distinct across the two types in the sense that their signatures did not show up in the others' orbit. For type C, the two orbits of size 18 are the same. Information about how these six distinct orbits relate to the other classifications is provided in Table~\ref{G_{23} Nice 3-Tuple Comparisons}. In this group we also find nice 4-tuples; details are provided in Section~\ref{G_{23} 4-tups}

		\subsubsection{$G_{24}$} \label{G_{24} 3-tups}
	The group $G_{24}$ of order 336 is also denoted $W(J_{3}(4))$ with degrees $\deg(G_{24}) = \{4, 6, 14\}$ \cite[Table D.3]{ld2009}. All reflections in this group are conjugate with nontrivial eigenvalue $-1$. 
	Looking at the degrees, we first expand the base field to be the cyclotomic field of degree 84. There is one inverse pair of nice 3-tuples in this group, so we only provide the details for one in Table~\ref{G_{24} Types}.
	\begin{table}[H]
		\centering
		\begin{tabular}{|c|c|c|} 
			\hline 
			\multicolumn{3}{|c|}{$G_{24}$ types} \rule[0pt]{0pt}{3ex}\\ 
			\hline
			Type & Tuple eigenvalues & Residues of inverse product  \rule[0pt]{0pt}{3ex}\\ [0.5ex] 
			\hline
			A & $(-1, -1, -1)$ & $\{\frac{1}{14}, \frac{9}{14}, \frac{11}{14}\}$  \rule[0pt]{0pt}{3ex} \\ [0.5ex]  
			\hline 
		\end{tabular} \caption{\label{G_{24} Types}}
	\end{table}
	
	Letting $z = \zeta_7$, an exemplar of type A is
	\begin{align*}
		A_1 = \begin{pmatrix}1&-z^4-z^2-z&-z^4-z^2-z\\ \:0&0&-1\\ \:0&-1&0\end{pmatrix},  \quad A_2 = \begin{pmatrix}1&1&0\\ \:0&-1&0\\ \:0&1&1\end{pmatrix}, \quad A_3 = \begin{pmatrix}-1&0&0\\ \:1&1&0\\ \:z^4+z^2+z&0&1\end{pmatrix}.
	\end{align*}
	The middle convolution data for $G_{24}$ is provided in Table~\ref{G_{24} Nice 3-Tuples}. These three distinct orbits are the only orbits from type A, and we checked that in all three cases, the inverse of the exemplars do not produce any new orbits. The comparison data for $G_{24}$ is provided in Table~\ref{G_{24} Type Comparisons}.
       
	\begin{table}[H]
		\centering
		\begin{tabular}{|c|c|c|c|c|c|c|c|c|} 
			\hline
			$\lambda$ & $\omega_{X, Y, Z, 4}$ & \begin{sideways} Lisovyy--Tykhyy $\#$ \end{sideways} & $\theta_{1, 2, 3, 4}$ & $\sigma_{12, 23, 13, 24}$   & \begin{sideways} Tykhyy $\#$ \end{sideways}  & \begin{sideways} Their orbit size \end{sideways} & \begin{sideways} Our orbit size  \end{sideways} & \begin{sideways} $\SL_2(\C)$ subgroup \vspace{0.5ex} \end{sideways} \rule[0pt]{0pt}{4ex}\\ [0.5ex] 
			\hline 
			\multicolumn{9}{|c|}{$G_{24}$ type A comparisons} \rule[0pt]{0pt}{3ex}\\ 
			\hline
			$-z^4$ & $(1, 1, 1, 0)$ & $8$ & $(\frac{4}{7}, \frac{4}{7}, \frac{3}{7}, \frac{1}{7})$ & $(\frac{1}{2}, \frac{1}{2}, \frac{1}{2}, \frac{2}{3})$  & 28 &  7 & 28 & 0 \rule[0pt]{0pt}{4ex} \\ [1ex] 
			\hline 
			$-z$ & \small{$(1, 1, 1, 0)$} & 8 & $(\frac{1}{7}, \frac{5}{7}, \frac{1}{7}, \frac{1}{7})$ & $(\frac{1}{2}, \frac{1}{2}, \frac{1}{3} , \frac{1}{2})$ & 30  & 7 & 28 & 0 \rule[0pt]{0pt}{4ex}\\ [1ex] \hline 
			$-z^2$ & \small{$(1, 1,1, 0)$} & 8 & $(\frac{2}{7}, \frac{2}{7}, \frac{4}{7}, \frac{2}{7})$ & $(\frac{1}{2}, \frac{1}{2}, \frac{1}{3}, \frac{1}{2})$ & 29 & 7 & 28 & 0 \rule[0ex]{0ex}{4ex} \\ [1ex]
			\hline
		\end{tabular} \caption{\label{G_{24} Type Comparisons}Here $z = \zeta_{7}$.}
	\end{table}
	
	\subsubsection{$G_{25}$} \label{G_{25} 3-tups}
	The group $G_{25}$ of order 648 is also denoted $W(L_3)$ with degrees $\deg(G_{25}) = \{6, 9, 12\}$; see  \cite[Table D.3]{ld2009}. 
	The group $G_{25}$ has two conjugacy classes, consisting of matrices with nontrivial eigenvalue $w$ or $w^2$, respectively, for $w = \zeta_3$.
	Expanding the base field to be the cyclotomic field of order 36, we get eight possible types, which actually split into four pairs of inverses. We only list one set of inverse eigenvalues for each pair in Table~\ref{G_{25} Types}.
	\begin{table}[h!]
		\centering
		\begin{tabular}{|c|c|c|} 
			\hline 
			\multicolumn{3}{|c|}{$G_{25}$ types} \rule[0pt]{0pt}{3ex}\\ 
			\hline
			Type & Tuple eigenvalues & Residues of inverse product  \rule[0pt]{0pt}{3ex}\\ [0.5ex] 
			\hline
			A & $(w,w, w^2)$ & $\{\frac{5}{6}, \frac{1}{6}, \frac{2}{3}\}$  \rule[0pt]{0pt}{3ex} \\ [0.5ex] 
			\hline 
			B & $(w, w^2, w^2)$ & $\left\{ \frac{1}{9}, \frac{7}{9},  \frac{4}{9} \right\}$  \rule[0pt]{0pt}{3ex}\\ [0.5ex] \hline
			C & $(w, w, w)$ & $\{ \frac{11}{12}, \frac{5}{12}, \frac{2}{3}\}$  \rule[0pt]{0pt}{3ex}\\ [0.5ex]\hline
			D & $(w, w, w)$ & $\{\frac{1}{3}, \frac{5}{6}, \frac{5}{6}\}$  \rule[0pt]{0pt}{3ex}\\ [0.5ex] \hline
		\end{tabular} \caption{\label{G_{25} Types}Here $w = \zeta_3$ denotes a $\suprd{3}$ root of unity.}
	\end{table}
	
	With $w = \zeta_3$ the exemplars for each type are provided below. 
	An exemplar of type A is  
	\begin{align*}
		A_1 = \begin{pmatrix}1&0&0\\ 0&1&w^2\\ 0&0&w\end{pmatrix}, \quad A_2 = \begin{pmatrix}w&0&0\\ w^2&1&0\\ 0&0&1\end{pmatrix}, \quad A_3 = \begin{pmatrix}1&w&0\\ 0&w^2&0\\ 0&w&1\end{pmatrix}.
	\end{align*}
	An exemplar of type B is  
	\begin{align*}
		B_1 = \begin{pmatrix}1&0&0\\ 0&1&w^2\\ 0&0&w\end{pmatrix}, \quad B_2 = \begin{pmatrix}-w&1&w^2\\ 1&-w^2&1\\ w^2&1&-w\end{pmatrix}, \quad B_3 = \begin{pmatrix}1&w&0\\ 0&w^2&0\\ 0&w&1\end{pmatrix}. 
	\end{align*}
	An exemplar of type C is 
	\begin{align*}
		C_1 = \begin{pmatrix}1&-w^2&0\\ 0&w&0\\ 0&-w^2&1\end{pmatrix}, \quad C_2 = \begin{pmatrix}w&0&0\\ w^2&1&0\\ 0&0&1\end{pmatrix}, \quad C_3 = \begin{pmatrix}1&0&0\\ -w^2&w&-w^2\\ 0&0&1\end{pmatrix}. 
	\end{align*}
	Note that for type D, there are only two choices of eigenvalue to take middle convolution with, and only one of them will result in $2\times 2$ matrices. An exemplar of type D is  
	\begin{align*}
		D_1 = \begin{pmatrix}1&-w^2&-w\\ 0&-w^2&1\\ 0&-w&0\end{pmatrix}, \quad D_2 = \begin{pmatrix}w&0&0\\ w^2&1&0\\ 0&0&1\end{pmatrix}, \quad D_3 = \begin{pmatrix}0&-w&w^2\\ 0&1&0\\ -w^2&-1&-w^2\end{pmatrix}. 
	\end{align*}
	The details for the middle convolution of each type are provided in Table~\ref{G_{25} Nice 3-Tuples}. For all three types, different exemplars do not produce new orbits, nor do their inverses. For type A, the two orbits of size 12 are distinct, as are all three orbits of size 36 from type B. For type C, we had to extend to a field containing $\zeta_{72}$ to construct the character. The two orbits of size 4 from type C are the same. Including the one from type~D, we get eight distinct orbits from $G_{25}$. The comparison data is provided in Table~\ref{G_{25} Nice 3-Tuples Comparisons}. Comparing with the Lisovyy--Tykhyy paper, all the $\omega$-parameters seemed to correspond to the family of orbits listed in \cite[Lemma~39]{T-L}, so we have provided the specific values for the variables that work in each case. 
	In this group we also find nice 4-tuples; details are provided in Section~\ref{G_{25} 4-tups}.

        \clearpage
	
	\begin{table}
	  \centering
		\begin{tabular}{|c|c|c|c|c|c|c|c|c|} 
			\hline
			$\lambda$ & $\omega_{X, Y, Z, 4}$ & \begin{sideways}Lisovyy--Tykhyy $\#$ \end{sideways} & $\theta_{1, 2, 3, 4}$ & $\sigma_{12, 23, 13, 24}$   & \begin{sideways} Tykhyy $\#$ \end{sideways}  & \begin{sideways} Their orbit size \end{sideways} & \begin{sideways} Our orbit size  \end{sideways} & \begin{sideways} $\SL_2(\C)$ subgroup \vspace{1ex} \end{sideways} \rule[0pt]{0pt}{4ex}\\ [1ex] 
			\hline 
			\multicolumn{9}{|c|}{$G_{25}$ type A comparisons} \rule[0pt]{0pt}{3ex}\\ 
			\hline
			$-w$ & $(0, 0, 3, -2)$& Type II, $X' = 2, X'' = 1$ & $( \frac{1}{6}, \frac{1}{2}, \frac{1}{6}, \frac{1}{2})$ & $(\frac{1}{2}, \frac{1}{2}, \frac{1}{3}, 0)$  & 3 &  2 & 12 & $\inn{24, 4}$ \rule[0pt]{0pt}{4ex} \\ [1ex] 
			\hline 
			$-w^2$ & $(0, 0, 3, -2)$ & Type II, $X' = 2, X'' = 1$ & $(\frac{1}{6}, \frac{1}{2}, \frac{1}{6}, \frac{1}{2})$ & $(\frac{1}{2}, \frac{1}{2}, \frac{1}{3}, 0)$ & 3  & 2 & 12 & $\inn{24, 4}$ \rule[0pt]{0pt}{4ex}\\ [1ex] \hline 
			$w^2$ & $(0, 0, 3, -2)$ & Type II, $X' = 2, X'' = 1$ & $(1,  \frac{1}{3},1, \frac{2}{3})$ & $( \frac{1}{2}, \frac{1}{2} ,0,  \frac{1}{3})$ & 4  & 2 & 24 & 0 \rule[0pt]{0pt}{4ex}\\ [1ex] \hline 
			\multicolumn{9}{|c|}{$G_{25}$ type B comparisons} \rule[0pt]{0pt}{3ex}\\ 
			\hline
			$z^4$ & (2, -1, -1, -1) & Type III, $\omega = -1$ & $(\frac{4}{9}, \frac{2}{9}, \frac{2}{9}, \frac{2}{3})$ & $(\frac{1}{2}, \frac{1}{3}, \frac{1}{2}, \frac{2}{3})$ & 5 & 3 & 36 & 0 \rule[0ex]{0ex}{4ex} \\ [1ex]
			\hline
			$-z^{10}$ & (2, -1, -1, -1) & Type III, $\omega = -1$ & $(\frac{8}{9}, \frac{4}{9}, \frac{4}{9}, \frac{2}{3})$ & $(\frac{1}{2}, \frac{1}{3}, \frac{1}{2}, \frac{2}{3}^*)$ &$ 5^*$ & 3 & 36 & 0 \rule[0ex]{0ex}{4ex} \\ [1ex]
			\hline
			$z^{16}$ & (2, -1, -1, -1) & Type III, $\omega = -1$ & $(\frac{2}{9}, \frac{1}{9}, \frac{1}{9}, \frac{2}{3})$ & $(\frac{1}{2}, \frac{1}{3}, \frac{1}{2}, \frac{1}{3})$ & 5 & 3 & 36 & 0 \rule[0ex]{0ex}{4ex} \\ [1ex]
			\hline
			\multicolumn{9}{|c|}{$G_{25}$ type C comparisons} \rule[0pt]{0pt}{3ex}\\ 
			\hline
			$-z^9+z^3$ & $(4,4,4, -8)$ & Type IV, $\omega = 4$ & $(\frac{1}{4}, \frac{1}{4}, \frac{1}{4}, \frac{1}{4})$ & $(\frac{1}{3}, \frac{1}{3}, \frac{1}{3}, 0)$ & 6 & 4 & 4 & $\inn{48, 28}$ \rule[0ex]{0ex}{4ex} \\ [1ex] \hline
			$-z^6$ & $(4, 4, 4, -8)$ & Type IV, $\omega = 4$ & $(0, 0, 0, \frac{1}{2})$ & $ (\frac{1}{3}, \frac{1}{3}, \frac{1}{3},0 )$ & 7 & 4 & 4 & 0   \rule[0ex]{0ex}{4ex} \\ [1ex] \hline
			\multicolumn{9}{|c|}{$G_{25}$ type D comparisons} \rule[0pt]{0pt}{3ex}\\ 
			\hline
			$w$ & $(3, 3, 3, -5)$ & Type I, $X =Y=Z=1$ & $(\frac{2}{3}, \frac{2}{3}, 0, \frac{1}{3})$ & $ (\frac{1}{3}, \frac{2}{3}, \frac{2}{3}, \frac{2}{3})$ & 2 & 1 & 4 & $\inn{24, 3} $  \rule[0ex]{0ex}{4ex} \\ [1ex] \hline
		\end{tabular} \caption{\label{G_{25} Nice 3-Tuples Comparisons}Here $w = \zeta_3, z = \zeta_{36}$. The subgroups $\inn{24, 4}, \inn{48, 28}, \inn{24, 3}$ correspond to the dicyclic group of order 24 (binary dihedral group), binary octahedral group, and  $\SL(2, 3)$, respectively.
        The signature for Tykhyy's orbit $5$ is $\theta_{1, 2, 3, 4} = (2x, x, x, \frac{2}{3})$ and $\sigma_{12, 23, 13, 24} = (\frac{1}{2}, \frac{1}{3}, \frac{1}{2}, 3x)$, but such a signature was not in the orbit, so we provide the closest match. The coordinate that differs from the formula is marked with an asterisk, but note that as $\sigma_{24}$ is only defined modulo 2 and up to sign, we believe this still corresponds to orbit 5 of Tykhyy's Table 1. }
	\end{table}

        \clearpage
	\subsubsection{$G_{26}$} \label{G_{26} 3-tups}
	The group $G_{26}$ of order 1296 is also denoted $W(M_3)$ with degrees $\deg(G_{26}) = \{6, 12, 18\}$; see \cite[Table D.3]{ld2009}.
	In $G_{26}$ there are three conjugacy classes of reflections, corresponding to nontrivial eigenvalues in $\{\zeta_3, \zeta_3^2, -1\}$. In this group there are six sets of eigenvalues for the inverse product, which correspond to three inverse pairs, so we only list one from each pair.  
	\begin{table}[h!]
		\centering
		\begin{tabular}{|c|c|c|} 
			\hline 
			\multicolumn{3}{|c|}{$G_{26}$ types} \rule[0pt]{0pt}{3ex}\\ 
			\hline
			Type & Tuple eigenvalues & Residues of inverse product  \rule[0pt]{0pt}{3ex}\\ [0.5ex] 
			\hline
			A & $(w,w^2, -1)$ & $\{\frac{1}{12}, \frac{7}{12}, \frac{5}{6}\}$  \rule[0pt]{0pt}{3ex} \\ [0.5ex] 
			\hline 
			B & $(w, w, -1)$ & $\left\{ \frac{5}{18}, \frac{11}{18},  \frac{17}{18} \right\}$  \rule[0pt]{0pt}{3ex}\\ [0.5ex] \hline
			C & $(w^2, w^2, -1)$ & $\{ \frac{1}{6}, \frac{1}{6}, \frac{5}{6}\}$  \rule[0pt]{0pt}{3ex}\\ [0.5ex]\hline
		\end{tabular} \caption{\label{G_{26} Types}Here $w = \zeta_{3}$.}
	\end{table}
        
	The exemplars for this group are all provided over $\Q(\zeta_{36})$ with $z = \zeta_{36}$.
	
	Here is an exemplar of type A:
	\begin{align*}
		A_1 = \begin{pmatrix}1&0&0\\ -1&2\cdot z^6&z^6\\ -z^6+2&-3\cdot z^6&-z^6\end{pmatrix}, \quad A_2 = \begin{pmatrix}0&z^6&0\\ 1&-z^6+1&0\\ z^6-2&z^6+1&1\end{pmatrix}, \quad A_3 = \begin{pmatrix}1&0&0\\ 0&1&1\\ 0&0&-1\end{pmatrix}.
	\end{align*}
	Here is an exemplar for type B: 
	\begin{align*}
		B_1 = \begin{pmatrix}z^6-1&2\cdot z^6&z^6\\ 0&1&0\\ 0&0&1\end{pmatrix}, \quad 	B_2 = \begin{pmatrix}0&z^6-1&z^6-1\\ z^6&2&1\\ -z^6-1&z^6-2&z^6-1\end{pmatrix}, \quad B_3 = \begin{pmatrix}1&0&0\\ 0&1&1\\ 0&0&-1\end{pmatrix}.
	\end{align*}
	Here is an exemplar for type C: 
	\begin{align*}
		C_1= \begin{pmatrix}1&0&0\\ -z^6&z^6-1&z^6-1\\ z^6+1&-3\cdot \left(z^6+1\right)&-2\cdot \left(z^6-1\right)\end{pmatrix}, \quad C_2 = \begin{pmatrix}1&-z^6+1&-z^6+1\\ 0&1&0\\ 0&-z^6-1&-z^6\end{pmatrix}, \quad C_3 = \begin{pmatrix}1&0&0\\ 0&1&1\\ 0&0&-1\end{pmatrix}. 
	\end{align*}
	
	In $G_{26}$ neither the choice of different exemplars nor the inverses for any type produce new orbits. The middle convolution data for this group can be found in Table~\ref{G_{26} Nice 3-Tuples}. Further, all three orbits from type A are distinct, as are all three orbits from type B. When comparing with Tykhyy's classification, there were two orbit (from type B) that we could not find an exact match for. We listed the closest match, distinguishing the differing coordinate with an asterisk. Note that both cases agree with orbit 5 in \cite[Table 1]{Tyk} when the formulas are read modulo 2 and up to sign, as Tykhyy intended. 
	
	In this group we also find nice 4-tuples; details are provided in Section~\ref{G_{26} 4-tups}

        \begin{sidewaystable}
	  \centering 
	    \begin{tabular}{|c|c|c|c|c|c|c|c|c|} 
			\hline
			$\lambda$ & $\omega_{X, Y, Z, 4}$ &  Lisovyy--Tykhyy $\#$ & $\theta_{1, 2, 3, 4}$ & $\sigma_{12, 23, 13, 24}$   & \begin{turn}{-90} Tykhyy $\#$ \end{turn}  & \begin{turn}{-90} Their orbit size \end{turn} & \begin{turn}{-90} Our orbit size  \end{turn} & \begin{turn}{-90} $\SL_2(\C)$ subgroup \vspace{0.5ex} \end{turn} \rule[0pt]{0pt}{0ex}\\ [0ex] 
			\hline 
			\multicolumn{9}{|c|}{$G_{26}$ type A comparisons} \rule[0pt]{0pt}{3ex}\\ 
			\hline
			$z^3$ & $(0, \sqrt{3}, 0, 0)$& Type II, $X' = 0, X'' = \sqrt{3}$ & $( \frac{5}{12}, \frac{1}{4}, \frac{5}{12}, \frac{3}{4})$ & $(\frac{1}{2}, \frac{1}{2}, \frac{5}{6}, \frac{1}{2})$  & 4 &  2 & 24 & 0 \rule[0pt]{0pt}{4ex} \\ [1ex] 
			\hline 
			$-z^3$ & $(0, -\sqrt{3}, 0, 0)$ & Type II, $X' = 0, X'' = -\sqrt{3}$ & $(\frac{1}{4}, \frac{1}{12}, \frac{1}{4}, \frac{11}{12})$ & $(\frac{1}{2}, \frac{1}{2}, \frac{1}{2} , \frac{5}{6})$ & 4  &  2 & 24 & 0 \rule[0pt]{0pt}{4ex}\\ [1ex] \hline 
			$-z^6+1$ & $(0, -z^9+2z^3, 0, 0)$ & \small{Type II, $X' = 0, X'' = -z^9+2z^3$} & $(\frac{1}{6}, \frac{1}{2}, \frac{2}{3}, \frac{1}{2})$ & $(\frac{1}{2}, \frac{1}{2}, , \frac{5}{6}, \frac{1}{2} )$ & 3  & 2 & 24 & \small{$\inn{24, 4}$} \rule[0pt]{0pt}{4ex}\\ [1ex] \hline 
			\multicolumn{9}{|c|}{$G_{26}$ type B comparisons} \rule[0pt]{0pt}{3ex}\\ 
			\hline
			$z^{10}$ & \small{$(-z^9+2z^3, -z^9+2z^3, 2, -1)$} & Type III, $\omega = -z^9+2z^3$ & $(\frac{7}{9},\frac{7}{18}, \frac{7}{18}, \frac{2}{3})$ & $( \frac{1}{2}, \frac{1}{3}, \frac{1}{2}, \frac{5}{6}^*)$ & $5^*$ & 3 & 36 & 0 \rule[0ex]{0ex}{4ex} \\ [1ex]
			\hline
			$-z^{4}$ & \small{$(z^9-2z^3, z^9-2z^3, 2, -1)$} & Type III, $\omega = z^9-2z^3$ & $(\frac{1}{9},\frac{1}{18}, \frac{1}{18}, \frac{2}{3})$ & $(\frac{1}{2}, \frac{1}{3}, \frac{1}{2}, \frac{1}{6})$ & 5 & 3 & 36 & 0 \rule[0ex]{0ex}{4ex} \\ [1ex]
			\hline
			$-z^{16}$ & \small{$(-z^9+2z^3, -z^9+2z^3, 2, -1)$} & Type III, $\omega = -z^9+2z^3$ & $(\frac{5}{9}, \frac{5}{18}, \frac{5}{18}, \frac{2}{3})$ & $(\frac{1}{2}, \frac{1}{3}, \frac{1}{2}, \frac{1}{6}^*)$ & $5^*$ & 3 & 36 & 0 \rule[0ex]{0ex}{4ex} \\ [1ex]
			\hline
			\multicolumn{9}{|c|}{$G_{26}$ type C comparisons} \rule[0pt]{0pt}{3ex}\\ 
			\hline
			$-z^6+1$ & $(0, 0, 2, -1)$ & Type III, $\omega = 0$ & $(1, \frac{1}{2}, \frac{1}{2}, \frac{2}{3})$ & $( \frac{1}{2}, \frac{1}{3}, \frac{1}{2}, \frac{1}{2})$ & 5 & 3 & 12 & $\inn{12, 1}$ \rule[0ex]{0ex}{4ex} \\ [1ex]  \hline
		\end{tabular}
                \caption{\label{G_{26} Nice 3-Tuples Comparisons}Here $w = \zeta_3, z = \zeta_{36}$. The subgroups $\inn{24, 4}$ and $\inn{12, 1}$ are the dicyclic groups of order 24 and 12, respectively. 
                  The signature for Tykhyy's orbit $5$ is $\theta_{1, 2, 3, 4} = (2x, x, x, \frac{2}{3})$ and $\sigma_{12, 23, 13, 24} = (\frac{1}{2}, \frac{1}{3}, \frac{1}{2}, 3x)$, but such a signature was not in the orbit, so we provide the closest match. The coordinate that differs from the formula is marked with an asterisk, but note that as $\sigma_{24}$ is only defined modulo 2 and up to sign, we believe this still corresponds to Orbit 5 of Tykhyy's Table 1. }
	\end{sidewaystable}     
 
	\clearpage
	
	\subsubsection{$G_{27}$}	\label{G_{27} 3-tups}
	This group is of order 2160 and is also denoted $W(J_3(5))$.
	This group only has one conjugacy class of reflections with nontrivial eigenvalue $-1$, and $\deg(G_{27}) = \{6, 12, 30\}$. The original base ring of this group in Magma is the cyclotomic field of order 15, and we expanded the base ring to the cyclotomic filed of order 60 to ensure all the eigenvalues are found. We found three distinct inverse pair types in this group.
	\begin{table}[h!]
		\centering
		\begin{tabular}{|c|c|c|} 
			\hline 
			\multicolumn{3}{|c|}{$G_{27}$ types} \rule[0pt]{0pt}{3ex}\\ 
			\hline
			Type & Tuple eigenvalues & Residues of inverse product  \rule[0pt]{0pt}{3ex}\\ [0.5ex] 
			\hline
			A & $(-1, -1 -1)$ & $\{\frac{2}{60}, \frac{38}{60}, \frac{50}{60}\}$  \rule[0pt]{0pt}{3ex} \\ [0.5ex] 
			\hline 
			B & $(-1, -1 -1)$ & $\left\{ \frac{5}{60}, \frac{35}{60},  \frac{50}{60} \right\}$  \rule[0pt]{0pt}{3ex}\\ [0.5ex] \hline
			C & $(-1, -1, -1)$ & $\{ \frac{10}{60}, \frac{34}{60}, \frac{46}{60}\}$  \rule[0pt]{0pt}{3ex}\\ [0.5ex]\hline
		\end{tabular} \caption{\label{G_{27} Types}}
	\end{table}
	
	Writing $z = \zeta_{60}$ and setting $u = -z^{10}-z^8+z^2$ and $v = z^{16} + z^4$, we get that $z^{50}v = -z^{14}+ z^6+ z^4$, so we can write an exemplar for type A as 
	\begin{align*}
		A_1 = \begin{pmatrix}z^{50}v&0&z^{20}\\ -v+1&1&v+1\\ v&0&-z^{50}v\end{pmatrix}, \quad A_2 = \begin{pmatrix}0&0&-z^{20}\\ u+z^{10}+2&1&v+z^{10}\\ z^{10}&0&0\end{pmatrix}, \quad A_3 = \begin{pmatrix}1&1&0\\ 0&-1&0\\ 0&1&1\end{pmatrix}.
	\end{align*}
	Using the same notation, we can write a type B exemplar as 
	\begin{align*}
		B_1 = \begin{pmatrix}-1&u&u\\ v+1&-u&-u-1\\ -v-1&u+1&u+2\end{pmatrix}, \quad B_2 = \begin{pmatrix}z^{50}v&0&z^{20}\\ -v+1&1&v+1\\ v&0&-z^{50}v\end{pmatrix}, \quad B_3 = \begin{pmatrix}1&1&0\\ 0&-1&0\\ 0&1&1\end{pmatrix}.
	\end{align*}
	Similarly, a type C exemplar is
	\begin{align*}
		C_1 = \begin{pmatrix}z^{50}v&0&z^{20}\\ -v+1&1&v+1\\ v&0&-z^{50}v\end{pmatrix}, \quad C_2= \begin{pmatrix}1&1&0\\ 0&-1&0\\ 0&1&1\end{pmatrix}, \quad  C_3 = \begin{pmatrix}1&0&0\\ -v-1&u+1&u+2\\ v+1&-u&-u-1\end{pmatrix}.
	\end{align*}
	In $G_{27}$ for all three types, different exemplars do not produce new orbits. Moreover, the inverse types do not produce new orbits, and we have checked that these nine orbits are distinct. The middle convolution details for types A, B, and C are listed in Tables~\ref{G_{27} Type A 3-Tuples}, \ref{G_{27} Type B 3-Tuples}, and~\ref{G_{27} Type C 3-Tuples}, respectively.  The comparison data is provided in Table~\ref{G_{27} Type Comparisons}. In this group we also find nice 4-tuples; details are provided in Section~\ref{G_{27} 4-tups}
	\begin{table}
		\centering
		\begin{tabular}{|c|c|c|c|c|c|c|c|c|} 
			\hline
			$\lambda$ & $\omega_{X, Y, Z, 4}$ & \begin{sideways} Lisovyy--Tykhyy $\#$ \end{sideways} & $\theta_{1, 2, 3, 4}$ & $\sigma_{12, 23, 13, 24}$   & \begin{sideways} Tykhyy $\#$ \end{sideways}  & \begin{sideways} Their orbit size \end{sideways} & \begin{sideways} Our orbit size  \end{sideways} & \begin{sideways} $\SL_2(\C)$ subgroup \vspace{1ex}  \end{sideways} \rule[0pt]{0pt}{4ex}\\ [0.5ex] 
			\hline 
			\multicolumn{9}{|c|}{$G_{27}$ type A comparisons} \rule[0pt]{0pt}{3ex}\\ 
			\hline
			$z^2$ & $(\frac{3-\sqrt{5}}{2}, \frac{3-\sqrt{5}}{2}, \frac{3-\sqrt{5}}{2}, \sqrt{5}-1)$& 26 & $(\frac{8}{15},  \frac{7}{15}, \frac{4}{5}, \frac{7}{15})$ & $(\frac{2}{5}, \frac{2}{5}, \frac{1}{2}, \frac{1}{2})$  & 78 &  15 & 60 & 0 \rule[0pt]{0pt}{4ex} \\ [1ex] 
			\hline 
			$-z^8$ & $(\frac{3-\sqrt{5}}{2}, \frac{3-\sqrt{5}}{2}, \frac{3-\sqrt{5}}{2}, \sqrt{5}-1)$ & 26 & $(  \frac{13}{15},  \frac{1}{5}, \frac{2}{15}, \frac{2}{15})$ & $(\frac{3}{5}, \frac{2}{5}, \frac{1}{2} , \frac{1}{2})$ & 87  & 15 & 60 & 0 \rule[0pt]{0pt}{4ex}\\ [1ex] \hline 
			$-z^{20}$ & $(\frac{3-\sqrt{5}}{2}, \frac{3-\sqrt{5}}{2}, \frac{3-\sqrt{5}}{2}, \sqrt{5}-1)$ & 26 & $(\frac{1}{3}, \frac{1}{3}, \frac{3}{5}, \frac{1}{3})$ & $(\frac{3}{5}, \frac{3}{5}, \frac{1}{2} , \frac{1}{2})$ & 81  & 15 & 60 & $\inn{120, 5}$ \rule[0pt]{0pt}{4ex}\\ [1ex] \hline 
			\multicolumn{9}{|c|}{$G_{27}$ type B comparisons} \rule[0pt]{0pt}{3ex}\\ 
			\hline
			$z^5$ & (1, 1, 1, 1) & 39 & $(\frac{7}{12}, \frac{5}{12}, \frac{7}{12}, \frac{1}{4})$ & $(\frac{2}{5}, \frac{1}{2}, \frac{1}{3}, \frac{1}{2})$ & 113 & 24 & 96 & 0 \rule[0ex]{0ex}{4ex} \\ [1ex]
			\hline
			$-z^{5}$ & (1, 1, 1, 1) & 39 & $(\frac{3}{4}, \frac{1}{12}, \frac{1}{12}, \frac{1}{12})$ & $(\frac{1}{2}, \frac{3}{5}, \frac{1}{3}, \frac{1}{2})$ & 114 & 24 & 96 & 0 \rule[0ex]{0ex}{4ex} \\ [1ex]
			\hline
			$-z^{20}$ & (1, 1, 1, 1) & 39 & $(\frac{1}{2},\frac{1}{3}, \frac{2}{3}, \frac{1}{3})$ & $(\frac{1}{2},  \frac{2}{5}, \frac{1}{2}, \frac{1}{3})$ & 112 & 24 & 96 & $\inn{120, 5}$ \rule[0ex]{0ex}{4ex} \\ [1ex]
			\hline
			\multicolumn{9}{|c|}{$G_{27}$ type C comparisons} \rule[0pt]{0pt}{3ex}\\ 
			\hline
			$z^{10}$ & $(\frac{3+\sqrt{5}}{2}, \frac{3+\sqrt{5}}{2}, \frac{3+\sqrt{5}}{2}, -1-\sqrt{5})$ & 27 & $(\frac{4}{5}, \frac{2}{3}, \frac{1}{3}, \frac{1}{3})$ & $(\frac{1}{5}, \frac{4}{5}, \frac{1}{2}, \frac{1}{2})$ & 84 & 15 & 60 & $\inn{120, 5}$ \rule[0ex]{0ex}{4ex} \\ [1ex] \hline
			$-z^4$ & $(\frac{3+\sqrt{5}}{2}, \frac{3+\sqrt{5}}{2}, \frac{3+\sqrt{5}}{2}, -1-\sqrt{5})$ & 27 & $(\frac{14}{15}, \frac{14}{15}, \frac{3}{5}, \frac{1}{15})$ & $ (\frac{1}{5}, \frac{4}{5}, \frac{1}{2}, \frac{1}{2})$ & 83 & 15 & 60 & 0   \rule[0ex]{0ex}{4ex} \\ [1ex] \hline
			$-z^{16}$ & $(\frac{3+\sqrt{5}}{2}, \frac{3+\sqrt{5}}{2}, \frac{3+\sqrt{5}}{2}, -1-\sqrt{5})$ & 27 & $(\frac{4}{15}, \frac{4}{15},  \frac{2}{5}, \frac{4}{15})$ & $ (\frac{1}{5}, \frac{1}{5}, \frac{1}{2}, \frac{1}{2})$ & 82 & 15 & 60 & 0   \rule[0ex]{0ex}{4ex} \\ [1ex] \hline
		\end{tabular} \caption{\label{G_{27} Type Comparisons}Here $z = \zeta_{60}$. All the finite subgroups correspond to $\SL(2, 5)$.}
	\end{table}
	
	\clearpage
	\subsection{Nice 4-tuples in rank 3 groups}\label{prim4}
	Nice tuples in this group are 4-tuples of reflections generating the whole group such that the inverse product has a nontrivial eigenvalue of multiplicity 2. If $\Q(\zeta)$ denotes the trace field for some FCRG $W$ and we have a matrix $A$ with eigenvalues $\{\lambda, \lambda, \lambda_1\}$, then we know $\tr(A) = 2\lambda + \lambda_1 \in \Q(\zeta)$. If we further know that $\lambda \in \Q(\zeta)$, then we can conclude that $\lambda_1 \in \Q(\zeta)$ and so we do not need to extend the base field to ensure all characteristic polynomials of such matrices $A$ split. By Theorem~\ref{ReflEigvThm}, we know that any such $\lambda$ of multiplicity 2 must have order dividing at least two of the degrees of $W$. Studying degrees, we see that in $G_{23}$ and $G_{24}$, $\lambda$ must have order dividing 2, \textit{i.e.}~$\lambda = -1$, and so we do not need to expand the base ring. Similarly, for $G_{25}$, $G_{26}$, $G_{27}$, such  $\lambda$ must have order dividing 6, which means we do not have to expand the base ring in these cases either. 
	Our search found nice 4-tuples in all the primitive rank 3 groups except for~$G_{24}$. 
	
	In comparing with the existing classifications, we can only compare with Tykhyy's $\mathcal{M}^{(5)}$ orbits, see \cite{Tyk}, since our 4-tuples give us local systems on the 5-punctured sphere. Unfortunately, without the $\omega$-parameters from the work of Lisovyy--Tykhyy \cite{T-L} like in the 3-tuple cases, it is more difficult to relate our orbit sizes to Tykhyy's. We still computed all the signatures for the different possible equivalent representatives for each orbit and attempted to find a match, which was successful. 
	
	\subsubsection{$G_{23}$} \label{G_{23} 4-tups}
	We only found one type of nice tuple. 
	\begin{table}[h!]
		\centering
		\begin{tabular}{|c|c|c|} 
			\hline 
			\multicolumn{3}{|c|}{$G_{23}$ 4-tuple types} \rule[0pt]{0pt}{3ex}\\ 
			\hline
			Type & Tuple eigenvalues & Residues of inverse product  \rule[0pt]{0pt}{3ex}\\ [0.5ex] 
			\hline
			A & $(-1, -1, -1, -1)$ & $\{\frac{1}{2}, \frac{1}{2}, 0\}$  \rule[0pt]{0pt}{3ex} \\ [0.5ex] 
			\hline 
		\end{tabular} \caption{\label{G_{23}}}
	\end{table}

	Letting $q = \zeta_5$, here is an exemplar for type A:  
	\begin{align*}
		A_1 &= \begin{pmatrix}q^3+q^2&q^3+q^2&q^3+q^2\\ -q^3-q^2+1&-q^3-q^2+1&-q^3-q^2\\ q^3+q^2&-1&0\end{pmatrix}, \quad A_2 = \begin{pmatrix}-q^3-q^2+1&-q^3-q^2+1&-q^3-q^2\\ q^3+q^2&q^3+q^2&q^3+q^2\\ -1&q^3+q^2&0\end{pmatrix}, \\ A_3 &= \begin{pmatrix}-q^3-q^2&1&0\\ q^3+q^2&q^3+q^2&0\\ -q^3-q^2&-q^3-q^2+1&1\end{pmatrix}, \quad A_4 = \begin{pmatrix}-1&0&0\\ -q^3-q^2&1&0\\ 0&0&1\end{pmatrix}.
	\end{align*}
	The middle convolution details can be found in Table~\ref{G_{23} Nice 4-Tuples}. This type is its own inverse. 
	For this tuple, we could not compute the orbit size as it was too large, but we were able to find a signature in the orbit (up to Tykhyy's equivalence relations) that agreed with one listed in Tykhyy's paper. Because we cannot compute the full braid group orbit, we also cannot check if different exemplars would result in a different orbit after middle convolution. We provide the comparison details in Table~\ref{G_{23} 4-Tuple Comparisons}. 
	\begin{table}[H]
		\small \centering
		\begin{tabular}{|c|c|c|c|c|c|c|c|} 
			\hline
			Type & $\lambda$ &  $\theta_{1, 2, 3, 4, 5}$ & $\sigma_{12, 23, 34, 45, 51, 13, 24}$   & \begin{sideways} Tykhyy $\#$ \end{sideways}  & \begin{sideways} Their orbit size \end{sideways} & \begin{sideways} Our orbit size  \end{sideways} & \begin{sideways} $\SL_2(\C)$ subgroup \vspace{0.5ex} \end{sideways} \rule[0pt]{0pt}{4ex}\\ [0.5ex] 
			\hline 
			\multicolumn{8}{|c|}{$G_{23}$ type comparisons} \rule[0pt]{0pt}{3ex}\\ 
			\hline
			A& $-1$ & $(1, 1, 1, 1, 1)$ & $(\frac{1}{3}, \frac{1}{3}, \frac{1}{3}, \frac{1}{3}, \frac{1}{3}, \frac{1}{5}, \frac{1}{5})$  &8 &  192 &  - & 0  \rule[0pt]{0pt}{4ex} \\ [1ex] 
			\hline
		\end{tabular} \caption{\label{G_{23} 4-Tuple Comparisons}}
	\end{table}
	
	\subsubsection{$G_{25}$}\label{G_{25} 4-tups}
	We found six different types which split into three inverse pairs, so we only provide the details for one type from each pair in Table~\ref{G_{25} 4-Tuple Types}. 
	\begin{table}[H]
		\centering
		\begin{tabular}{|c|c|c|} 
			\hline 
			\multicolumn{3}{|c|}{$G_{25}$ 4-tuple types} \rule[0pt]{0pt}{3ex}\\ 
			\hline
			Type & Tuple eigenvalues & Residues of inverse product  \rule[0pt]{0pt}{3ex}\\ [0.5ex] 
			\hline
			A & $(z, z^2, z^2, z^2)$ & $\{\frac{1}{3}, \frac{1}{3}, 0\}$  \rule[0pt]{0pt}{3ex} \\ [0.5ex] 
			\hline 
			B & $(z^2, z^2, z^2, z^2)$ & $\left\{ \frac{1}{3}, \frac{1}{3},  \frac{2}{3} \right\}$  \rule[0pt]{0pt}{3ex}\\ [0.5ex] \hline
			C & $(z^2, z, z, z^2)$ & $\{ \frac{2}{3}, \frac{1}{6}, \frac{1}{6}\}$  \rule[0pt]{0pt}{3ex}\\ [0.5ex]\hline
		\end{tabular} \caption{\label{G_{25} 4-Tuple Types} Here $z = \zeta_3$.}
	\end{table}
	
	For all the exemplars listed below, $z = \zeta_3$. 
	\newline
	Here is an exemplar for type A: 
	\begin{align*}
		A_1 = \begin{pmatrix}z&0&0\\ z^2&1&0\\ 0&0&1\end{pmatrix},\quad A_2 = \begin{pmatrix}-z&1&-z\\ 0&1&0\\ z&z^2&0\end{pmatrix}, \quad A_3 = \begin{pmatrix}1&z&0\\ 0&z^2&0\\ 0&z&1\end{pmatrix} \quad A_4 = \begin{pmatrix}-z&1&z^2\\ 1&z+1&1\\ z^2&1&-z\end{pmatrix}.
	\end{align*}
	Here is an exemplar for type B: 
	\begin{align*}
		B_1 = \begin{pmatrix}-z&1&0\\ z+1&0&0\\ 1&z&1\end{pmatrix},\quad B_2 = \begin{pmatrix}-z&1&-z\\ 0&1&0\\ z&z^2&0\end{pmatrix},\quad B_3 = \begin{pmatrix}0&z^2&z\\ 0&1&0\\ -z&1&-z\end{pmatrix}, \quad B_4 = \begin{pmatrix}1&z&0\\ 0&z^2&0\\ 0&z&1\end{pmatrix}.
	\end{align*}
	Here is an exemplar for type C:
	\begin{align*}
		C_1 = \begin{pmatrix}1&z&1\\ 0&0&z+1\\ 0&1&-z\end{pmatrix}, \quad C_2 = \begin{pmatrix}z&0&0\\ z^2&1&0\\ 0&0&1\end{pmatrix}, \quad C_3 = \begin{pmatrix}z+1&-1&z+1\\ 0&1&0\\ z^2&-z&0\end{pmatrix}, \quad C_4 = \begin{pmatrix}-z&1&0\\ z+1&0&0\\ 1&z&1\end{pmatrix}.
	\end{align*}
	The middle convolution details for $G_{25}$ can be found in Table~\ref{G_{25} Nice 4-Tuples}.
	The orbits for types A and~B are the same. All three types do not split into new orbits, and their inverses do not produce new orbits. The comparison details for the two distinct orbits can be found in Table~\ref{G_{25} 4-Tuple Comparisons}.

        \clearpage
	\begin{table}
		\centering
		\begin{tabular}{|c|c|c|c|c|c|c|c|} 
			
			\hline
			Type & $\lambda$ &  $\theta_{1, 2, 3, 4, 5}$ & $\sigma_{12, 23, 34, 45, 51, 13, 24}$   & \begin{sideways} Tykhyy $\#$ \end{sideways}  & \begin{sideways} Their orbit size \end{sideways} & \begin{sideways} Our orbit size  \end{sideways} & \begin{sideways} $\SL_2(\C)$ subgroup \vspace{0.5ex} \end{sideways} \rule[0pt]{0pt}{4ex}\\ [0.5ex] 
			\hline 
			\multicolumn{8}{|c|}{$G_{25}$ type comparisons} \rule[0pt]{0pt}{3ex}\\ 
			\hline
			A& $z$ & $(0, 0, 0, 0, \frac{2}{3})$ & $(0, 0, \frac{1}{3}, 0, \frac{1}{2}, 0, \frac{1}{3})$  &3 &  9 &  45 & 0  \rule[0pt]{0pt}{4ex} \\ [1ex] 
			\hline 
			C & $z+1$ & $(\frac{1}{2}, \frac{1}{6}, \frac{2}{3}, \frac{1}{6}, \frac{1}{2})$& $(\frac{1}{2}, \frac{5}{6}, \frac{5}{6}, \frac{1}{2}, 1, \frac{1}{2}, \frac{1}{3})$ & 1 & 4 &  120 & $\inn{24, 4}$  \rule[0ex]{0ex}{4ex} \\ [1ex]
			\hline
		\end{tabular} \caption{\label{G_{25} 4-Tuple Comparisons}Here $z = \zeta_{3}$. The subgroup of order 24 is the dicyclic group.}
	\end{table}

	\subsubsection{$G_{26}$}\label{G_{26} 4-tups}
	In this group we found two inverse pairs of nice tuples, and we provide the details for one type from each pair in Table~\ref{G_{26} 4-Tuple Types}. 

        \begin{table}[H]
		\centering
		\begin{tabular}{|c|c|c|} 
			\hline 
			\multicolumn{3}{|c|}{$G_{26}$ 4-tuples types} \rule[0pt]{0pt}{3ex}\\ 
			\hline
			Type & Tuple eigenvalues & Residues of inverse product  \rule[0pt]{0pt}{3ex}\\ [0.5ex] 
			\hline
			A & $(-1, z, z^2, z^2)$ & $\{\frac{1}{6}, \frac{1}{6}, \frac{1}{2}\}$  \rule[0pt]{0pt}{3ex} \\ [0.5ex] 
			\hline 
			B & $(z, z^2, z, -1)$ & $\{\frac{1}{6}, \frac{1}{6}, \frac{5}{6}\}$  \rule[0pt]{0pt}{3ex} \\ [0.5ex]
			\hline 
		\end{tabular} \caption{\label{G_{26} 4-Tuple Types} Here $z = \zeta_3$.}
	\end{table}
	With the notation $z = \zeta_3$, here is an exemplar for type A: 
	\begin{align*}
		A_1 = \begin{pmatrix}z&z+2&1\\ 0&1&0\\ z+2&z^2-z&-z\end{pmatrix}, \quad A_2 = \begin{pmatrix}z&0&0\\ z^2&1&0\\ 0&0&1\end{pmatrix}, \quad A_3 = \begin{pmatrix}-z+1&2z^2&z^2\\ z^2&-1&-1\\ 0&0&1\end{pmatrix}, \quad A_4 = \begin{pmatrix}-z&-1&-1\\ z&z+2&z+1\\ z^2-z&z^2-1&z^2\end{pmatrix}.
	\end{align*}
	An exemplar for type B is 
	\begin{align*}
		B_1 = \begin{pmatrix}1&z^2&0\\ 0&z&0\\ 0&-z+1&1\end{pmatrix}, \quad B_2 = \begin{pmatrix}-z&-1&0\\ z^2&0&0\\ 2z+1&z+2&1\end{pmatrix}, \quad B_3 = \begin{pmatrix}z&0&0\\ z+1&1&0\\ 0&0&1\end{pmatrix}, \quad B_4 = \begin{pmatrix}1&z^2-z&-z\\ 0&z^2&z^2\\ 0&2z+1&z+1\end{pmatrix}.
	\end{align*}
	The middle convolution details for $G_{26}$ can be found in Table~\ref{G_{26} Nice 4-Tuples}. 
	Neither of these types split into multiple orbits, nor do their inverse types produce new orbits. The finite group generated by the matrices in $\SL_2(\C)$ are the same for both groups. Magma identifies this as \texttt{SmallGroup}$<24, 4>$, which corresponds to the dicyclic group of order 24.

        \clearpage
	\begin{table}
		\centering
		\begin{tabular}{|c|c|c|c|c|c|c|c|} 
			\hline
			Type & $\lambda$ &  $\theta_{1, 2, 3, 4, 5}$ & $\sigma_{12, 23, 34, 45, 51, 13, 24}$   & \begin{sideways} Tykhyy $\#$ \end{sideways}  & \begin{sideways} Their orbit size \end{sideways} & \begin{sideways} Our orbit size  \end{sideways} & \begin{sideways} $\SL_2(\C)$ subgroup \vspace{0.5ex} \end{sideways} \rule[0pt]{0pt}{4ex}\\ [0.5ex] 
			\hline 
			\multicolumn{8}{|c|}{$G_{26}$ type comparisons} \rule[0pt]{0pt}{3ex}\\ 
			\hline
			A& $z+1$ & $(\frac{1}{2}, \frac{1}{6}, \frac{1}{6}, \frac{1}{3}, \frac{1}{2})$ & $(\frac{1}{2}, \frac{1}{3}, \frac{1}{2}, \frac{1}{2}, \frac{2}{3}, \frac{1}{2}, \frac{1}{2})$  &1 &  4 &  120 & $\inn{24, 4}$  \rule[0pt]{0pt}{4ex} \\ [1ex] 
			\hline 
			B & $z+1$ & $(\frac{1}{2}, \frac{1}{6},  \frac{2}{3}, \frac{1}{6}, \frac{1}{2})$ & $(\frac{1}{2}, \frac{5}{6}, \frac{5}{6}, \frac{1}{2}, 1, \frac{1}{2}, \frac{1}{3})$ & 1 & 4 &  240 & $\inn{24, 4}$  \rule[0ex]{0ex}{4ex} \\ [1ex]
			\hline
		\end{tabular} \caption{\label{G_{26} 4-Tuple Comparisons}Here $z = \zeta_{3}$. The subgroup of order 24 is the dicyclic group.}
	\end{table}
        	
	\subsubsection{$G_{27}$} \label{G_{27} 4-tups}
	In this group we found one inverse pair, and we provide the details for one type of the pair. 
	\begin{table}[h!]
		\centering
		\begin{tabular}{|c|c|c|} 
			\hline 
			\multicolumn{3}{|c|}{$G_{27}$ 4-tuples types} \rule[0pt]{0pt}{3ex}\\ 
			\hline
			Type & Tuple eigenvalues & Residues of inverse product  \rule[0pt]{0pt}{3ex}\\ [0.5ex] 
			\hline
			A & $(-1,-1, -1, -1)$ & $\{\frac{5}{6}, \frac{5}{6}, \frac{1}{3}\}$  \rule[0pt]{0pt}{3ex} \\ [0.5ex] 
			\hline  
		\end{tabular} \caption{\label{G_{27} 4-Tuple Types} Here $z = \zeta_{15}$.}
	\end{table}
	
	Here is an exemplar of type A, where $z = \zeta_{15}$, $u = z^4 + z$, $v = z^7-z^3+z^2-1$: 
	\begin{align*}
		A_1 &= \begin{pmatrix}u&u-1&-1\\ -u-1&u&-v-u\\ 0&0&1\end{pmatrix}, \quad A_2 = \begin{pmatrix}-1&0&0\\ 1&1&0\\ -u&0&1\end{pmatrix}, \quad A_3 = \begin{pmatrix}u+2v&v+z^{10}&-u+z^5\left(z^5-1\right)\\ 0&1&0\\ -u-1&-v-u-1&-2v-u\end{pmatrix}, \\ A_4 &= \begin{pmatrix}-v-u-z^5&-v-z^5&u+1\\ v+u+z^5&v&-u+z^{10}\\ -v-u+1&-v-z^{10}&u-z^{10}\end{pmatrix}.
	\end{align*}
	
	The middle convolution details can be found in Table~\ref{G_{27} Nice 4-Tuples}. In this case we could not compute the orbit size since it was too long. Similar to $G_{23}$, we were able to compute some signatures for tuples in the orbit, and we found a match with a signature listed in Tykhyy's classification.
	\begin{table}
		\centering
		\begin{tabular}{|c|c|c|c|c|c|c|c|} 
			\hline
			Type & $\lambda$ &  $\theta_{1, 2, 3, 4, 5}$ & $\sigma_{12, 23, 34, 45, 51, 13, 24}$   & \begin{sideways} Tykhyy $\#$ \end{sideways}  & \begin{sideways} Their orbit size \end{sideways} & \begin{sideways} Our orbit size  \end{sideways} & \begin{sideways} $\SL_2(\C)$ subgroup \vspace{0.5ex} \end{sideways} \rule[0pt]{0pt}{4ex}\\ [0.5ex] 
			\hline 
			\multicolumn{8}{|c|}{$G_{27}$ type comparisons} \rule[0pt]{0pt}{3ex}\\ 
			\hline
			A& $-z^5$ & $(\frac{1}{3}, \frac{1}{3}, \frac{1}{3}, \frac{1}{3}, \frac{1}{3})$ & $(0, \frac{1}{5}, \frac{1}{3}, \frac{1}{3}, \frac{1}{5}, \frac{3}{5}, \frac{1}{3})$  & 87 &  432 &  - & $\inn{120, 5}$  \rule[0pt]{0pt}{4ex} \\ [1ex] 
			\hline 
		\end{tabular} \caption{\label{G_{27} 4-Tuple Comparisons}Here $z = \zeta_{15}$. The subgroup of order 120 is $\SL(2, 5)$.}
	\end{table}
	
	\clearpage
	\subsection{Nice 4-tuples in rank 4 groups}
	In this case we are looking for nice tuples of four reflections generating the whole group such that their inverse product has a nontrivial eigenvalue of $T-2 = 2$. Of the five rank 4 reflection groups, we found nice 4-tuples in $G_{28}$, $G_{30}$, and $G_{32}$. In the case of $G_{32}$, the field containing all the eigenvalues would be $\Q(\zeta_{360})$, which is much too large to perform the computations over. For $G_{32}$, in Section~\ref{G32 4-tups} we justify why it is enough to search for such tuples over $\Q(\zeta_{12})$. 
	
	After we take the middle convolution, multiply with a suitable character, and adjoin the inverse product, our tuples become $\SL_2(\C)$ local systems on the 5-punctured sphere. We found that in some cases, one of the five matrices in the local system ends up being the identity, which means our tuple comes from some local system on the 4-punctured sphere. In the case where we have four nontrivial matrices, we can compute the $\omega$-parameters in the notation of Lisovyy--Tykhyy \cite{T-L} and identify the corresponding $\mathcal{M}^{(4)}$ orbit in Tykhyy's paper \cite{Tyk}. When all five of the matrices are nontrivial, we can only compare with Tykhyy's $\mathcal{M}^{(5)}$ orbit signatures from \cite{Tyk}.
	
	\subsubsection{$G_{28}$} \label{G_{28} 4-tups}
	The group $G_{28}$ of order 1152 is also denoted $W(F_4)$ with degrees $\deg(G_{28}) = \{2, 6, 8, 12\}$; see \cite[Table D.3]{ld2009}. 
	All reflections in $G_{28}$ have nontrivial eigenvalue $-1$ and form two conjugacy classes of reflections. We only found one type of nice 4-tuples in this group, with two possible choices of $\lambda$-parameters for middle convolution. 
	\begin{table}[H]
		\centering
		\begin{tabular}{|c|c|c|} 
			\hline 
			\multicolumn{3}{|c|}{$G_{28}$ types} \rule[0pt]{0pt}{3ex}\\ 
			\hline
			Type & Tuple eigenvalues & Residues of inverse product  \rule[0pt]{0pt}{3ex}\\ [0.5ex] 
			\hline
			A & $(-1, -1, -1, -1)$ & $\{\frac{1}{6}, \frac{1}{6}, \frac{5}{6}, \frac{5}{6}\}$  \rule[0pt]{0pt}{3ex} \\ [0.5ex] 
			\hline 
		\end{tabular} \caption{\label{G_{28} Types}}
	\end{table}

	\begin{table}[H]
		\centering
		\begin{tabular}{|c|c|c|c|c|c|c|c|c|} 
			\hline
			$\lambda$ & $\omega_{X, Y, Z, 4}$ & \begin{sideways} Lisovyy--Tykhyy $\#$ \end{sideways} & $\theta_{1, 2, 3, 4}$ & $\sigma_{12, 23, 13, 24}$   & \begin{sideways} Tykhyy $\#$ \end{sideways}  & \begin{sideways} Their orbit size \end{sideways} & \begin{sideways} Our orbit size  \end{sideways} & \begin{sideways} $\SL_2(\C)$ subgroup \vspace{0.5ex} \end{sideways} \rule[0pt]{0pt}{4ex}\\ [0.5ex] 
			\hline 
			\multicolumn{9}{|c|}{$G_{28}$ type A comparisons} \rule[0pt]{0pt}{3ex}\\ 
			\hline
			$z^4$ & $(2, 2, 2, -1)$ & Type III, $\omega = 2$ & $(\frac{2}{3}, \frac{1}{3}, \frac{1}{3}, \frac{2}{3})$ & $(\frac{1}{2}, \frac{1}{3}, \frac{1}{2}, 1)$  & 5 &  3 & 9 & $\inn{24, 3}$ \rule[0pt]{0pt}{4ex} \\ [1ex] \hline
		\end{tabular} \caption{\label{G_{28} Type Comparisons}Here $z = \zeta_{24}$, and the subgroup of order 24 is $\SL(2, 3)$.}
	\end{table}
        
	An exemplar for such a tuple of type A is 
	\begin{align*}
		A_1 = \begin{pmatrix}2&3&4&2\\ -1&-2&-4&-2\\ 0&0&1&0\\ 0&0&0&1\end{pmatrix}, \quad A_2 = \begin{pmatrix}1&2&2&0\\ 0&-1&-2&0\\ 0&0&1&0\\ 0&1&1&1\end{pmatrix}, \quad A_3 = \begin{pmatrix}-1&-2&-4&-2\\ 2&3&4&2\\ -1&-1&-1&-1\\ 0&0&0&1\end{pmatrix}, \quad A_4 = \begin{pmatrix}-1&0&0&0\\ 1&1&0&0\\ 0&0&1&0\\ 0&0&0&1\end{pmatrix}.
	\end{align*}	
	A different choice of exemplar does not produce a new orbit, and the two orbits computed above are the same, so we only get one new orbit from this group. The middle convolution details can be found in Table~\ref{G_{28} Nice 4-Tuples}. After putting the matrices in $\SL_2(\C)$, we see that the inverse product of the middle convolution of this tuple is the identity, so this tuple comes from a local system on the 4-punctured sphere. As in the case of 3-tuples, we provide the details comparing with the results of both Lisovyy--Tykhyy and Tykhyy \cite{T-L,Tyk}  in Table~\ref{G_{28} Type Comparisons}. The $\omega$-parameters appear to correspond to Orbit 31, but in fact this orbit is of Type III, as described in \cite[Lemma~39]{T-L}.

	\subsubsection{$G_{30}$} \label{G_{30} 4-tups}
	The group $G_{30}$ of order 14400 is also denoted $W(H_4)$ with degrees $\deg(G_{30}) = \{2, 12, 20, 30\}$; see  \cite[Table D.3]{ld2009}.
	All reflections in $G_{30}$ have nontrivial eigenvalue $-1$ and are conjugate to each other. Looking at the degrees, we extend the base field to $\Q(\zeta_{60})$. We found three distinct types of nice 4-tuples; the details are in Table~\ref{G_{30} Types}.
	\begin{table}[H]
		\centering
		\begin{tabular}{|c|c|c|} 
			\hline 
			\multicolumn{3}{|c|}{$G_{30}$ types} \rule[0pt]{0pt}{3ex}\\ \hline
			Type & Tuple eigenvalues & Residues of inverse product  \rule[0pt]{0pt}{3ex}\\ [0.5ex] 
			\hline
			A & $(-1, -1, -1, -1)$ & $\{\frac{3}{10},\frac{3}{10}, \frac{7}{10}, \frac{7}{10}\}$  \rule[0pt]{0pt}{3ex} \\ [0.5ex] 
			\hline 
			B & $(-1, -1, -1,-1)$ & $\left\{ \frac{1}{6}, \frac{1}{6},  \frac{5}{6}, \frac{5}{6}  \right\}$  \rule[0pt]{0pt}{3ex}\\ [0.5ex] \hline
			C & $(-1, -1, -1, -1)$ & $\{ \frac{1}{10}, \frac{1}{10}, \frac{9}{10}, \frac{9}{10}\}$  \rule[0pt]{0pt}{3ex}\\ [0.5ex]\hline
		\end{tabular} \caption{\label{G_{30} Types}Here $z = \zeta_{60}$.}
	\end{table}
	For all of the exemplars below, $ z = \zeta_{60}$ and $u = z^{14}-z^6-z^4$. 
	An exemplar of type A is given by 
	\begin{align*}
		A_1& = \begin{pmatrix}1&0&0&0\\ 0&1&0&0\\ 0&0&1&1\\ 0&0&0&-1\end{pmatrix}, \quad A_2 = \begin{pmatrix}1&0&0&0\\ 2u-1&2u-1&u-1&-1\\ -2u+1&-2u+2&-u+2&1\\ -u+1&-2u&-u&-u\end{pmatrix}, \\ A_3 &= \begin{pmatrix}u&u-1&u-1&-1\\ -2u+1&-2u+2&-2u+1&-u\\ 2u-1&2u-1&2u&u\\ -u&-u&-u&-u\end{pmatrix}, \quad A_4 = \begin{pmatrix}-u+1&-u+1&-u+1&-u\\ 0&1&0&0\\ u&u-1&u&u\\ -1&u&u&0\end{pmatrix}.
	\end{align*}	
	An exemplar of type B is given by
	\begin{align*}
		B_1 &= \begin{pmatrix}1&0&0&0\\ 0&1&0&0\\ 0&0&1&1\\ 0&0&0&-1\end{pmatrix}, \quad B_2 = \begin{pmatrix}0&-1&-1&-1\\ -u&-u+1&-u&-u\\ 0&0&1&0\\ u-1&u-1&u-1&u\end{pmatrix}, \\ B_3 &= \begin{pmatrix}1&-u&-u&-u\\ 0&0&-1&-1\\ 0&0&1&0\\ 0&-1&-1&0\end{pmatrix}, \quad B_4 = \begin{pmatrix}-2u+2&-3u+2&-2u+1&-u+1\\ 2u-1&3u-1&2u-1&u-1\\ -u&-u+1&-u+1&1\\ u&u-1&u&0\end{pmatrix}.
	\end{align*}
	An exemplar of type C is given by
	\begin{align*}
		C_1& = \begin{pmatrix}1&0&0&0\\ 0&1&0&0\\ 0&0&1&1\\ 0&0&0&-1\end{pmatrix}, \quad C_2 = \begin{pmatrix}1&0&0&0\\ -u+1&-2u+1&-u+1&-u\\ u-1&2u&u&u\\ 2u-1&2u-2&2u-1&u\end{pmatrix}, \\ C_3 &= \begin{pmatrix}u-1&2u-1&2u&u\\ -3u+1&-3u+3&-2u+2&-u+1\\ 3u-1&3u-2&2u-1&u-1\\ 0&0&0&1\end{pmatrix}, \quad C_4 = \begin{pmatrix}0&-1&-1&-1\\ -u&-u+1&-u&-u\\ 0&0&1&0\\ u-1&u-1&u-1&u\end{pmatrix}.
	\end{align*}
	The middle convolution details for each type is presented in Table~\ref{G_{30} Nice 4-Tuples}. For all three types, different exemplars do not produce new orbits, and since each type is their own inverse type, these are all the orbits from this group.  Further, the two orbits produced from type A are distinct, as are the two from type C. The two orbits from type B are actually the same. In total $G_{30}$ gives us five different orbits. 
	
	For all five orbits, once we put the matrices in $\SL_2(\C)$, we see that the last matrix is the identity; hence these actually correspond to orbits of four matrices multiplying to the identity. We compared these to the existing classification(s) by Lisovyy--Tykhyy and Tykhyy \cite{T-L,Tyk} as we did for the rank 3 complex reflection groups. Within each type, the two orbits are equivalent to the same orbit of $\mathcal{M}^{(4)}$ in Tykhyy's classification, so we only provide the details once in Table~\ref{G_{30} Type Comparisons}.
        
	\begin{table}[H]
		\centering
		\begin{tabular}{|c|c|c|c|c|c|c|c|c|c|} 
			\hline
			Type & $\lambda$ & $\omega_{X, Y, Z, 4}$ & \begin{sideways}Lisovyy--Tykhyy $\#$ \end{sideways} & $\theta_{1, 2, 3, 4}$ & $\sigma_{12, 23, 13, 24}$   & \begin{sideways} Tykhyy $\#$ \end{sideways}  & \begin{sideways} Their orbit size \end{sideways} & \begin{sideways} Our orbit size  \end{sideways} & \begin{sideways} $\SL_2(\C)$ subgroup  \vspace{0.5ex}\end{sideways} \rule[0pt]{0pt}{4ex}\\ [0.5ex] 
			\hline 
			\multicolumn{10}{|c|}{$G_{30}$ type comparisons} \rule[0pt]{0pt}{3ex}\\ 
			\hline
			A & $z^{18}$ & $(3 + \sqrt{5}, 3+ \sqrt{5}, 3+ \sqrt{5}, - \frac{7\sqrt{5} +11}{2})$ & $17$ & $(\frac{4}{5}, \frac{4}{5}, \frac{4}{5}, \frac{4}{5})$ & $(\frac{1}{5}, 0, \frac{1}{3}, \frac{1}{3})$  & 59 &  10 & 50 & $\inn{120, 5}$ \rule[0pt]{0pt}{4ex} \\ [1ex] 
			\hline 
			B& $z^{10}$ & \small{$(2, 2, 2, -1)$} & 31 & $(\frac{2}{3}, \frac{2}{3}, \frac{2}{3}, \frac{2}{3})$ & $(0,\frac{1}{5}, \frac{3}{5}, \frac{3}{5})$ & 96  & 18 & 7 & $\inn{120, 5}$ \rule[0pt]{0pt}{4ex}\\ [1ex] \hline 
			C & $z^6$ & \small{$(3 -\sqrt{5}, 3 -\sqrt{5}, 3-\sqrt{5}, \frac{7\sqrt{5} -11}{2})$} & 16 & $(\frac{3}{5}, \frac{2}{5}, \frac{3}{5}, \frac{2}{5})$ & $(\frac{2}{5}, \frac{2}{5}, \frac{1}{3}, \frac{3}{5})$ & 57 & 10 & 50 & $\inn{120, 5}$ \rule[0ex]{0ex}{4ex} \\ [1ex]
			\hline
		\end{tabular} \caption{\label{G_{30} Type Comparisons}Here $z = \zeta_{60}$, and the group $\inn{120, 5}$ corresponds to $\SL(2, 5)$.}
	\end{table}

	\subsubsection{$G_{32}$} \label{G32 4-tups}
	The group $G_{32}$ of order 155520 is also denoted  $W(L_4)$; see \cite[Table D.3]{ld2009}.
	In this group the nontrivial eigenvalues of reflections are in $\{\zeta_3, \zeta_3^2 \}$, corresponding to two conjugacy classes. 
	For $G_{32}$, we have $\deg(G_{32}) = \{12, 18, 24, 30\}$, so the eigenvalues of multiplicity 2 must be $\supth{12}$ roots of unity. Expanding the field to $\Q(\zeta_{12})$, we found five distinct eigenvalue sets for the inverse product. Note that to ensure we found all examples, we would have to run the code over $\Q(\zeta_{360})$ since there may be matrices with one or two eigenvalues other than the $\supth{12}$ roots which would not be found; however, we claim it is enough to check over $\Q(\zeta_{12})$.
        
	\begin{lemma}\label{lem5.2}
		If $A = A_1A_2A_3A_4$, with $[A_1, A_2, A_3, A_4]$ a nice tuple in $G_{32}$ so that $A$ has a nontrivial eigenvalue of multiplicity $2$, then the other two eigenvalues of $A$ are both $\supth{m}$ roots of unity $($not necessarily primitive$)$ for the same $m \in \{12, 18, 24, 30\}$.
	\end{lemma}
	\begin{proof}
	  Let the eigenvalues of $A$ be $\{\lambda, \lambda, \lambda_1, \lambda_2\}$. Taking residues, we write
          $$ \Res(\lambda) = \frac{x}{12}, \quad \Res(\lambda_1) = \frac{p_1}{u_1\cdot 6}, \quad \Res(\lambda_2) = \frac{p_2}{u_2 \cdot 6}, $$
          where $u_i \in  \{2, 3, 4, 5\}$, which we know must be the case by the degrees of $G_{32}$. Suppose the $\lambda_i$ are not $\supth{24}$ roots of unity. So $u_i$ must be 3 or 5 and $(p_i, u_i) = 1$. 
		
		Using that $\det(A_i)$ is a $\suprd{3}$ root of unity, we know that $\det(A) \in \{1, \zeta_3, \zeta_3^2\}$. This tells us that $\lambda_1\lambda_2 = (\lambda^2)^{-1} \det(A)$, and so $\lambda_1\lambda_2$ must be a $\supth{6}$ root of unity since $\lambda^2$ is a $\supth{6}$ root.  
		In terms of the residues, this means 
		\begin{align*}
			\frac{p_1}{u_1 \cdot 6} + \frac{p_2}{u_2 \cdot 6} = \frac{u_2p_1 + u_1p_2}{6 \cdot u_1u_2} = \frac{y}{6}
		\end{align*} for some $1\leq y \leq 6$.
		
		This equation tells us that $u_1 \mid u_2 p_1$, and since $\gcd(p_1, u_1) = 1$, we know $u_1 \mid u_2$. A similar argument shows $u_2 \mid u_1$, so $u_1 = u_2$ if one of them is 3 or 5.
		
		Now suppose  that $\Res(\lambda_1) = \frac{p_1}{4 \cdot 6}$ is a primitive $\supth{24}$ root of unity, \textit{i.e.}~$u_1 = 4$, so $(p_1, 24) = 1$. We know $u_2$ is not 3 or 5, so it can only be $2$ or $4$. If $u_2 = 4$ and $(p_2, 24) = 1$, then we are done since they are both primitive $\supth{24}$ roots. Else $\lambda_2$ is a $\supth{12}$ root of unity, which means $\lambda_1 = (\lambda^2)^{-1} \det(A) \lambda_2^{-1}$. Since the right-hand side is a product of $\supth{12}$ roots, $\lambda_1$ must also be a $\supth{12}$ root, completing the proof.
	\end{proof}
        
	In the notation of  Lemma~\ref{lem5.2}, we know that for nice 4-tuples, if the eigenvalues of the inverse product~$A$ are $\{\lambda, \lambda, \lambda_1, \lambda_2\}$, then $\lambda_1 + \lambda_2 = -2\lambda + \tr(A) \in \Q(\zeta_{12})$, and similarly $\lambda_1\lambda_2 = \lambda^2 \det(A) \in \Q(\zeta_{12})$. 
	With the result from Lemma~\ref{lem5.2}, we did a computer search in each of the fields $\Q(\zeta_{18})$, $\Q(\zeta_{24})$, $\Q(\zeta_{30})$ to find a pair of roots of unity $(s, t)$ other than $\supth{12}$ roots such that $s +t, s\cdot t \in \Q(\zeta_{12})$. For $\Q(\zeta_{18})$ and $\Q(\zeta_{30})$, no such pair could be found. In $\Q(\zeta_{24})$ one example that works is the pair $(\zeta_{24}, -\zeta_{24})$; however, no pairs are found if we further require that $s \neq -t$. This tells us the only cases we might be missing when searching over $\Q(\zeta_{12})$ are tuples whose inverse product has eigenvalues $\{\lambda, \lambda, \lambda_1, -\lambda_1\}$ with $\lambda$ a $\supth{12}$ root (not necessarily primitive) and $\lambda_1$ a primitive $\supth{24}$ root. In the notation of Lemma~\ref{lem5.2}, using that $\det(A)$ is a $\suprd{3}$ root, we know $-\lambda^2 \cdot \lambda_1^2$ is also a $\suprd{3}$ root. Note that if $\lambda_1$ is a primitive $\supth{24}$-root, then $-\lambda_1^2$ is a primitive $\supth{12}$ root. Again using a computer (or by hand), we can check that no primitive $\supth{12}$ root multiplied with the square of a $\supth{12}$-root can be a $\suprd{3}$ root. Thus, we conclude it is enough to search for tuples over $\Q(\zeta_{12})$.
	
	Searching for nice tuples of in $G_{32}$, we find three inverse pairs. We list only one from each pair, noting that type B is its own inverse.
	\begin{table}[H]
		\centering
		\begin{tabular}{|c|c|c|} 
			\hline 
			\multicolumn{3}{|c|}{$G_{32}$ types} \rule[0pt]{0pt}{3ex}\\
			\hline
			Type & Tuple eigenvalues & Residues of inverse product  \rule[0pt]{0pt}{3ex}\\ [0.5ex] 
			\hline
			A & $(w^2, w, w, w)$ & $\{\frac{5}{6}, \frac{5}{6}, \frac{1}{12}, \frac{7}{12} \}$  \rule[0pt]{0pt}{3ex} \\ [0.5ex] 
			\hline 
			B & $(w^2, w, w^2, w)$ & $\left\{ \frac{1}{6}, \frac{1}{6},  \frac{5}{6}, \frac{5}{6}  \right\}$  \rule[0pt]{0pt}{3ex}\\ [0.5ex] \hline
			C & $(w^2, w^2, w^2, w^2)$ & $\{ \frac{1}{12}, \frac{1}{12}, \frac{7}{12}, \frac{7}{12}\}$  \rule[0pt]{0pt}{3ex}\\ [0.5ex]\hline
		\end{tabular} \caption{\label{G_{32} Types}Here $w = \zeta_3$ denotes a $\suprd{3}$ root of unity. }
	\end{table}
	
	For all the exemplars listed below, $w = \zeta_3$. 
	An exemplar of type $A$ is given by 
	\begin{align*}
		A_1 &= \begin{pmatrix}1&0&0&0\\ -w^2&w&w+2&w\\ -w&w+2&-2w&1\\ 0&0&0&1\end{pmatrix}, \quad A_2 = \begin{pmatrix}1&0&0&0\\ 0&1&w^2&0\\ 0&0&w&0\\ 0&0&w^2&1\end{pmatrix}, \\ A_3 &= \begin{pmatrix}1&0&0&0\\ -w^2&w&-w^2&0\\ 0&0&1&0\\ w^2&-w+1&w^2&1\end{pmatrix}, \quad A_4 = \begin{pmatrix}0&w&w^2-1&-w\\ -w^2&2&-w+w^2&1\\ w&-w^2&w&w^2\\ 0&0&0&1\end{pmatrix}.
	\end{align*}	
	An exemplar of type $B$ is given by
	\begin{align*}
		B_1 &= \begin{pmatrix}-w&1&-w&0\\ 0&1&0&0\\ w&w^2&0&0\\ w^2&1&-w&1\end{pmatrix}, \quad B_2 = \begin{pmatrix}1&0&0&0\\ 0&1&w^2&0\\ 0&0&w&0\\ 0&0&w^2&1\end{pmatrix}, \\ B_3 &= \begin{pmatrix}1&w&-w^2&w\\ 0&-w&-1&w^2\\ 0&-1&-w^2&-1\\ 0&w^2&-1&-w\end{pmatrix}, \quad B_4 = \begin{pmatrix}0&-w&w^2&-w\\ 0&1&0&0\\ 1&w&w+2&w\\ -w^2&-1&w&0\end{pmatrix}.
	\end{align*}	
	An exemplar of type $C$ is given by 
	\begin{align*}
		C_1 &= \begin{pmatrix}1&0&0&0\\ 0&1&0&0\\ -1&-w+w^2&-1&w^2\\ w^2&-w+1&2w^2&-w+1\end{pmatrix}, \quad C_2 = \begin{pmatrix}0&w^2&w-1&w^2\\ -w&2&-w+w^2&1\\ w^2&-w&w^2&-w\\ 0&0&0&1\end{pmatrix}, \\ C_3 &= \begin{pmatrix}1&w&1&-w^2\\ 0&0&-w^2&w\\ 0&0&1&0\\ 0&-w&-1&-w\end{pmatrix},  \quad C_4 = \begin{pmatrix}1&w&0&0\\ 0&w^2&0&0\\ 0&w&1&0\\ 0&0&0&1\end{pmatrix}.
	\end{align*}	
	
	The middle convolution details can be found in Table~\ref{G_{32} Nice 4-Tuples}. Neither different exemplars nor the inverse types produce new orbits. The two orbits from type B are distinct, as are the two orbits from type C.
	Of the five middle convolutions computed in $G_{32}$, after putting the matrices into $\SL_2(\C)$, three of them (one from type~B and both orbits from type C) produce tuples where the fifth matrix is the identity. For these three, we compare with the classification(s) by Lisovyy--Tykhyy and Tykhyy \cite{T-L,Tyk} to see to which $\mathcal{M}^{(4)}$ orbits they correspond. The orbits from type C correspond to the same orbit in Tykhyy's classification, so we only provide the details for it once. 
	The tuples that correspond to nontrivial local systems on the 5-punctured sphere are those coming from type A and the other middle convolution from type B. For the type B tuple, after putting the matrices into $\SL_2(\C)$, we get that the product of the third and fourth matrices is not of finite order; hence we cannot compute its signature to compare with Tykhyy's table of $\mathcal{M}^{(5)}$.
	The matrices in this tuple $MC_{-z^2+1}(\mathbf{B})$ are 
	\begin{align*}
		\hat{B}_1 &= \begin{pmatrix}z^3&z\\ 0&-z^3\end{pmatrix}, \ \hat{B}_2 = \begin{pmatrix}z^3-z&0\\ -z^3+z&-z\end{pmatrix},\ 
		\hat{B}_3 = \begin{pmatrix}-z^3&-z\\ 0&z^3\end{pmatrix},\ 
		\hat{B}_4 = \begin{pmatrix}z&0\\ -z&-z^3+z\end{pmatrix},\ 
		\hat{B}_5 = \begin{pmatrix}z^2-1&0\\ 0&-z^2\end{pmatrix}.
	\end{align*}
	\begin{table}[H]
		\centering
		\begin{tabular}{|c|c|c|c|c|c|c|c|c|c|} 
			\hline
			Type & $\lambda$ & $\omega_{X, Y, Z, 4}$ & \begin{sideways} Lisovyy--Tykhyy $\#$ \end{sideways} & $\theta_{1, 2, 3, 4}$ & $\sigma_{12, 23, 13, 24}$   & \begin{sideways} Tykhyy $\#$ \end{sideways}  & \begin{sideways} Their orbit size \end{sideways} & \begin{sideways} Our orbit size  \end{sideways} & \begin{sideways} $\SL_2(\C)$ subgroup  \vspace{1ex} \end{sideways} \rule[0pt]{0pt}{4ex}\\ [0.5ex] 
			\hline 
			\multicolumn{10}{|c|}{$G_{32}$ comparing to $\mathcal{M}^{(4)}$} \rule[0pt]{0pt}{3ex}\\ 
			\hline
			B & $z^2$ & $(0, 3, 0, -2)$ & Type II, $X' = 1, X'' = 2$ & $(\frac{5}{6}, \frac{1}{2}, \frac{1}{6}, \frac{1}{2})$ & $(\frac{1}{2}, \frac{1}{2}, 1, \frac{2}{3})$  & 3 &  2 & 60 & $\inn{24, 4}$ \rule[0pt]{0pt}{4ex} \\ [1ex] 
			\hline 
			C & $z$ & \small{$(4, 4,4, -8)$} & Type IV, $\omega = 4$ & $(\frac{1}{4}, \frac{1}{4}, \frac{1}{4}, \frac{1}{4})$ & $(\frac{1}{3},\frac{1}{3}, \frac{1}{3}, 0)$ & 6  & 4 & 20 & $\inn{48, 28}$ \rule[0pt]{0pt}{4ex}\\ [1ex] \hline 
		\end{tabular} \caption{\label{G_{32} 5-Tuple length 4 Type Comparisons}Here $z = \zeta_{12}$. The subgroups $\inn{24, 4}$ and $\inn{48, 28}$ corresponds to the dicyclic group and the binary octahedral group, respectively.}
	\end{table}
	
	\begin{table}[H]
		\centering
		\begin{tabular}{|c|c|c|c|c|c|c|c|} 
			\hline
			Type & $\lambda$ & $\theta_{1, 2, 3, 4, 5}$ & $\sigma_{12, 23,34, 45, 51, 13, 24}$   & \begin{sideways} Tykhyy $\#$ \end{sideways}  & \begin{sideways} Their orbit size \end{sideways} & \begin{sideways} Our orbit size  \end{sideways} & \begin{sideways} $\SL_2(\C)$ subgroup \vspace{0.5ex} \end{sideways} \rule[0pt]{0pt}{4ex}\\ [0.5ex] 
			\hline 
			\multicolumn{8}{|c|}{$G_{32}$ comparing to $\mathcal{M}^{(5)}$} \rule[0pt]{0pt}{3ex}\\ 
			\hline
			A & $-z^2+1$ & $(\frac{1}{2}, \frac{5}{6}, \frac{1}{6}, \frac{1}{6}, \frac{1}{2})$ & $(\frac{1}{2}, 1, \frac{1}{3}, \frac{1}{2}, \frac{5}{6}, \frac{1}{2}, 1)$  & 1 &  4 & 40 & $\inn{24, 4}$ \rule[0pt]{0pt}{4ex} \\ [1ex] 
			\hline 
			B & $-z^2+1$ & -
			& -
			& -  & - & $\geq 350$ & 0 \rule[0pt]{0pt}{4ex}\\ [1ex] \hline 
		\end{tabular} \caption{\label{G_{32} 5-Tuple length 5 Type Comparisons}Here $z = \zeta_{12}$. The subgroup $\inn{24, 4}$ corresponds to the dicyclic group of order 24. }
	\end{table}	
	
	\subsection{Nice 5-tuples in rank 4 groups}
	In this case we are searching for tuples of five reflections generating the group such that their inverse product has a nontrivial eigenvalue of multiplicity 3. Searching all the rank 4 groups, we found that such tuples only existed in $G_{32}$. 
	\subsubsection{$G_{32}$}
	Since $\deg (G_{32}) = \{12, 18, 24, 30\}$, we know that any eigenvalue of multiplicity 3 must be a $\supth{6}$ root of unity. Using the trace, we then get that the $\supth{4}$ eigenvalue must also be a $\supth{6}$ root, so we do not need to expand the field to compute all the examples. 
	In $G_{32}$ there were too many possible tuples to perform an exhaustive search for all the nice tuples. We were able to find at least four types, which correspond to two inverse pairs, so we only provide the details for one type from each pair. 
	\begin{table}[H]
		\centering
		\begin{tabular}{|c|c|c|} 
			\hline 
			\multicolumn{3}{|c|}{$G_{32}$ 5-tuple types} \rule[0pt]{0pt}{3ex}\\
			\hline
			Type & Tuple eigenvalues & Residues of inverse product  \rule[0pt]{0pt}{3ex}\\ [0.5ex] 
			\hline
			A & $(w, w, -w-1, w, w)$ & $\{ \frac{1}{2},\frac{5}{6}, \frac{5}{6}, \frac{5}{6}\}$  \rule[0pt]{0pt}{3ex} \\ [0.5ex] 
			\hline 
			B & $(-w-1, w, -w-1, w, w)$ & $\left\{ \frac{1}{6}, \frac{5}{6},  \frac{5}{6}, \frac{5}{6}  \right\}$  \rule[0pt]{0pt}{3ex}\\ [0.5ex] \hline
		\end{tabular} \caption{\label{G_{32} 5-Tuple Types}Here $w = \zeta_3$.}
	\end{table}
	
	An exemplar of type A is provided below with $w = \zeta_3$: 
	\begin{align*}
		A_1 &= \begin{pmatrix}1&0&0&0\\ w+1&w&w+1&0\\ 0&0&1&0\\ -w-1&-w+1&-w-1&1\end{pmatrix}, \quad A_2 = \begin{pmatrix}1&0&0&0\\ 0&1&-w-1&-w\\ 0&0&w+1&-1\\ 0&0&w&0\end{pmatrix}, \quad A_3 = \begin{pmatrix}1&w&1&w+1\\ 0&0&w+1&w\\ 0&0&1&0\\ 0&-w&-1&-w\end{pmatrix}, \\ A_4 &= \begin{pmatrix}w&0&0&0\\ -w-1&1&0&0\\ 0&0&1&0\\ 0&0&0&1\end{pmatrix}, \quad A_5 = \begin{pmatrix}1&0&0&0\\ 1&2w+2&-w+1&w+1\\ -w-1&-w+1&-w-1&-w\\ 0&0&0&1\end{pmatrix}.
	\end{align*}
	An exemplar of type A is provided below with $w = \zeta_3$:
	\begin{align*}
		B_1 &= \begin{pmatrix}1&w&1&w\\ 0&0&w+1&-1\\ 0&w+1&-w+1&w+1\\ 0&-1&w+1&0\end{pmatrix}, \quad B_2 = \begin{pmatrix}w&0&0&0\\ -w-1&1&0&0\\ 0&0&1&0\\ 0&0&0&0\end{pmatrix}, \quad B_3 = \begin{pmatrix}1&w&1&0\\ 0&0&w+1&0\\ 0&1&-w&0\\ 0&w+1&-w&1\end{pmatrix}, \\
		B_4 &= \begin{pmatrix}1&0&0&0\\ 0&1&-w-1&0\\ 0&0&w&0\\ 0&0&-w-1&1\end{pmatrix}, \quad B_5 = \begin{pmatrix}1&w+1&-w&1\\ 0&w+1&1&w+1\\ 0&0&1&0\\ 0&-w-1&w&0\end{pmatrix}.
	\end{align*}
	The details of the middle convolution for both types is presented in Table~\ref{G_{32} Nice 5-Tuples}. The $\SL_2(\C)$ subgroups generated by both tuple types is the dicyclic group of order 24. Inspecting the six $\SL_2(\C)$ matrices we get for both of these types, we see that both tuples satisfy the description of ``TYPE'' A of $\mathcal{M}^{(6)}$ belonging to the dihedral group in Tykhyy's classification; see \cite{Tyk}. In these cases we also could not check if different exemplars or inverses produce different orbits as we could not generate the list of all nice tuples. 

	\newpage
	\appendix 
	\appendixpage
	The notation for all the appendices is the same.
	The ``$\xi$'' column gives the residue of the root of unity in the ``$\lambda$'' column used to compute middle convolution. The ``Character'' column is the character we tensor our tuple with to get matrices in $\SL_2(\C)$. The column ``O.~size'' is the size of the $\SL_2(\C)$-character variety orbit of the tuple (\textit{i.e.}~after it is tensored with the character). In these tables ``S.~size'' denotes the size of the $\GL_2(\C)$ subgroup generated by the $2 \times 2$  matrices  we get after middle convolution but \textit{before} we tensor with the character. The $\SL_2(\C)$-subgroup size after tensoring with the character is provided in each section for each group as part of the comparison tables. For infinite subgroups, the order is listed as ``0'', following the notation used in Magma.
        
	\section{Middle convolution tables for nice 3-tuples} \label{appendix A}
	\begin{center}
	  \begin{tabular}{|Sc|Sc|Sc|Sc|Sc|Sc|Sc|Sc|} 
			\hline 
			\multicolumn{8}{|c|}{$G_{23}$ middle convolutions}	\rule[0pt]{0pt}{3ex}\\
			\hline 
			$\xi$ & $\lambda$ & $M_1$ & $M_2$ & $M_3$ & Character & O.~size & S.~size \rule[0pt]{0pt}{4ex}\\ [0.5ex] 
			\hline
			\multicolumn{8}{|c|}{$G_{23}$ type A middle convolutions} \rule[0pt]{0pt}{3ex}\\ 
			\hline
			$\frac{1}{2}$ & $-1$ & $\begin{pmatrix}1 & 0 \\ -q^3-q^2 + 1 & 1\end{pmatrix}$ & $\begin{pmatrix}1 & -1 \\ 0 & 1\end{pmatrix}$ & $\begin{pmatrix}1 & 0 \\1 & 1\end{pmatrix}$ & $(1, 1, 1, 1)$ & 40 & 0 \rule[0pt]{0pt}{4ex} \\ 
			\hline 
			$\frac{1}{10}$ & $-q^3$ & $\begin{pmatrix}q^3 & 0 \\ -q^3 & 1 
			\end{pmatrix}$ & $\begin{pmatrix} q^3 & -1 \\ 0 & 1 \end{pmatrix}$ & $\begin{pmatrix}1 & 0 \\q^3 & q^3\end{pmatrix}$  & $(q, q, q, q^2)$ & 10 & 600 \rule[0pt]{0pt}{4ex}\\ \hline
			$\frac{9}{10}$ & $-q^2$ & $ \begin{pmatrix}q^2 & 0 \\ -q^2 & 1 
			\end{pmatrix}$ & $\begin{pmatrix} q^2 & -1 \\ 0 & 1 \end{pmatrix}$ & $\begin{pmatrix}1 & 0 \\q^2 & q^2\end{pmatrix}$ & $(q^4, q^4, q^4, q^3)$ & 10 & 600 \rule[0pt]{0pt}{4ex}\\ \hline
			\multicolumn{8}{|c|}{ $G_{23}$ type B middle convolutions} \rule[0pt]{0pt}{3ex}\\ 
			\hline
			$\frac{1}{2}$ & $-1$ & ${\begin{pmatrix}1&0\\ 1&1\end{pmatrix}}$ & ${\begin{pmatrix}1&-1\\ 0&1\end{pmatrix}}$ & ${\begin{pmatrix}1&0\\ q^3+q^2+2&1\end{pmatrix}}$ & $(1, 1, 1, 1)$ & 40 & 0 \rule[0pt]{0pt}{4ex} \\ 
			\hline 
			$\frac{3}{10}$ & $-q^4$ & ${\begin{pmatrix}q^4&0\\ -q^4+q^3+1&1\end{pmatrix}}$ & ${\begin{pmatrix}q^4&-1\\ 0&1\end{pmatrix}}$ & ${\begin{pmatrix}1&0\\ q^4-q^3-1&q^4\end{pmatrix}}$ & $(q^3, q^3, q^3, q)$ & 10 & 600 \rule[0pt]{0pt}{4ex}\\ \hline
			$\frac{7}{10}$ & $-q$ & ${\begin{pmatrix}q&0\\ q^2-q+1&1\end{pmatrix}}$ & ${\begin{pmatrix}q&-1\\ 0&1\end{pmatrix}}$ & ${\begin{pmatrix}1&0\\ -q^2+q-1&q\end{pmatrix}}$ & $(q^2, q^2, q^2, q^4)$ & 10 & 600 \rule[0pt]{0pt}{4ex}\\ \hline
			\multicolumn{8}{|c|}{ $G_{23}$ type C middle convolutions} \rule[0pt]{0pt}{3ex}\\ 
			\hline
			$\frac{1}{2}$ & $-1$ & ${\begin{pmatrix}u-1&3u-5\\ 1&-u+3\end{pmatrix}}$ & ${\begin{pmatrix}1&-1\\ 0&1\end{pmatrix}}$ & ${\begin{pmatrix}1&0\\ u+1&1\end{pmatrix}}$ & $(1, 1, 1, 1)$ & 72 & 0 \rule[0pt]{0pt}{4ex} \\ 
			\hline 
			$\frac{1}{6}$ & $z^5$ & $ {\begin{pmatrix}-u-v-z^{10}&-z^5\\ u+2v+z^{10}&u+v\end{pmatrix}}$ & $ {\begin{pmatrix}-z^5&-1\\ 0&1\end{pmatrix}}$ & $ {\begin{pmatrix}1&0\\ -u-v-z^5&-z^5\end{pmatrix}}$ & $ {-(z^5, z^5, z^5, -1)}$ & 18 & 360 \rule[0pt]{0pt}{4ex}\\ \hline
			$\frac{5}{6}$ & $-z^{10}$ & $ {\begin{pmatrix}z^5+v&z^{10}\\ -u-2v-z^5&v\end{pmatrix}}$ & $ {\begin{pmatrix}z^{10}&-1\\ 0&1\end{pmatrix}}$ & $ {\begin{pmatrix}1&0\\ v+z^{10}&z^{10}\end{pmatrix}}$ & ${(z^{10}, z^{10}, z^{10}, 1)}$ & 18 & 360 \rule[0pt]{0pt}{4ex}\\ \hline
		\end{tabular} \captionof{table}{Here $q = \zeta_5$, $z = \zeta_{30}$, $u = z^7-z^3-z^2+1$, and $v = -z^4+z$.}\label{G_{23} Nice 3-Tuples}
	\end{center}
		\medskip
		\begin{center}  
			\begin{tabular}{|Sc|Sc|Sc|Sc|Sc|Sc|Sc|Sc|}
				\hline 
				\multicolumn{8}{|c|}{$G_{24}$ middle convolutions}
				\rule[0pt]{0pt}{3ex}\\ \hline 
				$\xi$ & $\lambda$ & $M_1$ & $M_2$ & $M_3$ & Character & O.~size & S.~size \rule[0pt]{0pt}{3ex}\\ [0.5ex]
				\hline 
				\multicolumn{8}{|c|}{$G_{24}$ type A middle convolutions} \rule[0pt]{0pt}{3ex}\\ 
				\hline
				$\frac{1}{14}$ & $-z^4$ & $ {\begin{pmatrix}z^5+z^4+z+1&-z^5+z^4+z^3+1\\ \:-z^2&-z^5-z\end{pmatrix}}$ & $ {\begin{pmatrix}z^4&1\\ 0&1\end{pmatrix}}$ & $ {\begin{pmatrix}1&0\\ -z^4&z^4\end{pmatrix}}$ & $(z^5, z^5, z^5, z^6)$ & 28 & 0 \rule[0pt]{0pt}{4ex}\\ \hline
				$\frac{9}{14}$ & $-z$ & $ {\begin{pmatrix}-z^6-z^5-z^4&z^6-z^3+1\\ -z^4&-z^3-z^2\end{pmatrix}}$ & $ {\begin{pmatrix}z&1\\ 0&1\end{pmatrix}}$ & $ {\begin{pmatrix}1&0\\ -z&z\end{pmatrix}}$ & $(z^{3}, z^{3}, z^{3}, z^5)$ & 28 & 0 \rule[0pt]{0pt}{4ex}\\ \hline 	
				$\frac{11}{14}$ & $-z^2$ & $ {\begin{pmatrix}-z^5-z^3-z&-z^6 +z^5 +z^2 +1\\ -z&-z^6-z^4\end{pmatrix}}$ & $ {\begin{pmatrix}z^2&1\\ 0&1\end{pmatrix}}$ & $ {\begin{pmatrix}1&0\\ -z^2&z^2\end{pmatrix}}$ & $-(z^6, z^6, z^6, z^3)$ & 28 & 0 \rule[0pt]{0pt}{4ex} \\ 
				\hline
			\end{tabular} \captionof{table}{Here $z = \zeta_{7}$.} \label{G_{24} Nice 3-Tuples}
		\end{center}
		
		\begin{table}[H]  \begin{adjustbox}{width=\textwidth,center}
			\centering
			\begin{tabular}{|Sc|Sc|Sc|Sc|Sc|Sc|Sc|Sc|} 
				\hline 
				\multicolumn{8}{|c|}{$G_{25}$ middle convolutions}	\rule[0pt]{0pt}{3ex}\\
				\hline 
				$\xi$ & $\lambda$ & $M_1$ & $M_2$ & $M_3$ & Character & O.~size & S.~size \rule[0pt]{0pt}{4ex}\\ [0.5ex] 
				\hline
				\multicolumn{8}{|c|}{$G_{25}$ type A middle convolutions} \rule[0pt]{0pt}{3ex}\\ 
				\hline
				$\frac{5}{6}$ & $-w$ & $\begin{pmatrix}1&-w\\ 0&-w^2\end{pmatrix}$ & $\begin{pmatrix}-w^2&w\\ 0&1\end{pmatrix}$ & $\begin{pmatrix}1&0\\ -1&-1\end{pmatrix}$ & $(iw^2, iw^2, i, iw^2)$ & 12 & 72 \rule[0pt]{0pt}{4ex} \\ 
				\hline 
				$\frac{1}{6}$ & $-w^2$ & $\begin{pmatrix}1&-w\\ 0&-1\end{pmatrix}$ & $\begin{pmatrix}-1&w\\ 0&1\end{pmatrix}$ & $\begin{pmatrix}1&0\\ -w&-w\end{pmatrix}$ & $(i, i, iw, iw^2)$ & 12 & 72 \rule[0pt]{0pt}{4ex}\\ \hline
				$\frac{2}{3}$ & $w^2$ & $ \begin{pmatrix}1&w\\ 0&1\end{pmatrix}$ & $\begin{pmatrix}1&w\\ 0&1\end{pmatrix}$ & $\begin{pmatrix}1&0\\ w&w\end{pmatrix}$ & $(1,1, w, w^2)$ & 24 & 0 \rule[0pt]{0pt}{4ex}\\ \hline
				\multicolumn{8}{|c|}{$G_{25}$ type B middle convolutions} \rule[0pt]{0pt}{3ex}\\ 
				\hline
				$\frac{1}{9}$ & $z^4$ & $ {\begin{pmatrix}-z^8&z^{14}+z^6\\ -z^{10}-z^2&z^{16}+z^8+1\end{pmatrix}}$ &  $ {\begin{pmatrix}-z^{10}&z^6\\ 0&1\end{pmatrix}}$  & $ {\begin{pmatrix}1&0\\ z^4&-z^{10}\end{pmatrix}}$ & $ {(-z^{10}, z^4, z^4, 1)}$ & 36 & 0 \rule[0pt]{0pt}{4ex} \\ 
				\hline 
				$\frac{7}{9}$ & $-z^{10}$ & $ {\begin{pmatrix}z^2&-z^8+z^6\\ z^{16}-z^{14}&z^4-z^2+1\end{pmatrix}}$ & $ {\begin{pmatrix}u&z^6\\ 0&1\end{pmatrix}}$ & $ {\begin{pmatrix}1&0\\ -z^{10}&u\end{pmatrix}}$ & $ {(u, -z^{10}, -z^{10}, 1)}$ & 36 & 0 \rule[0pt]{-5pt}{4ex}\\ \hline
				$\frac{4}{9}$ & $z^{16}$ &  $ {\begin{pmatrix}z^{14}&z^6+z^2\\ z^8+z^4&-z^{10}-z^{14}+1\end{pmatrix}}$ & $ {\begin{pmatrix}z^4&z^6\\ 0&1\end{pmatrix}}$ & $ {\begin{pmatrix}1&0\\ u&z^4\end{pmatrix}}$ & $ {(z^4, u, u, 1)}$ & 36 & 0 \rule[0pt]{-5pt}{4ex}\\ \hline
				\multicolumn{8}{|c|}{$G_{25}$ type C middle convolutions} \rule[0pt]{0pt}{3ex}\\ 
				\hline
				$\frac{11}{12}$ & $-z^9+z^3$ &  $\frac{1}{2}\cdot \begin{pmatrix}z^9+1& -z^9+1\\ -z^9+1&z^9+1\end{pmatrix}$ & $\begin{pmatrix}z^9&0\\ 0&1\end{pmatrix}$ & $\begin{pmatrix}1&0\\ 0&z^9\end{pmatrix}$ & $(q^{63}, q^{63}, q^{63}, q^{27})$ & 4 & 96 \rule[0pt]{0pt}{4ex} \\ 
				\hline 
				$\frac{5}{12}$ & $z^{9}-z^3$ & $\frac{1}{2} \cdot \begin{pmatrix}-z^9+1&z^9+1\\ z^9+1&-z^9+1\end{pmatrix}$ & $\begin{pmatrix}-z^9&0\\ 0&1\end{pmatrix}$ & $\begin{pmatrix}1&0\\ 0&-z^9\end{pmatrix}$ & $(q^9, q^9, q^9, q^{45})$ & 4 & 96 \rule[0pt]{0pt}{4ex}\\ \hline
				$\frac{2}{3}$ & $-z^6$ &  $\begin{pmatrix}1&z^6\\ 0&1\end{pmatrix}$ & $\begin{pmatrix}1&0\\ z^6-1&1\end{pmatrix}$ & $\begin{pmatrix}1&0\\ z^6-1&1\end{pmatrix}$ & $(1, 1, 1, 1)$ & 16 & 0 \rule[0pt]{0pt}{4ex}\\ \hline
				\multicolumn{8}{|c|}{$G_{25}$ Type D middle convolutions} \rule[0pt]{0pt}{3ex}\\ 
				\hline
				$\frac{1}{3}$ & $w$ &  $\begin{pmatrix}-w&1\\ -w^2&0\end{pmatrix}$ & $\begin{pmatrix}w^2&w^2\\ 0&1\end{pmatrix}$ & $\begin{pmatrix}1&0\\ -1&w^2\end{pmatrix}$ & $(w^2, w^2, w^2, 1)$ & 4 & 24 \rule[0pt]{0pt}{4ex} \\ 
				\hline 
				$\frac{5}{6}$ & $-w$ & $ (-w^2)$ & $(-w^2)$ & $(-w^2)$ & - & - & - \rule[0pt]{0pt}{4ex}\\ [0.5ex]   \hline
			\end{tabular} \end{adjustbox}\caption{\label{G_{25} Nice 3-Tuples}Here $w = \zeta_3$, $z = \zeta_{36}$, and $u = z^{16} = z^{10}-z^4$, $z^{14} = z^8-z^2$ and $q = \zeta_{72}$.}
		\end{table}
		
		\begin{table}
			\centering \begin{adjustbox}{width=\textwidth,center}
			\begin{tabular}{|Sc|Sc|Sc|Sc|Sc|Sc|Sc|Sc|} 
				\hline 
				\multicolumn{8}{|c|}{$G_{26}$ middle convolutions}	
				\rule[0pt]{0pt}{3ex}\\
				\hline 
				$\xi$ & $\lambda$ & $M_1$ & $M_2$ & $M_3$ & Character & O.~size & S.~size \rule[0pt]{0pt}{4ex}\\ [0.5ex] 
				\hline
				\multicolumn{8}{|c|}{$G_{26}$ type A middle convolutions} \rule[0pt]{0pt}{3ex}\\ 
				\hline
				$\frac{1}{12}$ & $z^3$ & $\begin{pmatrix}-z^9&z^9+1\\ -u&u+1\end{pmatrix}$ & $\begin{pmatrix}-z^9&1\\ 0&1\end{pmatrix}$ & $\begin{pmatrix}1&0\\ z^9+z^3&-z^3\end{pmatrix}$ & $(q^{21}, q^9, q^{15}, -q^{63})$  &  24 & 0 \rule[0pt]{0pt}{4ex} \\ 
				\hline 
				$\frac{7}{12}$ & $-z^3$ & $\begin{pmatrix}z^9&-z^9+1\\ u&-u+1\end{pmatrix}$ & $\begin{pmatrix}z^9&1\\ 0&1\end{pmatrix}$ & $\begin{pmatrix}1&0\\ -z^9-z^3&z^3\end{pmatrix}$ & $(q^3, q^{27}, q^{33}, q^9)$ & 24 & 0 \rule[0pt]{0pt}{4ex}\\ \hline
				$\frac{5}{6}$ & $-z^6+1$ & $\begin{pmatrix}1&0\\ -z^6+1&z^6\end{pmatrix}$ & $\begin{pmatrix}-1&1\\ 0&1\end{pmatrix}$ & $\begin{pmatrix}1&0\\ -z^6+2&z^6-1\end{pmatrix}$ & $(z^{15}, z^9, z^6-1, 1)$ & 24 & 72 \rule[0pt]{0pt}{4ex}\\ \hline
				\multicolumn{8}{|c|}{$G_{26}$ type B middle convolutions} \rule[0pt]{0pt}{3ex}\\ 
				\hline
				$\frac{5}{18}$ & $z^{10}$ & $\begin{pmatrix}-z^4&0\\ z^8+z^4+1&1\end{pmatrix}$ & $\begin{pmatrix}-z^4&-1\\ 0&1\end{pmatrix}$ & $\begin{pmatrix}1&0\\ -z^{10}-z^4&-z^{10}\end{pmatrix}$ & $(z^7, z^7, z^4, -1)$ & 36 & 0 \rule[0pt]{0pt}{4ex} \\ 
				\hline 
				$\frac{11}{18}$ & $-z^{4}$ & $\begin{pmatrix}-z^{16}&0\\ z^{16}-z^{14}+1&1\end{pmatrix}$ & $\begin{pmatrix}-z^{16}&-1\\ 0&1\end{pmatrix}$ & $\begin{pmatrix}1&0\\ -z^{16}+z^4&z^4\end{pmatrix}$ & $(z, z, z^{16}, -1)$ & 36 & 0 \rule[0pt]{0pt}{4ex}\\ \hline
				$\frac{17}{18}$ & $-z^{16}$ &  $\begin{pmatrix}z^{10}&0\\ -z^{10}-z^2+1&1\end{pmatrix}$ & $\begin{pmatrix}z^{10}&-1\\ 0&1\end{pmatrix}$ & $\begin{pmatrix}1&0\\ z^{16} + z^{10}&z^{16}\end{pmatrix}$ & $(z^{13}, z^{13}, z^{10}, 1)$ & 36 & 0 \rule[0pt]{0pt}{4ex}\\ \hline
				\multicolumn{8}{|c|}{$G_{26}$ type C middle convolutions} \rule[0pt]{0pt}{3ex}\\ 
				\hline
				$\frac{5}{6}$ & $-z^6+1$ &  $\begin{pmatrix}z^6&z^6\\ -2z^6+1&-z^6\end{pmatrix}$ & $\begin{pmatrix}-1&-1\\ 0&1\end{pmatrix}$ & $\begin{pmatrix}1&0\\ z^6-2&z^6-1\end{pmatrix}$ & $(i, i, w, w+1)$ & 12 & 18 \rule[0pt]{0pt}{4ex} \\ 
				\hline 
				$\frac{1}{6}$ & $z^6$ & $ (-z^6+1)$ & $(-z^6+1)$ & $(-z^6)$ & - & - & - \rule[0pt]{0pt}{4ex}\\ [0.5ex]   \hline
			\end{tabular} \end{adjustbox}\caption{\label{G_{26} Nice 3-Tuples}Here $z = \zeta_{36}$, $q = \zeta_{72}$, $w = \zeta_3$, and $u = 2z^9-z^3$.}
		\end{table}

		  \begin{sidewaystable}
			\begin{tabular}{|Sc|Sc|Sc|Sc|Sc|Sc|Sc|Sc|} 
				\hline 
				\multicolumn{8}{|c|}{$G_{27}$ type A middle convolutions} \rule[0pt]{0pt}{3ex}\\ 
				\hline
				$\xi$ & $\lambda$ & $\tilde{A}_1$ & $\tilde{A}_2$ & $\tilde{A}_3$ & Character & O.~size & S.~size \rule[0pt]{0pt}{4ex}\\ [0.5ex] 
				\hline
				$\frac{2}{60}$ & $z^2$ & $\begin{pmatrix}z^6u-z^2&2\left(z^{50}+z^4\right)+z^{22}+z^{18}-z^6\\ \left(u+1\right)\left(z^{12}+1\right)&-z^6u+1\end{pmatrix}$ & $\begin{pmatrix}-z^2&z^{50}\\ 0&1\end{pmatrix}$ & $\begin{pmatrix}1&0\\ z^{24}+2z^{12}+1&-z^2\end{pmatrix}$ & $(z^{14}, z^{14}, z^{14}, z^{18})$ & 60  & 0 \rule[0pt]{0pt}{4ex} \\ 
				\hline 
				$\frac{38}{60}$ & $-z^8$ & $\begin{pmatrix}z^{24}-z^{20}&\left(z^{28}+2z^4\right)u\:+2\left(z^8-z^2\right)\\ z^{14}+u-z^2-1&-z^{24}-u+z^2\end{pmatrix}$ & $\begin{pmatrix}z^8&-z^{20}\\ 0&1\end{pmatrix}$ & $\begin{pmatrix}1&0\\ -2z^{18}-z^6&z^8\end{pmatrix}$ & $(z^{26}, z^{26}, z^{26}, z^{42})$ &  60 & 0 \rule[0pt]{0pt}{4ex}\\ \hline
				$\frac{50}{60}$ & $-z^{20}$ & $\begin{pmatrix}-u-z^{20}&u\\ -u-z^{20}-1&u+1\end{pmatrix}$ & $\begin{pmatrix}-z^{20}&z^{20}\\ 0&1\end{pmatrix}$ & $\begin{pmatrix}1&0\\ -z^{24}+z^6+1&-z^{20}\end{pmatrix}$ & $(z^{20}, z^{20}, z^{20}, 1)$ & 60 & 360 \rule[0pt]{0pt}{4ex}\\ \hline
			\end{tabular} \caption{\label{G_{27} Type A 3-Tuples}Here $z = \zeta_{60}$ and $u = -z^{10}-z^8+z^2$.}
			\vspace{2ex}
			\centering
			\begin{tabular}{|Sc|Sc|Sc|Sc|Sc|Sc|Sc|Sc|} 	
				\hline 
				\multicolumn{8}{|c|}{$G_{27}$ type B middle convolutions} \rule[0pt]{0pt}{3ex}\\ 
				\hline
				$\xi$ & $\lambda$ & $\tilde{B}_1$ & $\tilde{B}_2$ & $\tilde{B}_3$ & Character & O.~size & S.~size \rule[0pt]{0pt}{4ex}\\ [0.5ex] 
				\hline
				$\frac{5}{60}$ & $z^{5}$ & $\frac{1}{3} \begin{pmatrix}z^5f_1-2z^{10}+1&f_1+2z^{15}+3z^{10}-z^5\\ f_1-g_1&-z^5f_1+2z^{10}-3z^5+2\end{pmatrix}$ & $\begin{pmatrix}-z^5&-z^{20}f_1-u+2\\ 0&1\end{pmatrix}$ & $\begin{pmatrix}1&0\\ z^{25}u+z^{15}&-z^5\end{pmatrix}$  & $(q^{25}, q^{25}, q^{25}, q^{45})$ & 96 & 0 \rule[0pt]{0pt}{4ex} \\ 
				\hline 
				$\frac{35}{60}$ & $-z^{5}$ & $\frac{1}{3}\begin{pmatrix}-z^5f_1-2z^{10}+1&f_1-2z^{15}+3z^{10}+z^5\\ f_1+g_1&z^5f_1+2z^{10}+3z^5+2\end{pmatrix}$ & $\begin{pmatrix}z^5&-z^{20}f_1-u+2\\ 0&1\end{pmatrix}$ & $\begin{pmatrix}1&0\\ z^{25}u-z^{15}&z^5\end{pmatrix}$ & $(q^{55}, q^{55}, q^{55}, q^{75})$ & 96 & 0 \rule[0pt]{-5pt}{4ex}\\ \hline
				$\frac{50}{60}$ & $-z^{20}$ &  $\begin{pmatrix}f_1-u-1+2z^{10}&-f_1+2u+1\\ f_1-u+2z^{10}&-f_1+u-z^{20}\end{pmatrix}$ & $\begin{pmatrix}z^{20}&-f_1+2u+2\\ 0&1\end{pmatrix}$ & $\begin{pmatrix}1&0\\ f_1-u+1&z^{20}\end{pmatrix}$ & $(z^{20}, z^{20}, z^{20}, 1)$  & 96 & 360 \rule[0pt]{-5pt}{4ex}\\ \hline
			\end{tabular} \caption{\label{G_{27} Type B 3-Tuples}Here $z = \zeta_{60}$, $q = \zeta_{120}$, $f_1 = -z^{14}-2z^{10}-z^8+z^6+z^4+z^2$, and $g_1 = 4z^{15}+3z^{13}-3z^{11}-3z^9-3z^7-5z^5+3z$.}
                  \end{sidewaystable}
                  \vspace{2ex}
                  
 \begin{sidewaystable}
			\begin{tabular}{|Sc|Sc|Sc|Sc|Sc|Sc|Sc|Sc|} 
				\hline 
				\multicolumn{8}{|c|}{$G_{27}$ type C middle convolutions} \rule[0pt]{0pt}{3ex}\\ 
				\hline
				$\xi$ & $\lambda$ & $\tilde{C}_1$ & $\tilde{C}_2$ & $\tilde{C}_3$ & Character & O.~size & S.~size \rule[0pt]{0pt}{4ex}\\ [0.5ex] 
				\hline
				$\frac{10}{60}$ & $z^{10}$ &   {$\begin{pmatrix}z^{50}v-z^{10}-1&z^{50}v-u+z^{10}\\ 2z^{50}v-3&-z^{50}v+2\end{pmatrix}$} & $\begin{pmatrix}-z^{10}&u+2\\ 0&1\end{pmatrix}$ &  {$\begin{pmatrix}1&0\\ \frac{1}{2}\left(3z^{20}v+u+2z^{10}+4\right)&-z^{10}\end{pmatrix}$} & $(z^{10}, z^{10}, z^{10}, -1)$ & 60 & 360 \rule[0pt]{0pt}{4ex} \\ 
				\hline 
				$\frac{34}{60}$ & $-z^{4}$ &  {$\begin{pmatrix}-z^8-z^6&z^{18}f_1+2z^{12}+4\\ z^{14}-z^{50}v+2z^{10}+u&u-z^{50}v+1\end{pmatrix}$} & $\begin{pmatrix}z^4&u+2\\ 0&1\end{pmatrix}$ &  {$\begin{pmatrix}1&0\\ \frac{1}{2}\left(z^4f_1+3z^{24}-z^{12}+z^6\right)&z^4\end{pmatrix}$} & $(z^{28}, z^{28}, z^{28}, z^{36})$ & 60 & 0 \rule[0pt]{0pt}{4ex}\\ \hline
				$\frac{46}{60}$ & $-z^{16}$ &   {$\begin{pmatrix}z^{14}-z^4+z^2&f_2\\ -v+z^{28}v+z^{10}-z^8&\left(z^{50}+z^{28}\right)v-u\end{pmatrix}$} & $\begin{pmatrix}z^{16}&u+2\\ 0&1\end{pmatrix}$ &  {$\begin{pmatrix}1&0\\ \frac{1}{2}\left(z^{22}+2z^{16}u+z^{10}-3z^6\right)&z^{16}\end{pmatrix}$} & $(z^{22}, z^{22}, z^{22}, z^{54})$ & 60 & 0 \rule[0pt]{0pt}{4ex}\\ \hline
			\end{tabular} \caption{\label{G_{27} Type C 3-Tuples}Here $z = \zeta_{60}$, $v = z^{16}+z^4$, $u = -z^{10}-z^8+z^2$, $f_1 = -z^{14}-2z^{10}-z^8+z^6+z^4+z^2$, and $f_2 = -4z^{14}-z^{12}+z^{10}+2z^{6}+3z^4-2z^2$.}
	          \end{sidewaystable}
                  
		  
		\section{Middle convolution tables for nice 4-tuples} \label{appendix B}
		In these sections the notation is the same as above.

		\begin{center}
			\begin{tabular}{|Sc|Sc|Sc|Sc|Sc|Sc|Sc|Sc|Sc|Sc|} 
				\hline 
				\multicolumn{10}{|c|}{$G_{23}$ 4-tuple middle convolutions} \rule[0pt]{0pt}{3ex}\\ 
				\hline
				Type & $\xi$ & $\lambda$ & $\tilde{A}_1$ & $\tilde{A}_2$ & $\tilde{A}_3$ & $\tilde{A}_4$ &  Character & O.~size & S.~size \rule[0pt]{0pt}{4ex}\\ [0.5ex] 
				\hline
				A & $\frac{1}{2}$ & $-1$ &  {$\begin{pmatrix}u&-u+1\\ u-1&-u+2\end{pmatrix}$} & $\begin{pmatrix}0&1\\ -1&2\end{pmatrix}$ & $\begin{pmatrix}1&u+2\\ 0&1\end{pmatrix}$ & $\begin{pmatrix}1&0\\ -1&1\end{pmatrix}$ & $(1, 1, 1, 1, 1)$ & - & 0 \rule[0pt]{0pt}{4ex} \\ 
				\hline 
			\end{tabular} \captionof{table}{\label{G_{23} Nice 4-Tuples}Here $q = \zeta_5$ and $u = q^3 +q^2$.}
		\end{center}
		\begin{center}
			\medskip
			\centering\begin{adjustbox}{width=\textwidth,center}
			\begin{tabular}{|Sc|Sc|Sc|Sc|Sc|Sc|Sc|Sc|Sc|Sc|} 
				\hline 
				\multicolumn{10}{|c|}{$G_{25}$ 4-tuple middle convolutions} \rule[0pt]{0pt}{3ex}\\ 
				\hline
				  Type & $\xi$ & $\lambda$ & $\tilde{M}_1$ & $\tilde{M}_2$ & $\tilde{M}_3$ & $\tilde{M}_4$ & Character & O.~size & S.~size \rule[0pt]{0pt}{4ex}\\ [0.5ex] 
				\hline
				A &  $\frac{1}{3}$ & $z$ & $\begin{pmatrix}0&-z\\ z&-z\end{pmatrix}$ & $\begin{pmatrix}1&1\\ 0&1\end{pmatrix}$ & $\begin{pmatrix}1&1\\ 0&1\end{pmatrix}$ & $\begin{pmatrix}1&0\\ -1&1\end{pmatrix}$ & $(z^2, 1, 1, 1, z)$  &  45 & 0 \rule[0pt]{0pt}{4ex} \\ 
				\hline 
				B &  $\frac{1}{3}$ & $z$ & $\begin{pmatrix}1&z^2\\ 0&1\end{pmatrix}$ & $\begin{pmatrix}1&0\\ -z&1\end{pmatrix}$ & $\begin{pmatrix}1&0\\ -z&1\end{pmatrix}$ & $\begin{pmatrix}1&0\\ -z&1\end{pmatrix}$ & $(1, 1, 1, 1, 1)$  & 45  &  0\rule[0pt]{0pt}{4ex} \\ 
				\hline 
				C & $\frac{1}{6}$ &  $z+1$ & $\begin{pmatrix}-z&0\\ -1&1\end{pmatrix}$ &   $\begin{pmatrix}-z&-1\\ z^2-1&z\end{pmatrix}$ &   $\begin{pmatrix}-1&z+1\\ 0&1\end{pmatrix}$ & $\begin{pmatrix}1&0\\ 1&-z\end{pmatrix}$ &  $(u^7, u^3, u^3, u^7, u^4)$  & 120  & 72 \rule[0pt]{0pt}{4ex} \\ 
				\hline 
			\end{tabular}\end{adjustbox} \captionof{table}{Here $z = \zeta_{3}$ and $u = \zeta_{12}$.} \label{G_{25} Nice 4-Tuples}
		\end{center}
		\begin{sidewaystable}
			\centering
			\begin{tabular}{|Sc|Sc|Sc|Sc|Sc|Sc|Sc|Sc|Sc|Sc|} 
				\hline 
				\multicolumn{10}{|c|}{$G_{26}$ 4-tuple middle convolutions } \rule[0pt]{0pt}{4ex}\\
				\hline
				  Type & $\xi$ & $\lambda$ & $\tilde{M}_1$ & $\tilde{M}_2$ & $\tilde{M}_3$ & $\tilde{M}_4$ & Character & O.~size & S.~size \rule[0pt]{0pt}{4ex} \\
				\hline
				A & $\frac{1}{6}$ & $z+1$ &   $\begin{pmatrix}z^2&0\\ z+2&1\end{pmatrix}$ & $\begin{pmatrix}-1&-1\\ 0&1\end{pmatrix}$ & $\begin{pmatrix}1&0\\ z^2&-z\end{pmatrix}$ & $\begin{pmatrix}1&0\\ z^2&-z\end{pmatrix}$ &   $(u^8, -u^3, u^7, u^7, u^5)$  & 120  & 72 \rule[0pt]{0pt}{4ex} \\ 
				\hline 
				B & $\frac{1}{6}$ & $z+1$ &   $\begin{pmatrix}-z&z+1\\ -z+1&z\end{pmatrix}$ & $\begin{pmatrix}1&0\\ z&-z\end{pmatrix}$ & $\begin{pmatrix}-1&-z\\ 0&1\end{pmatrix}$ &   $\begin{pmatrix}1&0\\ 2z+1&z^2\end{pmatrix}$ &   $(u^3, u^7, u^3, u^8, u^3)$  & 240  &  72 \rule[0pt]{0pt}{4ex} \\ 
				\hline 
			\end{tabular} \caption{\label{G_{26} Nice 4-Tuples}Here $z = \zeta_{3}$ and $u = \zeta_{12}$.}
			\bigskip
			\centering
			\begin{tabular}{|Sc|Sc|Sc|Sc|Sc|Sc|Sc|Sc|Sc|Sc|} 
				\hline 
				\multicolumn{10}{|c|}{$G_{27}$ 4-tuple middle convolutions} \rule[0pt]{0pt}{3ex}\\ 
				\hline
				  Type & $\xi$ & $\lambda$ & $\tilde{M}_1$ & $\tilde{M}_2$ & $\tilde{M}_3$ & $\tilde{M}_4$ & Character &   O.~size &   S.~size \rule[0pt]{0pt}{4ex}\\ [0.5ex] 
				\hline
				A & $\frac{5}{6}$ & $-z^5$ & $\begin{pmatrix}-f_2+1&\frac{1}{4}g_1\\ g_2&z^{5\:}+f_2\end{pmatrix}$ & $\begin{pmatrix}f_1+2&\frac{1}{2}f_2\\ g_3&g_4\end{pmatrix}$ & $\begin{pmatrix}z^5&\frac{1}{4}g_5\\ 0&1\end{pmatrix}$ & $\begin{pmatrix}1&0\\ g_6&z^5\end{pmatrix}$ & $(z^5, z^5, z^5, z^5, z^{10})$  & - & 360 \rule[0pt]{0pt}{4ex} \\ 
				\hline 
			\end{tabular} \caption{\label{G_{27} Nice 4-Tuples}Here $z = \zeta_{15}$ $f_1 = z^7-z^3+z^2-1$, $f_2 = z^5-z^4-z +2$, $g_1 = 2z^7-3z^5+3z^4-2z^3+2z^2+3z-2$, $g_2 = 2z^7-z^4-2z^3+2z^2-z+1$, $g_3 = -3z^7+3z^5-2z^4+3z^3-3z^2-2z^2-2z+3$, $g_4 = -z^7+z^5+z^3-z^2$, $g_5 =  -z^7-z^5+z^3-z^2-2$, $g_6 = -z^7+3z^5-z^4+z^3-z^2-z+2$.}
			\bigskip
			\centering
			\begin{tabular}{|Sc|Sc|Sc|Sc|Sc|Sc|Sc|Sc|Sc|} 
				\hline 
				\multicolumn{9}{|c|}{$G_{28}$ type A middle convolutions} \rule[0pt]{0pt}{3ex}\\ 
				\hline
				$\xi$ & $\lambda$ & $\tilde{A}_1$ & $\tilde{A}_2$ & $\tilde{A}_3$ & $\tilde{A}_4$ & Character & O.~size & S.~size \rule[0pt]{0pt}{4ex}\\ [0.5ex] 
				\hline
				$\frac{1}{6}$ & $z^4$ & $\begin{pmatrix}1&1\\ 0&-z^4\end{pmatrix}$ & $\begin{pmatrix}-z^4+2&-z^4+1\\ -2&-1\end{pmatrix}$ & $\begin{pmatrix}-z^4&-1\\ 0&1\end{pmatrix}$ & $\begin{pmatrix}1&0\\ -2z^4&-z^4\end{pmatrix}$ & $(-z^4, -z^4, -z^4, -z^4, (-z^4)^2)$  &  45 & 72 \rule[0pt]{0pt}{4ex} \\ 
				\hline 
				$\frac{5}{6}$ & $-z^4+1$ & $\begin{pmatrix}1&1\\ 0&z^8\end{pmatrix}$ & $\begin{pmatrix}z^4+1&z^4\\ -2&-1\end{pmatrix}$ & $\begin{pmatrix}z^8&-1\\ 0&1\end{pmatrix}$ & $\begin{pmatrix}1&0\\ 2z^8&z^8\end{pmatrix}$ & $(z^8, z^8, z^8, z^8,-z^4)$  &  45 & 72 \rule[0pt]{0pt}{4ex} \\ 
				\hline
			\end{tabular} \caption{\label{G_{28} Nice 4-Tuples}Here $z = \zeta_{24}$.}
		\end{sidewaystable}

		\begin{sidewaystable}
			\centering\begin{adjustbox}{width=\textwidth}
			\begin{tabular}{|Sc|Sc|Sc|Sc|Sc|Sc|Sc|Sc|Sc|} 
				\hline 
				$\xi$ & $\lambda$ & $M_1$ & $M_2$ & $M_3$ & $M_4$ & Character & O.~size & S.~size \rule[0pt]{0pt}{4ex}\\ [0.5ex] 
				\hline
				\multicolumn{9}{|c|}{$G_{30}$ type A middle convolutions} \rule[0pt]{0pt}{3ex}\\ \hline
				$\frac{3}{10}$ & $z^{18}$ &   $\begin{pmatrix}-z^6v&z^6\\ -z^{24}+z^{12}-1&-z^{24}\end{pmatrix}$ &   $\begin{pmatrix}-z^{18}+z^6&-z^{18}\\ v^2&-v\end{pmatrix}$ &$\begin{pmatrix}-z^{18}&-1\\ 0&1\end{pmatrix}$ & $\begin{pmatrix}1&0\\ -z^{18}+v&-z^{18}\end{pmatrix}$&  $(z^{6}, z^{6}, z^{6}, z^{6}, z^{36})$ &  50 & 600 \rule[0pt]{0pt}{4ex} \\ 
				\hline 
				$\frac{7}{10}$ & $-z^{12}$ &   $\begin{pmatrix}z^{12}-v&-z^{24}\\ -z^{18}+v&z^6\end{pmatrix}$ &   $\begin{pmatrix}-z^{24}+z^{12}&z^{12}\\ z^{24}-z^6v&z^{24}+1\end{pmatrix}$ & $\begin{pmatrix}z^{12}&-1\\ 0&1\end{pmatrix}$ &   $\begin{pmatrix}1&0\\ -z^{24}+z^{12}-1&z^{12}\end{pmatrix}$ & $(z^{24}, z^{24}, z^{24}, z^{24}, z^{24})$& 50 & 600 \rule[0pt]{0pt}{4ex}\\ \hline
				\multicolumn{9}{|c|}{$G_{30}$ type B middle convolutions} \rule[0pt]{0pt}{3ex}\\\hline
				$\frac{1}{6}$ & $z^{10}$ &  $\begin{pmatrix}-z^{20}-z^{16}&-z^{16}-z^{10}\\ z^{16}-1&z^{16}\end{pmatrix}$&   $\begin{pmatrix}-z^{10}&0\\ z^{10}&1\end{pmatrix}$ & $\begin{pmatrix}-z^{10}&u-1\\ 0&1\end{pmatrix}$ & $\begin{pmatrix}1&0\\ -z^{10}&-z^{10}\end{pmatrix}$ &   $-(z^{10}, z^{10}, z^{10}, z^{10}, -z^{20})$ & 90 & 360 \rule[0pt]{0pt}{4ex} \\ 
				\hline 
				$\frac{5}{6}$ &   $-z^{10}+1$ &   $\begin{pmatrix}y+z^{20}&y+z^{10}-2\\ -y&-y+1\end{pmatrix}$ & $\begin{pmatrix}z^{20}&0\\ -z^{20}&1\end{pmatrix}$ & $\begin{pmatrix}z^{20}&u-1\\ 0&1\end{pmatrix}$ & $\begin{pmatrix}1&0\\ z^{20}&z^{20}\end{pmatrix}$ &   $(z^{20}, z^{20}, z^{20}, z^{20}, -z^{10})$ & 90 & 360 \rule[0pt]{-5pt}{4ex}\\ \hline 
				\multicolumn{9}{|c|}{$G_{30}$ type C middle convolutions} \rule[0pt]{0pt}{3ex}\\
				\hline
				$\frac{1}{10}$ & $z^{6}$ &  $\begin{pmatrix}1&-u-1\\ 0&-z^6\end{pmatrix}$ &   $\begin{pmatrix}-v&-z^{12}+v\\ u&0\end{pmatrix}$ & $\begin{pmatrix}-z^6&u+1\\ 0&1\end{pmatrix}$ & $\begin{pmatrix}1&0\\ -z^{12}-1&-z^6\end{pmatrix}$ & $(z^{12}, z^{12}, z^{12}, z^{12}, z^{12})$ & 50 & 600 \rule[0pt]{0pt}{4ex} \\ 
				\hline 
				$\frac{9}{10}$ & $-z^{24}$ &   $\begin{pmatrix}1&-u-1\\ 0&z^{24}\end{pmatrix}$ & $\begin{pmatrix}z^{24}+1&z^6v\\ u&0\end{pmatrix}$ & $\begin{pmatrix}z^{24}&u+1\\ 0&1\end{pmatrix}$ & $\begin{pmatrix}1&0\\ u+z^{12}&z^{24}\end{pmatrix}$ & $(z^{18}, z^{18}, z^{18}, z^{18}, z^{48})$ & 50 & 600 \rule[0pt]{-5pt}{4ex}\\ \hline
			\end{tabular} \end{adjustbox}\caption{\label{G_{30} Nice 4-Tuples}Here $z$ denotes a $\supth{60}$ root of unity, $u = z^{14}-z^6-z^4$, and $v = z^{6} -1$ and $y = z^{10} +z^8-z^2$.}
                \end{sidewaystable}

                \begin{sidewaystable}
                        \centering
			\begin{tabular}{|Sc|Sc|Sc|Sc|Sc|Sc|Sc|Sc|Sc|} 
				\hline 
				$\xi$ & $\lambda$ & $M_1$ & $M_2$ & $M_3$ & $M_4$ & Character & O.~size & S.~size \rule[0pt]{0pt}{4ex}\\ [0.5ex] 
				\hline
				\multicolumn{9}{|c|}{$G_{32}$ type A middle convolutions} \rule[0pt]{0pt}{3ex}\\ \hline
				$\frac{5}{6}$ & $-z^{2} +1$ & $\begin{pmatrix}0&-1\\ -1&0\end{pmatrix}$ & $\begin{pmatrix}1&0\\ 0&z^2\end{pmatrix}$ & $\begin{pmatrix}z^2&0\\ 0&1\end{pmatrix}$ & $\begin{pmatrix}1&0\\ 0&z^2\end{pmatrix}$ &  $(-z^{3}, z^{5}, z^{5}, z^{5}, 1)$ &  40 & 72 \rule[0pt]{0pt}{4ex} \\
				\hline 
				\multicolumn{9}{|c|}{$G_{32}$ type B middle convolutions} \rule[0pt]{0pt}{3ex}\\ \hline
				$\frac{1}{6}$ & $z^{2}$ & $\begin{pmatrix}1&-w\\ 0&-w\end{pmatrix}$ & $\begin{pmatrix}w^2&w-1\\ -1&-w^2\end{pmatrix}$ & $\begin{pmatrix}-w&w\\ 0&1\end{pmatrix}$ & $\begin{pmatrix}1&0\\ -w&-1\end{pmatrix}$ &  $(z, z^{3}, z, z^{3}, z^2-1)$ &  60 & 72 \rule[0pt]{0pt}{4ex} \\ 
				\hline 
				$\frac{5}{6}$ & $-z^{2} +1$ & $\begin{pmatrix}1&-w\\ 0&-1\end{pmatrix}$ & $\begin{pmatrix}1&0\\ -1&-w^2\end{pmatrix}$ & $\begin{pmatrix}-1&w\\ 0&1\end{pmatrix}$ & $\begin{pmatrix}1&0\\ -1&-w\end{pmatrix}$ & $(z^3, z^5, z^3, z, 1)$&  & 0 \rule[0pt]{0pt}{4ex}\\ \hline
				\multicolumn{9}{|c|}{$G_{32}$ type C middle convolutions} \rule[0pt]{0pt}{3ex}\\ \hline
				$\frac{1}{12}$ & $z$ & $\begin{pmatrix}-z^3&0\\ z^{11}&1\end{pmatrix}$ & $\begin{pmatrix}1&z^4\\ 0&-z^3\end{pmatrix}$ & $\begin{pmatrix}-z^3&-z^4\\ 0&1\end{pmatrix}$ & $\begin{pmatrix}1&0\\ -z^{11}&-z^3\end{pmatrix}$ &  $(q^3, q^{3}, q^3, q^{3}, -1)$ &  20 & 96 \rule[0pt]{0pt}{4ex} \\ 
				\hline 
				$\frac{7}{12}$ & $-z$ & $\begin{pmatrix}z^3&0\\ -z^{11}&1\end{pmatrix}$ & $\begin{pmatrix}1&z^4\\ 0&z^3\end{pmatrix}$ & $\begin{pmatrix}z^3&-z^4\\ 0&1\end{pmatrix}$ & $\begin{pmatrix}1&0\\ z^{11}&z^3\end{pmatrix}$ & $(q^9, q^9, q^9, q^9, -1)$& 20 & 96 \rule[0pt]{0pt}{4ex}\\ \hline
			\end{tabular} \caption{\label{G_{32} Nice 4-Tuples}Here $z = \zeta_{12}$, $w = \zeta_3$, and $q = \zeta_{24}$.}
		\end{sidewaystable}
		
		\section{Middle convolution tables for nice 5-tuples} \label{appendix C}
		The notation here is the same as above.
      
	 	 \begin{table}[H] 
 		 \centering               
			\rotatebox{90}{%
			 \begin{minipage}{0.7\linewidth}
			   \centering
			\hspace*{-10cm}
                        \begin{tabular}{|Sc|Sc|Sc|Sc|Sc|Sc|Sc|Sc|Sc|Sc|Sc|}
						\hline 
						\multicolumn{11}{|c|}{$G_{32}$ 5-tuple middle convolutions} \rule[0pt]{0pt}{4ex}\\ 
						\hline
						  Type   & $\xi$ & $\lambda$ & $\tilde{M}_1$ & $\tilde{M}_2$ & $\tilde{M}_3$ & $\tilde{M}_4$ & $\tilde{M}_5$ & Character & O.~size & S.~size \rule[0pt]{0pt}{4ex} \\
						\hline
						A &  $\frac{5}{6}$ & $-w$ & $\begin{pmatrix}1&0\\ 0&w+1\end{pmatrix}$ & $\begin{pmatrix}1&0\\ 0&w+1\end{pmatrix}$ & $\begin{pmatrix}0&w+1\\ -w&0\end{pmatrix}$ & $\begin{pmatrix}w+1&0\\ 0&1\end{pmatrix}$ & $\begin{pmatrix}1&0\\ 0&w+1\end{pmatrix}$ & $(z^{11}, z^{11}, -z^3, z^{11}, z^{11}, z^7)$  & -  & 72 \rule[0pt]{0pt}{4ex} \\ 
						\hline 
						B &  $\frac{5}{6}$ & $-w$ & $\begin{pmatrix}0&1\\ 1&0\end{pmatrix}$ & $\begin{pmatrix}1&0\\ 0&w+1\end{pmatrix}$ & $\begin{pmatrix}0&-w-1\\ w&0\end{pmatrix}$ & $\begin{pmatrix}w+1&0\\ 0&1\end{pmatrix}$ & $\begin{pmatrix}1&0\\ 0&w+1\end{pmatrix}$&  $(z^3, z^{11}, z^3, z^{11}, z^{11}, -z^3)$  & -  &  72 \rule[0pt]{0pt}{4ex} \\ 
						\hline 
			\end{tabular}
                       \caption{Here $w = \zeta_{3}$ and $z = \zeta_{12}$.} \label{G_{32} Nice 5-Tuples}
			\end{minipage}}
\end{table}
                                   \clearpage


\end{document}